\documentclass{article}[11pt]

\usepackage{amsmath}   
\usepackage{amssymb}
\usepackage{amsthm}
\usepackage{accents}
\usepackage{xcolor}
\usepackage{tikz-cd}

\newtheorem{Th}{Theorem}[section] 
\newtheorem{Prop}{Proposition}[section]   
\newtheorem{Lem}{Lemma}[section]   
\newtheorem{Coro}{Corollary}[section]   
\newtheorem{Def}{Definition}  
\newtheorem{Rem}{Remark}[section]

\newcommand{\R}{\mathbb{R}}
\newcommand{\Z}{\mathbb{Z}}
\newcommand{\C}{\mathbb{C}}

\newcommand{\N}{\mathbb{N}}
\newcommand{\Sz}{\mathcal{S}}
\newcommand{\LL}{\mathcal{L}}
\newcommand{\D}{\mathcal{D}}
\newcommand{\E}{\mathcal{E}}

\newcommand{\U}{\mathcal{U}}

\newcommand{\G}{\mathcal{G}}
\newcommand{\aW}{{\accentset{\,\circ}W}}
\newcommand{\tG}{\widetilde{G}}
\newcommand{\tE}{\widetilde{\mathcal{E}}}
\newcommand{\x}{\langle x\rangle}
\newcommand{\y}{\langle y\rangle}
\newcommand{\p}{{\rm p}}

\newcommand{\id}{\text{\rm id}}
\newcommand{\Id}{\text{\rm Id}}

\newcommand{\tr}{\text{\rm tr}\,}
\newcommand{\sign}{\text{\rm sign}\,}

\newcommand{\real}{\text{\rm Re}\,}
\newcommand{\Ker}{\text{\rm Ker}\,}
\newcommand{\Range}{\text{\rm Range}\,}
\newcommand{\Div }{\mathop{\rm div}\nolimits}

\newcommand{\dt}[1]{\accentset{\mbox{\bfseries .}}{#1}}
\newcommand{\bfdot}{{\mbox{\bfseries .}}}
\newcommand{\dd}{{\rm d}}

\newcommand{\2}{\rangle}
\newcommand{\3}{\vert}

\begin{document}

\title{The viscosity limit of fluid flows with growth/decay conditions at infinity}   
 
\author{R. McOwen and P. Topalov} 

\maketitle

\begin{abstract}  
We prove that the Navier-Stokes equation is well-posed in function spaces on $\R^d$, $d\ge 2$, that contain
vector fields of order $O(|x|^\kappa)$ as $|x|\to\infty$ with $\kappa<1/2$. 
The corresponding solutions depend continuously on the viscosity parameter $\nu\ge 0$ and converge to 
the solutions of the Euler equation as $\nu\to 0+$. Our proof is based on the properties of 
the conjugated heat flow on weighted Sobolev spaces and on a new variant of the Lie-Trotter product formula
for nonlinear semigroups.
\end{abstract}   

%\tableofcontents 

\section{Introduction}\label{sec:introduction}
The Navier-Stokes equation describes the time evolution of the velocity field $u(x,t)$, $x\in\R^d$,
$t\ge 0$, $d\ge 2$, of a viscous fluid,
\begin{equation}
\left\{
\begin{array}{l}\label{eq:NS}
u_t+u\cdot\nabla u=\nu \Delta u-\nabla\p,\quad\Div u=0,\\
u|_{t=0}=u_0,
\end{array}
\right.
\end{equation}
where $\p(x,t)$ is the pressure and the constant $\nu>0$ is the viscosity parameter. 
Here, $u\cdot\nabla :=\sum_{j=1}^d u_j \frac{\partial}{\partial x_j}$ is the derivative in the direction of $u$, 
$\Div u$ is the divergence of $u$, and $\Delta:=\sum_{j=1}^d\frac{\partial^2}{\partial x_j^2}$ is the Laplace operator. 

Our space of choice is the weighted Sobolev space $W^{m,p}_\delta$ with weight $\delta\in\R$ and
regularity exponent $m\ge 0$. Depending on the value of the weight $\delta\in\R$, the space controls
the spatial growth/decay of its elements as $|x|\to\infty$.
More specifically, for $1<p<\infty$, a regularity exponent $m\ge 0$, and a weight $\delta\in\R$, one defines the function space
\begin{equation}\label{eq:W-space}
W^{m,p}_\delta:=\big\{f\in\Sz'(\R^d,\R)\,\big|\,\x^{\delta+|\alpha|}\partial^\alpha f\in L^p\big\}
\end{equation}
equipped with the norm
\begin{equation}\label{eq:W-norm}
\|f\|_{W^{m,p}_\delta}:=
\Big(\sum_{|\alpha|\le m}\big\|\x^{\delta+|\alpha|}\partial^\alpha f\big\|_{L^p}^p\Big)^{1/p},
\quad\x:=\sqrt{1+|x|^2},
\end{equation}
where $\Sz'(\R^d,\R)$ denotes the space of real valued tempered distributions on $\R^d$,
$\partial^\alpha:=\partial_1^{\alpha_1}...\partial_d^{\alpha_d}$, where $\alpha\in\Z_{\ge 0}^d$
is a multi-index and $\partial_j$ denotes the weak derivative $\frac{\partial}{\partial x_j}$ for $j=1,...,d$.
By $\Z_{\ge 0}$ we denote the set of non-negative integers.
The space $W^{m,p}_\delta$ is the closure of $C^\infty_c(\R^d)$ with respect to the norm \eqref{eq:W-norm}
where $C^\infty_c(\R^d)$ denotes the space of $C^\infty$ functions on $\R^d$ with compact support.
Note that $\x^\kappa\in W^{m,p}_\delta$ provided $\kappa<-\big(\delta+\frac{d}{p}\big)$.
In addition to $W^{m,p}_\delta$ we will also need its complexification $W^{m,p}_{\delta,\C}:=W^{m,p}_\delta\otimes\C$
consisting of complex valued functions on $\R^d$ whose real and imaginary parts belong to $W^{m,p}_\delta$. 
For simplicity of notation, we will also denote by $W^{m,p}_\delta$ the spaces of tensor fields on $\R^d$
whose components are in $W^{m,p}_\delta$. The norms in all these spaces will be denoted by $\|\cdot\|_{W^{m,p}_\delta}$
(or by $\|\cdot\|_{W^{m,p}_{\delta,\C}}$ for the complex case).
It follows directly from the definition \eqref{eq:W-space} that for any multi-index $\alpha\in\Z_{\ge 0}^d$
with $|\alpha|\le m$ the map $\partial^\alpha : W^{m,p}_\delta\to W^{m-|\alpha|,p}_{\delta+|\alpha|}$
is bounded. For the convenience of the reader, we summarize the basic properties of these spaces
in Appendix \ref{sec:W-spaces}. Note that by Lemma \ref{lem:W-decay} in Appendix \ref{sec:W-spaces}, 
for $m>k+\frac{d}{p}$ and $k\in\Z_{\ge 0}$ we have the bounded embedding $W^{m,p}_\delta\subseteq C^k$
where $C^k$ denotes the space of $k$-times continuously differentiable functions on $\R^d$.
Moreover, for $m>k+\frac{d}{p}$ and $k\in\Z_{\ge 0}$ there exists a constant $C>0$ such that
\begin{equation}\label{eq:decay}
\x^{\delta+\frac{d}{p}+|\alpha|}\big|(\partial^\alpha f)(x)\big|\le C\,\|f\|_{W^{m,p}_\delta}
\end{equation}
for any $|\alpha|\le k$ and $x\in\R^d$. By Lemma \ref{lem:W-multiplication} in Appendix \ref{sec:W-spaces} the space
$W^{m,p}_\delta$ is a Banach algebra for $m>\frac{d}{p}$ and $\delta+\frac{d}{p}\ge 0$. 
For $m\ge 1$, we let $\aW^{m,p}_\delta$ denote the divergence free vector fields in
$W^{m,p}_\delta$, i.e., $\aW^{m,p}_\delta:=\big\{u\in W^{m,p}_\delta\,\big|\,\Div u=0\big\}$.
For a Banach space $X$ we denote by $C^k_b\big([0,\tau),X\big)$, $k\in\Z_{\ge 0}$,
the Banach space of bounded curves with bounded derivatives of order $\le k$.

When $\nu=0$ the equation \eqref{eq:NS} becomes the Euler equation for the flow
of a perfect (inviscid and incompressible) fluid in $\R^d$.
It was proven in \cite[Theorem 1.1]{McOwenTopalov5} that for $m>1+\frac{d}{p}$
the Euler equation is well-posed in $\aW^{m,p}_\delta$ for $-1/2<\delta+\frac{d}{p}<d+1$.
The value $d+1$ of the {\em weight parameter} $\delta+\frac{d}{p}$ appears as a threshold since
for $\delta+\frac{d}{p}\ge 0$ the spaces $\aW^{m,p}_\delta$ are no longer preserved by the solutions of
the Euler equation, because they develop spatial asymptotic expansions as $|x|\to\infty$ of leading order 
$O\big(1/|x|^{d+1}\big)$ (see \cite[Theorem 1.2 and Proposition 1.1]{McOwenTopalov5} for the details
as well as \cite{DobrShaf}). In this way, for initial data $u_0\in\aW^{m,p}_\delta$ with $\delta+\frac{d}{p}\ge d+1$
the evolution of the Euler equation happens in the asymptotic spaces introduced in \cite{McOwenTopalov2,McOwenTopalov3}. 
Recall from \cite[Theorem 1.1]{McOwenTopalov5} that the Navier-Stokes equation is well-posed in the asymptotic spaces.
As a consequence, it was recently proven in \cite[Theorem 1.1 and Proposition 1.1]{Top1} that 
the Navier-Stokes equation is well-posed in $\aW^{m,p}_\delta$ for $0\le\delta+\frac{d}{p}<d+1$ and
that its solutions develop spatial asymptotic expansions as $|x|\to\infty$ in the same way as the solutions of the Euler equation
for $\delta+\frac{d}{p}\ge d+1$. These results lead to the question of the existence of (unbounded) solutions
of the Navier-Stokes equation in $\aW^{m,p}_\delta$ for $-1/2<\delta+\frac{d}{p}<0$ and, most importantly, 
to the question of the existence of a viscosity limit of the solutions of the Navier-Stokes equation
for the maximal range $-1/2<\delta+\frac{d}{p}<d+1$ of well-posedness in $\aW^{m,p}_\delta$ as $\nu\to 0+$. 
These questions are answered affirmatively in the present paper.

The following Theorem shows that the Navier-Stokes equation \eqref{eq:NS} admits
classical solutions of order $O(|x|^\kappa)$ as $|x|\to\infty$ with $\kappa<1/2$.
More specifically, we have 

\begin{Th}\label{th:NS}
Assume that $m>2+\frac{d}{p}$, $-1/2<\delta+\frac{d}{p}<d+1$, and let $m_0:=m+8$. 
Then for any $u_\bullet\in\aW^{m_0,p}_\delta$ there exist $\tau>0$ and an open neighborhood 
$\mathcal{O}(u_\bullet)$ in $\aW^{m_0,p}_\delta$ such that for any $u_0\in\mathcal{O}(u_\bullet)$ 
there exists a unique solution 
\[
u\in C_b\big([0,\tau),\aW^{m,p}_\delta\big)\cap C^1_b\big([0,\tau),\aW^{m-2,p}_\delta\big)
\]
of the Navier-Stokes equation \eqref{eq:NS} such that $|\nabla\p(t,x)|=o(1)$ as 
$|x|\to\infty$ for $t\in[0,\tau)$. The solution depends continuously on the initial data
$u_0\in\mathcal{O}(u_\bullet)$, $u(t)\in\aW^{m_0,p}_\delta$ for $t\in[0,\tau)$, 
and the curve $u : [0,\tau)\to\aW^{m_0,p}_\delta$ is uniformly bounded and right continuous.
\end{Th}

The solution in Theorem \ref{th:NS} depends on the value of the viscosity parameter $\nu>0$.
We will indicate this by denoting the solution in Theorem \ref{th:NS} by $u^{(\nu)}$. If $\nu=0$ then 
$u^{(0)}$ will denote the solution of the Euler equation constructed in \cite[Theorem 1.1]{McOwenTopalov4}. 
We have the following Theorem.

\begin{Th}\label{th:NS_nu}
Assume that $m>2+\frac{d}{p}$, $-1/2<\delta+\frac{d}{p}<d+1$, and let $m_0:=m+8$. 
Then for any $u_\bullet\in\aW^{m_0,p}_\delta$ there exist $\tau>0$ and an open neighborhood 
$\mathcal{O}(u_\bullet)$ in $\aW^{m_0,p}_\delta$ independent of the choice of $0\le\nu\le 1$ such that the map
\[
\mathcal{O}(u_\bullet)\times[0,1]\to 
C_b\big([0,\tau),\aW^{m,p}_\delta\big)\cap C^1_b\big([0,\tau),\aW^{m-2,p}_\delta\big),\quad
(u_0,\nu)\mapsto u^{(\nu)},
\]
where $u^{(\nu)}$ is the solution of the Navier-Stokes equation (or of the Euler equation when $\nu=0$), is continuous. 
In particular, for any initial data $u_0\in\mathcal{O}(u_\bullet)$ we have that $u^{(\nu)}\to u^{(0)}$ as $\nu\to 0+$ 
where the convergence happens in the space  
$C_b\big([0,\tau),\aW^{m,p}_\delta\big)\cap C^1_b\big([0,\tau),\aW^{m-2,p}_\delta\big)$.
\end{Th}

The condition on the pressure in Theorem \ref{th:NS} is needed for the uniqueness of the solutions.
This condition can be replaced by similar conditions on the pressure or on the solution, but cannot be avoided.
Note that the pressure in Theorem \ref{th:NS} is determined uniquely (up to an additive constant)
and $\p\in C_b\big([0,\tau),W^{m+1,p}_{2\delta+\frac{d}{p}}\big)$. In particular,
$\p(t,x)=O(|x|^{2\kappa})$ as $|x|\to\infty$ where $-(d+1)<\kappa<1/2$  with constants uniform 
on $t\in[0,\tau)$.

Theorem \ref{th:NS} and Theorem \ref{th:NS_nu} are proved in Section \ref{sec:unbounded_solutions}.
The proof is based on the mapping properties of the conjugated heat flow on weighted Sobolev spaces 
(Theorem \ref{th:conjugate_Laplace_operator}) proved in Section \ref{sec:conjugated_Laplace_operator}
as well as on a new variant of the Lie-Trotter product formula for nonlinear semigroups (Theorem \ref{th:product_formula}), 
which is formulated in Section \ref{sec:Lie-Trotter} and proved in Section \ref{sec:Lie-Trotter_proofs}. 
Since the known variants of the Lie-Trotter product formula (cf., e.g., \cite{EM}) are not applicable
in our situation as the configuration space is not compact and the heat semigroup on weighted spaces is
not a quasi-contraction, we had to weaken the conditions needed for the Lie-Trotter formula to hold. 
We believe that Theorem \ref{th:product_formula} can be applied for proving analogous results in other situations.
The drawback of this approach is that the obtained regularity of the solutions is most likely not optimal (cf. \cite{Top1}).
Note also that the method in \cite{McOwenTopalov5} for proving the well-posedness of the Navier-Stokes equation
in the asymptotic spaces, based on the Duhamel's formula and the properties of the heat semigroup,  is not
applicable in the case of weighted Sobolev spaces $\aW^{m,p}_\delta$ with $-1/2<\delta+\frac{d}{p}<0$. 
The reason is that the smoothing of the heat semigroup on weighted Sobolev spaces comes with 
a ``loss of weight'' as indicated by \cite[Corollary 2.2]{McOwenTopalov5}.
A standard approach for proving the existence of the zero viscosity limit is based on energy estimates that are uniform in 
the viscosity parameter $\nu$ (see, e.g., \cite{Kato2,Kato3}, \cite[Chapter 3]{MB}). 
Since our weighted spaces contain unbounded functions, this approach is not readily applicable. 
The methods applied in \cite{Wol,McGrath,Swann,Leray}) also require the initial velocity to tend to zero at infinity.

\medskip

\noindent{\em The conjugated Laplace operator on weighted spaces.}
Here we formulate Theorem \ref{th:conjugate_Laplace_operator} needed for the application 
of the Lie-Trotter product formula. Assume that $m>1+\frac{d}{p}$ and $\delta+\frac{d}{p}>-1$.
Then, we can define (cf. \cite[Section 2, Appendix A]{McOwenTopalov4}) the group of diffeomorphisms on $\R^d$,
\begin{equation}\label{eq:group_D}
\D^{m,p}_\delta:=
\big\{\varphi : \R^d\to\R^d\,\big|\,\varphi=\id+f, f : \R^d\to\R^d, f\in W^{m,p}_\delta,\,\text{and}\,
\det(\dd\varphi)>0\big\},
\end{equation}
where $\id$ is the identity on $\R^d$ and $\dd\varphi$ denotes the differential of $\varphi : \R^d\to\R^d$.
The assumption $m>1+\frac{d}{p}$, $\delta+\frac{d}{p}>-1$, $\det(\dd\varphi)>0$, and
Lemma \ref{lem:W-decay} in Appendix \ref{sec:W-spaces} imply that the elements of $\D^{m,p}_\delta$ are 
orientation preserving $C^1$-diffeomorphisms on $\R^d$.
The same Lemma implies that $\det(\dd\varphi)>0$ is an open condition on the map $f$ in $W^{m,p}_\delta$.
In this way, by identifying $\varphi=\id+f\in\D^{m,p}_\delta$ with the map $f : \R^d\to\R^d$, 
we identify the elements of $\D^{m,p}_\delta$ with an open set in $W^{m,p}_\delta$
(see \cite[Section 2, Appendix A]{McOwenTopalov4}).
By \cite[Theorem 2.1]{McOwenTopalov4}, $\D^{m,p}_\delta$ is a topological group such that 
for any $0\le k\le m$ and $\gamma\in\R$ the right translation
\begin{equation}\label{eq:right_translation}
\D^{m,p}_\delta\times W^{k,p}_\gamma\to W^{k,p}_\gamma,\quad(\varphi,u)\mapsto R_\varphi u:=u\circ\varphi,
\end{equation}
is continuous (see Lemma \ref{lem:continuity_composition_general} in Appendix \ref{sec:W-spaces}). 
Let us now fix $m>1+\frac{d}{p}$, $\delta+\frac{d}{p}>-1$, and $\gamma\in\R$.
For a given $\varphi\in\D^{m,p}_\delta$ consider the {\em conjugated Laplace operator}
\begin{equation}\label{eq:conjugated_Laplace_operator}
\Delta_\varphi:=R_\varphi\circ\Delta\circ R_{\varphi^{-1}} : 
R_\varphi\big(\widetilde{W}^{m+2,p}_\gamma\big)\to W^{m,p}_\gamma
\end{equation}
where $\Delta=\sum_{j=1}^d\partial_j^2$ is the (distributional) Laplace operator on $\R^d$,
\begin{equation}\label{eq:W-tilde}
\widetilde{W}^{m+2,p}_\gamma:=\big\{f\in\Sz'\,\big|\,\partial^\alpha f\in W^{m,p}_\gamma, |\alpha|\le 2\big\},
\end{equation}
and $R_\varphi\big(\widetilde{W}^{m+2,p}_\gamma\big)$ is the image of
$\widetilde{W}^{m+2,p}_\gamma\subseteq W^{m,p}_\gamma$
under the map $R_\varphi : W^{m,p}_\gamma\to W^{m,p}_\gamma$.
It follows from \eqref{eq:right_translation} and \eqref{eq:W-tilde} that the map \eqref{eq:conjugated_Laplace_operator} is
well-defined. In what follows we consider $\Delta_\varphi$ as an unbounded operator on $W^{m,p}_\gamma$ with domain
$D_{\Delta_\varphi}:= R_\varphi\big(\widetilde{W}^{m+2,p}_\gamma\big)$ in $W^{m,p}_\gamma$.
Since the map $R_\varphi : W^{m,p}_\gamma\to W^{m,p}_\gamma$ is an isomorphism of Banach spaces we see that
$D_{\Delta_\varphi}$ is dense in $W^{m,p}_\gamma$.
For $X$ and $Y$ Banach spaces we denote the Banach space of bounded linear maps
$X\to Y$ by $\LL(X,Y)$ and let $\LL(X):=\LL(X,X)$ when $X=Y$. We have the following Theorem.

\begin{Th}\label{th:conjugate_Laplace_operator}
Assume $m>1+\frac{d}{p}$, $\delta+\frac{d}{p}>-1$, and $\gamma\in\R$.
\begin{itemize}
\item[(i)] For any given $\varphi\in\D^{m,p}_\delta$ the conjugated Laplace operator $\Delta_\varphi$
considered as an unbounded operator on $W^{m,p}_\gamma$ with domain 
$D_{\Delta_\varphi}= R_\varphi\big(\widetilde{W}^{m+2,p}_\gamma\big)$ is closed and
its spectrum is contained in $(-\infty,0]$.
\item[(ii)] For any given $\varphi\in\D^{m,p}_\delta$ the operator $\Delta_\varphi$ on $W^{m,p}_\gamma$ 
with domain $D_{\Delta_\varphi}$ is a generator of an analytic semigroup $\{e^{t\Delta_\varphi}\}_{t\ge 0}$
of angle $\pi/2$. For $t>0$ we have that $e^{t\Delta_\varphi}=S_\varphi(t)$ where $S_\varphi(t)$ is
given by \eqref{eq:S(t)-conjugated}.
\item[(iii)] For $\gamma+\frac{d}{p}>-1$ the map
\begin{equation}\label{eq:exp_map}
(0,\infty)\times\D^{m,p}_\delta\to\LL(W^{m,p}_\gamma),\quad(t,\varphi)\mapsto S_\varphi(t),
\end{equation}
is analytic. Moreover, for any given $\varphi_0\in\D^{m,p}_\delta$ and for any $\ell_0\ge 1$ there exist
an open neighborhood of $\varphi_0$ in $\D^{m,p}_\delta$ and a constant $C>0$ such that for any $0\le\ell\le\ell_0$
the $\ell$-th differential of the map \eqref{eq:exp_map} in the direction of the first variable satisfies
the global in time estimate
\begin{equation}\label{eq:exp_map_uniformly_bounded}
\big\|\dd^\ell_\varphi\big(S_{(\cdot)}(t)\big)\big\|_{\LL((W^{m,p}_\gamma)^\ell,\LL(W^{m,p}_\gamma))}\le 
C\,(1+t)^{|\gamma|/2},
\end{equation}
for any $\varphi$ in the open neighborhood of $\varphi_0$ and for any $t\in[0,\infty)$.
\end{itemize}
\end{Th}

\noindent For any given $t>0$ and $\ell\ge 0$ the $\ell$-th differential 
$\dd_\varphi^\ell\big(S_{(\cdot)}(t)\big)$ of \eqref{eq:exp_map} in the direction of the first variable
belongs to $\LL\big(\big(W^{m,p}_\delta\big)^\ell,\LL\big(W^{m,p}_\delta\big)\big)$
where $\big(W^{m,p}_\delta\big)^\ell$ is the direct product of $\ell$ copies of $W^{m,p}_\delta$.
Recall that an analytic semigroup $\{e^{t A}\}_{t\ge 0}$ on a Banach space $X$ with a generator $A$
is called {\em analytic of angle} $\pi/2$ if for any $0<\vartheta<\pi/2$ the semigroup extends to an analytic map
$\mathbb{S}_\vartheta\to\LL(X)$, $t\mapsto e^{tA}$, where $\mathbb{S}_\vartheta$ denotes
the centered at zero cone
\begin{equation}\label{eq:the_cone}
\mathbb{S}_\vartheta:=\big\{z\in\C\setminus\{0\}\,\big|\,|\arg z|<\vartheta\big\}
\end{equation}
in the complex plane $\C$ (see, e.g., \cite[\S\,2.5]{Pazy} for the definition of an analytic semigroup).
Note that the uniformity in $t\in[0,\infty)$ and the local uniformity in $\varphi\in\D^{m,p}_\delta$ of
the estimate \eqref{eq:exp_map_uniformly_bounded} are crucial for the applications to the Navier-Stokes equation.
By $X^*$ we denote the space of bounded linear functionals on $X$.
In the case when $X$ and $Y$ are linear spaces over $\C$, we will assume that $\LL(X,Y)$ and $X^*$
consist of complex linear maps. 

\medskip

\noindent{\em Related work.} There are many important works on the solutions of the Navier-Stokes
equation in various function spaces.
Here we concentrate only on works related to the existence an the properties of non-decaying solutions
of the Euler and the Navier-Stokes equation on unbounded domains; we refer to the monographs 
\cite{MB,Chemin} and \cite{Kato3} for more complete list of references. 
We first mention the classical work of Wolibner \cite{Wol}
and the work of Serfati \cite{Serfati1}, where solutions of the Euler equation were constructed in 
H\"older spaces on unbounded domains.
Non-decaying bounded solutions of the Navier-Stokes equation were also considered in \cite{GMZ} and \cite{GIM}.
Unbounded solutions for the 2d Euler equation were constructed by Cozzi and Kelliher \cite{CK} and by Benedetto,
Marchioro, and Pulvirenti \cite{BMP1}. The general case of $\R^d$, $d\ge 2$, is considered in
\cite[Theorem 1.1]{McOwenTopalov4} by using the framework of the weighted Sobolev spaces. 
We also mention the recent work of Comb and Koch \cite{CK} where unbounded solutions of order
$O\big(|x|^\kappa\big)$, $\kappa<1/2$, as $|x|\to\infty$ are constructed for the 2d Euler equation. 
The existence of weak solutions of sublinear growth for the 2d Navier-Stokes equation was obtained in \cite{Ben}
under particular conditions on the velocity and vorticity. 
These works do not discuss the zero viscosity limit of the solutions of th Navier-Stokes equation.
To our best knowledge, Theorem \ref{th:NS} and Theorem \ref{th:NS_nu} are the only result that
allows classical solutions of the Navier-Stokes equation to grow as $|x|\to\infty$ in dimensions greater
than two and that establish the existence of the zero viscosity limit for such solutions. We expect that
similar results hold also for the asymptotic spaces.
Finally, note that in the 1d case unbounded solutions of the KdV and the modified KdV equations were 
constructed in \cite{BS,KPST}.

\section{The conjugated heat flow on weighted spaces}\label{sec:conjugated_Laplace_operator}
The purpose of this Section is to prove Theorem \ref{th:conjugate_Laplace_operator} formulated in the Introduction.
Recall from \cite[Corollary 2.2]{McOwenTopalov5} that for any $k\ge 0$ and $\gamma\in\R$
the Laplace operator $\Delta$ on $\R^d$ generates an analytic semigroup $\{e^{t\Delta}\}_{t\ge 0}$ of bounded linear
operators on the weighted Sobolev space $W^{k,p}_\gamma$. 
More specifically, the generator of $\{e^{t\Delta}\}_{t\ge 0}$ is the Laplace operator $\Delta$ considered as
an unbounded operator on $W^{k,p}_\gamma$ with domain $\widetilde{W}^{k+2,p}_\gamma$ (cf. \eqref{eq:W-tilde}).
Take $u_0\in W^{k,p}_\gamma$ and set $u(t):=e^{t\Delta}u_0$, $t\ge 0$. 
Then $u\in C\big([0,\infty), W^{k,p}_\gamma\big)\cap C^1\big((0,\infty),W^{k,p}_\gamma\big)$,
$u(t)\in\widetilde{W}^{k+2,p}_\gamma$ for $t>0$, and $u$ is a solution of the heat equation
$u_t=\Delta u$, $u|_{t=0}=u_0$.
It follows from \cite[Corollary B.1]{McOwenTopalov5} that $e^{t\Delta}=S(t)$, $t>0$, where
\begin{equation}\label{eq:S(t)}
\big(S(t)u_0\big)(x):=\frac{1}{(4\pi t)^{d/2}}\int_{\R^d}e^{-\frac{|x-y|^2}{4t}}u_0(y)\,dy
\end{equation}
for $u_0\in W^{k,p}_\gamma$.
Since the right translation \eqref{eq:right_translation} is continuous, we see that for any $t>0$ 
the {\em conjugated heat flow map}
\begin{equation}\label{eq:conjugated_heat_flow}
\D^{m,p}_\delta\times W^{k,p}_\gamma\to W^{k,p}_\gamma,\quad
(\varphi,u)\mapsto S_\varphi(t) u:=R_\varphi\circ S(t)\circ R_{\varphi^{-1}}u,\quad 0\le k\le m,
\end{equation}
is continuous. 
In order to apply the version of the Lie-Trotter product formula (Theorem \ref{th:product_formula}), we 
need to know that \eqref{eq:conjugated_heat_flow} (with $k=m$ and $\gamma=\delta$) is $C^2$
with bounds on its derivatives that are uniform in $t>0$ and locally uniform in $\varphi\in\D^{m,p}_\delta$.
In addition to this, we will also show that \eqref{eq:conjugated_heat_flow} is analytic via a multi-step process:
First, in order to use complex-analytic methods such as Cauchy's estimate, we (locally) extend 
the map \eqref{eq:conjugated_heat_flow} to a complex neighborhood of its domain of definition. 
Then we start our analysis by considering the case when $k=0$ and show that the complex extension of 
\eqref{eq:conjugated_heat_flow} is analytic and locally uniformly bounded with a constant that is uniform in $t>0$ 
(see Proposition \ref{prop:S(t)-conjugated_preliminary}). The next step is to bootstrap from $k$ to $k+2$,
so we get analyticity for all even values of $0\le k\le m$. The existence of kernel for 
$\Delta : W^{m+2,p}_\gamma\to W^{m,p}_{\gamma+2}$ when $-1/2<\gamma+\frac{d}{p}<0$ and
cokernel for non-integer $\gamma+\frac{d}{p}>d-2$ requires splitting arguments in the analysis. 
Then we use interpolation to obtain analyticity for all values of $0\le k\le m$, including the integer values of
$\gamma+\frac{d}{p}$ that were excluded from the previous analysis (see Proposition \ref{prop:S(t)-conjugated}).
Theorem \ref{th:conjugate_Laplace_operator} is proven at the end of the Section.

\subsection{Analyticity for $k=0$}
For a given $\varphi\in\D^{m,p}_\delta$, $\delta+\frac{d}{p}>-1$, and $u_0\in W^{k,p}_\gamma$, 
$0\le k\le m$, $\gamma\in\R$, we obtain from \eqref{eq:S(t)} that for $t>0$,
\begin{equation}\label{eq:S(t)-conjugated}
\big(S_\varphi(t)u_0\big)(x)=\big(R_\varphi\circ S(t)\circ R_{\varphi^{-1}}u_0\big)(x)=
\frac{1}{(4\pi t)^{d/2}}\int_{\R^d}e^{-\frac{(\varphi(x)-\varphi(y))^2}{4 t}}u_0(y)\det(\dd_y\varphi)\,dy
\end{equation}
where we changed the variables in the integral and set $z^2:=\sum_{j=1}^dz_j^2$ for $z\in\C^d$.
The integral converges due to the exponentially decaying factor in the integrand (cf. Lemma \ref{lem:t-complex} below).
We will use formula \eqref{eq:S(t)-conjugated} to extend the map \eqref{eq:conjugated_heat_flow} to a complex
neighborhood of its domain of definition.
To this end, take $\varphi_0\in\D^{m,p}_\delta$ such that $\varphi_0=\id+f_0$ for some $f_0\in W^{m,p}_\delta$.
We then choose an open neighborhood $U(\varphi_0)$ of $f_0$ in $W^{m,p}_\delta$ such that
$\id+f\in\D^{m,p}_\delta$ for $f\in U(\varphi_0)$, an open neighborhood $V(0)$ of zero in 
$W^{m,p}_\delta$, and consider the complex neighborhood
\begin{equation}\label{eq:U_C}
\U_\C(\varphi_0):=\big\{f+i f_I\,\big|\,f\in U(\varphi_0),f_I\in V(0)\}
\end{equation}
of $f_0$ in $W^{m,p}_{\delta,\C}$.
(The notation $U(\varphi_0)$ for the open neighborhood of $f_0$ in $W^{m,p}_\delta$ reflects the fact that
we regard $f\in U(\varphi_0)$ as a coordinate on the Banach manifold $\D^{m,p}_\delta$.)
It follows from Theorem \ref{th:composition} and Lemma \ref{lem:continuity_composition_general} in 
Appendix \ref{sec:W-spaces} that we can choose the open neighborhood $U(\varphi_0)$ of $f_0$ in 
$W^{m,p}_\delta$, so that there exists a constant $C>0$ such that for any $0\le k\le m$ and $\gamma\in\R$,
\begin{equation}\label{eq:R-local_boundedness}
\|u\circ\varphi\|_{W^{k,p}_\gamma}\le C\|u\|_{W^{k,p}_\gamma}\quad\text{and}\quad\,
\|u\circ\varphi^{-1}\|_{W^{k,p}_\gamma}\le C\|u\|_{W^{k,p}_\gamma}
\end{equation}
for any $\varphi=\id+f$ with $f\in U(\varphi_0)$ and for any $u\in W^{k,p}_\gamma$.
Now consider the map 
\[
\varphi_\C(x):=x+f_\C(x)\quad\text{\rm where}\quad f_\C:=f+i f_I\in\U_\C(\varphi_0).
\]
Note that this map can no longer be interpreted as a diffeomorphism on $\R^d$  for $f_I\ne 0$ but
the integral on the right side of \eqref{eq:S(t)-conjugated} is well-defined for $\varphi$ replaced by $\varphi_\C$.
For $f_\C\in\U_\C(\varphi_0)$, $u_0\in W^{k,p}_{\gamma,\C}$, and $\real(t)>0$, we set
\begin{equation}\label{eq:S(t)-congugatedC}
\big(S_{\varphi_\C}(t)u_0\big)(x):=
\frac{1}{(4\pi t)^{d/2}}\int_{\R^d}e^{-\frac{(\varphi_\C(x)-\varphi_\C(y))^2}{4 t}}
u_0(y)\det(\dd_y\varphi_\C)\,dy
\end{equation}
where $\det(\dd_y\varphi_\C)=\det\big(\Id+\dd_y(f+i f_I)\big)$ and $\Id$ denotes the identity $d\times d$-matrix.
By the definition of $\U_\C(\varphi_0)$, we have that $\varphi_\C(x)=\varphi(x)+i f_I(x)$ 
where $\varphi:=\id+f\in\D^{m,p}_\delta$ and $f_I\in V(0)\subseteq W^{m,p}_\delta$.
This implies that
\begin{align*}
&\big(R_{\varphi^{-1}}\circ S_{\varphi_\C}(t)\circ R_{\varphi}u_0\big)(x)=\\
&=\frac{1}{(4\pi t)^{d/2}}\int_{\R^d}
\exp\left(-\frac{\Big(\big(x+i f_I\circ\varphi^{-1}(x)\big)-\big(\varphi(y')+i f_I(y')\big)\Big)^2}{4 t}\right)
u_0\circ\varphi(y')\det(\dd_{y'}\varphi_\C)\,dy'\\
&=\frac{1}{(4\pi t)^{d/2}}\int_{\R^d}
\exp\left(-\frac{\Big(\big(x+i\tilde{f}_I(x)\big)-\big(y+i\tilde{f}_I(y)\big)\Big)^2}{4 t}\right)
u_0(y)\det\big(\Id+i \dd_y\tilde{f}_I\big)\,dy
\end{align*}
where we passed to the new variable $y=\varphi(y')$ in the integral, used the relation
$\varphi_\C\circ\varphi^{-1}(y)=y+i f_I\circ\varphi^{-1}(y)$ to conclude that
$(\dd_{\varphi^{-1}(y)}\varphi_\C)(\dd_y\varphi^{-1})=\Id+i\,\dd_y(f_I\circ\varphi^{-1})$, 
and set $\tilde{f}_I:=f_I\circ\varphi^{-1}$.
Hence, for any $f_\C\in\U_\C(\varphi_0)$, $u_0\in W^{k,p}_{\gamma,\C}$, and $\real(t)>0$, we have that
\begin{align}\label{eq:S(t)-conjugated*}
\big(R_{\varphi^{-1}}\circ S_{\varphi_\C}(t)\circ R_{\varphi}u_0\big)(x)=
\int_{\R^d}E(x,y,t)\,\tilde{u}_0(x-y)\,dy
\end{align}
where 
\begin{equation}\label{eq:u_0,f_I-tilde}
\tilde{u}_0(x):=u_0(x)\det\big(\Id+i(\dd_{x}\tilde{f}_I)\big)\quad\text{with}\quad
\tilde{f}_I= f_I\circ\varphi^{-1}
\end{equation}
and
\begin{equation*}
E(x,y,t):=
\frac{1}{(4\pi t)^{d/2}}\exp\left(-\frac{\Big(y+i\big(\tilde{f}_I(x)-\tilde{f}_I(x-y)\big)\Big)^2}{4 t}\right).
\end{equation*}
Our proof of the analyticity of the conjugated heat flow is (partially) based on
the analysis of the integral transform \eqref{eq:S(t)-conjugated*}.
We first prove the following Lemma.

\begin{Lem}\label{lem:t-complex}
Assume that $m>1+\frac{d}{p}$ and $\delta+\frac{d}{p}>-1$.
Then, for any $\varphi_0=\id+f_0\in\D^{m,p}_\delta$ and $0<\vartheta<\pi/2$ there exists an open 
neighborhood $\U_\C(\varphi_0)$ of $f_0$ in $W^{m,p}_{\delta,\C}$, so that for any 
$f_\C= f+i f_I\in \U_\C(\varphi_0)$ and for any $t\in\mathbb{S}_\vartheta$ (cf. \eqref{eq:the_cone}),
\begin{equation}\label{eq:t-complex}
\left|\exp\left(-\frac{\Big(y+i\big(\tilde{f}_I(x)-\tilde{f}_I(x-y)\big)\Big)^2}{4 t}\right)\right|\le
e^{-\frac{|y|^2}{8|t|}\cos\vartheta}\quad\forall x,y\in\R^d,
\end{equation}
where $\tilde{f}_I= f_I\circ\varphi^{-1}$ and $\varphi=\id+f$.
\end{Lem}

\begin{proof}[Proof of Lemma \ref{lem:t-complex}]
Since $z^2=\sum_{j=1}^dz_j^2$ for $z\in\C^d$ we conclude that
\begin{align}
&\real\left(\frac{\Big(y+i\big(\tilde{f}_I(x)-\tilde{f}_I(x-y)\big)\Big)^2}{4 |t|e^{i\vartheta}}\right)\ge\nonumber\\
&\ge\frac{\Big(|y|^2-\big|\tilde{f}_I(x)-\tilde{f}_I(x-y)\big|^2\Big)\cos\vartheta
-2\,|y|\big|\tilde{f}_I(x)-\tilde{f}_I(x-y)\big|}{4 |t|}\label{eq:real_part}
\end{align}
where we write $t=|t|e^{i\vartheta}$ and use the Cauchy-Schwarz inequality. 
Since $W^{m,p}_\delta\subseteq C^1$ for $m>1+\frac{d}{p}$, we obtain from the mean value theorem
that for any $x,y\in\R^d$,
\begin{equation}\label{eq:mean-value1}
\tilde{f}_I(x)-\tilde{f}_I(x-y)=\Big(\int_0^1\dd_{x-t y}\tilde{f}_I\,dt\Big)\,y\,.
\end{equation}
Lemma \ref{lem:W-decay} in Appendix \ref{sec:W-spaces} and the assumption $\delta+\frac{d}{p}>-1$ 
on the weight imply that $W^{m-1,p}_{\delta+1}\subseteq L^\infty$.  
Since $\dd\tilde{f}_I\in W^{m-1,p}_{\delta+1}$ we then conclude from \eqref{eq:mean-value1} and
\eqref{eq:R-local_boundedness} that for any $x,y\in\R^d$,
\begin{equation}\label{eq:mean-value2}
\big|\tilde{f}_I(x)-\tilde{f}_I(x-y)\big|\le C\,\|f_I\|_{W^{m,p}_\delta}|y|
\end{equation}
with a constant $C>0$ independent of the choice of $f_I\in W^{m,p}_\delta$ and 
$\varphi=\id+f$ with $f\in U(\varphi_0)$. 
By shrinking the neighborhood $V(0)$ of zero in $W^{m,p}_\delta$ if necessary
we then obtain from \eqref{eq:mean-value2} that
\begin{equation*}
\big|\tilde{f}_I(x)-\tilde{f}_I(x-y)\big|^2\le\frac{|y|^2}{64}\cos^2\vartheta
\end{equation*}
for any $f_\C= f+i f_I\in\U_\C(\varphi_0)$. Then, we have
\[
\Big(|y|^2-\big|\tilde{f}_I(x)-\tilde{f}_I(x-y)\big|^2\Big)\cos\vartheta
-2\,|y|\big|\tilde{f}_I(x)-\tilde{f}_I(x-y)\big|\ge\frac{|y|^2}{2}\cos\vartheta.
\]
By combining this with \eqref{eq:real_part} we conclude the proof of the Lemma.
\end{proof}

We are now ready to prove

\begin{Prop}\label{prop:S(t)-conjugated_preliminary}
Assume that $m>1+\frac{d}{p}$, $\delta+\frac{d}{p}>-1$, and $\gamma\in\R$.
Then, for any given $\varphi_0=\id+f_0\in\D^{m,p}_\delta$ and $0<\vartheta<\pi/2$
there exists an open neighborhood $\U_\C(\varphi_0)$ of $f_0$ in $W^{m,p}_{\delta,\C}$ so that the map 
\begin{equation}\label{eq:analyticity_preliminary}
\mathbb{S}_{\vartheta}\times\U_\C(\varphi_0)\to\LL\big(L^p_{\gamma,\C}\big),\quad
(t,f_\C)\mapsto S_{\varphi_\C}(t),
\end{equation}
is analytic. Moreover, there exists a constant $C>0$ such that
\begin{equation}\label{eq:bound_preliminary}
\big\|S_{\varphi_\C}(t)\big\|_{\LL(L^p_{\gamma,\C})}\le C\big(1+|t|\big)^{|\gamma|/2}
\end{equation}
for any $f_\C\in\U_\C(\varphi_0)$ and $t\in\mathbb{S}_\vartheta$.
\end{Prop}

\begin{proof}[Proof of Proposition \ref{prop:S(t)-conjugated_preliminary}]
Take $\varphi_0=\id+f_0\in\D^{m,p}_\delta$ and $0<\vartheta<\pi/2$.
By Lemma \ref{lem:t-complex} we can find an open neighborhood $\U_\C(\varphi_0)$ of
$f_0$ in $W^{m,p}_{\delta,\C}$ such that the inequality \eqref{eq:t-complex} holds for any $x,y\in\R^d$ and
for any $f_\C= f+i f_I\in \U_\C(\varphi_0)$ and $t\in\mathbb{S}_\vartheta$.
Then, we obtain from \eqref{eq:S(t)-conjugated*} and \eqref{eq:t-complex} that there exists a constant $C_1>0$ such that 
for any $\varphi_\C=\id+f_\C$ with $f_\C\in\U_\C(\varphi_0)$, $t\in\mathbb{S}_\vartheta$, and
for any $u_0\in L^p_{\gamma,\C}$ and $x\in\R^d$,
\begin{equation}\label{eq:convolution1}
\big|\big(R_{\varphi^{-1}}\circ S_{\varphi_\C}(t)\circ R_{\varphi}u_0\big)(x)\big|
\le C_1\int_{\R^d}K_t(y)\,|\tilde{u}_0(x-y)|\,dy
\end{equation}
where
\[
K_t(y):=\frac{1}{|t|^{d/2}}\,e^{-M_\vartheta\frac{|y|^2}{|t|}}
\]
with a constant $M_\vartheta>0$ depending only on the choice of $\vartheta$.
It follows from \eqref{eq:u_0,f_I-tilde}, \eqref{eq:R-local_boundedness}, the Banach algebra property of 
$W^{m-1,p}_{\delta+1,\C}$, and Lemma \ref{lem:W-decay} and Lemma \ref{lem:W-multiplication} in 
Appendix \ref{sec:W-spaces}, that
\begin{equation}\label{eq:u_0-tilde}
\tilde{u}_0(x)= u_0(x)\det\big(\Id+i(\dd_{x}\tilde{f}_I)\big)\in L^p_{\gamma,\C}
\quad\text{and}\quad
\|\tilde{u}_0\|_{L^p_{\gamma,\C}}\le C_2\|u_0\|_{L^p_{\gamma,\C}}
\end{equation}
with a constant $C_2>0$ independent of the choice of $f_\C\in\U_\C(\varphi_0)$ and 
$u_0\in L^p_{\gamma,\C}$. We then obtain from Peetre's inequality that
\begin{align*}
\int_{\R^d}K_t(y)\,|\tilde{u}_0(x-y)|\,dy&\le
\int_{\R^d}\frac{\x^\gamma}{\1x-y\2^{\gamma}}K_t(y)\,\big(\1x-y\2^{\gamma}\tilde{u}_0(x-y)|\big)\,dy\nonumber\\
&\le 2^{|\gamma|/2}\int_{\R^d}\big(\y^{|\gamma|}K_t(y)\big)\,
\big(\1x-y\2^{\gamma}\tilde{u}_0(x-y)|\big)\,dy\,.
\end{align*}
This, together with \eqref{eq:convolution1}, \eqref{eq:u_0-tilde}, and Young's inequality about the boundedness
of the convolution map $L^1\times L^p\to L^p$, implies that
\begin{equation}\label{eq:convolution2}
\big\|R_{\varphi^{-1}}\circ S_{\varphi_\C}(t)\circ R_{\varphi}u_0\big\|_{L^p_{\gamma,\C}}
\le C_3\,\big\|K_{t,\gamma}\big\|_{L^1}\|u_0\|_{L^p_{\gamma,\C}},
\end{equation}
where
\[
K_{t,\gamma}(y):=\y^{|\gamma|}K_t(y),\quad y\in\R^d,
\]
and the constant $C_3>0$ is independent of the choice of $f_\C\in\U_\C(\varphi_0)$, $t\in\mathbb{S}_\vartheta$,
and $u_0\in L^p_{\gamma,\C}$. A direct computation shows that 
\[
\big\|K_{t,\gamma}\big\|_{L^1}\le C_4(1+|t|)^{|\gamma|/2}
\]
with a constant depending only on the choice of $\vartheta$, $\gamma$, and $d$.
This, together with \eqref{eq:convolution2} then implies that
\[
\big\|R_{\varphi^{-1}}\circ S_{\varphi_\C}(t)\circ R_{\varphi}u_0\big\|_{L^p_{\gamma,\C}}
\le C_5(1+|t|)^{|\gamma|/2}\|u_0\|_{L^p_{\gamma,\C}}
\]
with a constant independent of the choice of $f_\C\in\U_\C(\varphi_0)$, $t\in\mathbb{S}_\vartheta$,
and $u_0\in L^p_{\gamma,\C}$. Hence, \eqref{eq:bound_preliminary} follows.

Let us now prove that \eqref{eq:analyticity_preliminary} is weakly analytic. To this end, we fix
$\chi_1,\chi_2\in C^\infty_c$, $f_\C\in\U_\C(\varphi_0)$ and $h_\C\in W^{m,p}_{\delta,\C}$, 
and consider the map
\begin{equation}\label{eq:weak_analyticity}
\mathbb{S}_\vartheta\times\mathcal{O}(0)\to\C,\quad 
(t,z)\mapsto\int_{\R^d}\big(S_{\varphi_\C+z h_\C}(t)\chi_1\big)(x)\chi_2(x)\,dx,
\end{equation}
where $\mathcal{O}(0)$ is a sufficiently small open neighborhood of zero in $\C$.
In view of \eqref{eq:S(t)-congugatedC} we have that
\begin{align*}
&\int_{\R^d}\big(S_{\varphi_\C+z h_\C}(t)\chi_1\big)(x)\chi_2(x)\,dx=\\
&=\frac{1}{(4\pi t)^{d/2}}\int_{\R^d}\int_{\R^d}
e^{-\frac{((\varphi_\C(x)+z h_\C(x))-(\varphi_\C(y)+z h_\C(y))^2}{4 t}}
\det\big(\dd_y(\varphi_\C+z h_\C)\big)\chi_1(y)\chi_2(x)\,dy\,dx.
\end{align*}
Since this integral depends analytically on the parameters $z\in\mathcal{O}(0)$ and $t\in\mathbb{S}_\vartheta$
where the complex neighborhood of zero $\mathcal{O}(0)$ in $\C$ is chosen sufficiently small,
we conclude that the map \eqref{eq:weak_analyticity} is analytic. 
Recall that the space $C^\infty_c$ is dense in $L^p_\gamma$ as well as in its dual 
$\big(L^p_\gamma)^*=L^q_{-\gamma}$, $\frac{1}{p}+\frac{1}{q}=1$. 
Hence, the map \eqref{eq:analyticity_preliminary} is analytic since it is
weakly analytic and locally bounded (see e.g.\ Theorem 3.12  and the remark after it in \cite{Kato}).
\end{proof}

\subsection{A splitting in $W^{k,p}_\gamma$ for $-1<\gamma+\frac{d}{p}<0$}
Consider the space $W^{k,p}_\gamma$ where $-1<\gamma+\frac{d}{p}<0$ and $k\ge 0$.
It follows easily from the definition of the space $W^{k,p}_\gamma$ that the constants
(i.e. the functions on $\R^d$ with constant value) lie in $W^{k,p}_\gamma$. 
Denote by $\langle{\rm constants}\rangle$ the subspace in $W^{k,p}_\gamma$ consisting of all constants.

Take $\chi\in C^\infty_c$ such that $\chi\ge 0$ and $\int_{\R^d}\chi(x)\,dx=1$, and
consider the bounded linear functional $\alpha_\chi\in\big(W^{k,p}_\gamma\big)^*$,
\begin{equation}\label{eq:alpha}
\alpha_\chi(w):=\int_{\R^d}\chi(x) w(x)\,dx\,.
\end{equation}
We have the following splitting lemma.

\begin{Lem}\label{lem:constants_splitting}
Assume that $-1<\gamma+\frac{d}{p}<0$ and $k\ge 0$. Then
\begin{equation}\label{eq:constants_splitting}
W^{k,p}_\gamma=\langle{\rm constants}\rangle\oplus\Ker\alpha_\chi
\end{equation}
The subspaces in \eqref{eq:constants_splitting} are closed and the corresponding projections
are given by
\begin{equation}\label{eq:projections}
\Pi_c : W^{k,p}_\gamma\to\langle{\rm constants}\rangle,\quad w\mapsto\alpha_\chi(w),
\quad\text{\rm and}\quad
\Pi_o : W^{k,p}_\gamma\to\Ker\alpha_\chi,\quad w\mapsto w-\alpha_\chi(w)\,.
\end{equation}
\end{Lem}

The proof of the Lemma is immediate and follows from the fact that $\alpha_\chi(1)=1\ne 0$.
Note that the splitting \eqref{eq:constants_splitting} and the projections \eqref{eq:projections}
naturally extend to the corresponding complexified spaces.

Assume that $m>1+\frac{d}{p}$, $\delta+\frac{d}{p}>-1$, $-1<\gamma+\frac{d}{p}<0$, and $0\le k\le m$. 
Then, for any $\varphi\in\D^{m,p}_\delta$, $w\in W^{k,p}_\gamma$, and $t>0$, we have 
\begin{align*}
\Pi_c\big(R_\varphi\circ S(t)\circ R_{\varphi^{-1}}w\big)&=
\int_{\R^d}\chi(x)\big(R_\varphi\circ S(t)\circ R_{\varphi^{-1}}w\big)(x)\,dx\\
&=\int_{\R^d}\big(\chi\circ\varphi^{-1}\big)(y)\big(S(t)(w\circ\varphi^{-1})\big)(y)\,
\frac{1}{\det(\dd_x\varphi)|_{x=\varphi^{-1}(y)}}\,dy\\
&=\int_{\R^d}\Big(S(t)\big(\chi\circ\varphi^{-1}/\det(\dd_{\varphi^{-1}}\varphi\big)\Big)(y)
\big(w\circ\varphi^{-1}\big)(y)\,dy\\
&=\int_{\R^d}\Big(\big(R_\varphi\circ S(t)\circ R_{\varphi^{-1}}\big)\big(\chi/\det(\dd\varphi)\big)\Big)(x)\,
\det(\dd_x\varphi)\,w(x)\,dx\,.
\end{align*}
Consider the map
\begin{equation}\label{eq:conjugate_dual}
\D^{m,p}_\delta\to L^q_{-\gamma}\subseteq\big(W^{k,p}_\gamma\big)^*,\quad
\varphi\mapsto\Big(\big(R_\varphi\circ S(t)\circ R_{\varphi^{-1}}\big)\big(\chi/\det(\dd\varphi)\big)\Big)\det(\dd\varphi)
\end{equation}
where $\frac{1}{p}+\frac{1}{q}=1$ and the bounded embedding $L^q_{-\gamma}\subseteq\big(W^{k,p}_\gamma\big)^*$ is
given by H\"older's inequality. We have the following Lemma.

\begin{Lem}\label{lem:conjugate_dual}
Assume that $m>1+\frac{d}{p}$, $\delta+\frac{d}{p}>-1$, $-1<\gamma+\frac{d}{p}<0$, and $0\le k\le m$.
Then, for any given $\varphi_0=\id+f_0\in\D^{m,p}_\delta$ and $0<\vartheta<\pi/2$ there exists
an open neighborhood $\U_\C(\varphi_0)$ of $f_0\in W^{m,p}_{\delta,\C}$ so that the map
\eqref{eq:conjugate_dual} extends to an analytic map 
$\U_\C(\varphi_0)\times\mathbb{S}_\vartheta\to\big(W^{k,p}_{\gamma,\C}\big)^*$.
\end{Lem}
\begin{proof}[Proof of Lemma \ref{lem:conjugate_dual}]
The map \eqref{eq:conjugate_dual} is a composition of the following maps
\begin{equation}\label{eq:conjugate_dual1}
\D^{m,p}_\delta\to L^q_{-\gamma},\quad\varphi\mapsto\chi/\det(\dd\varphi),
\end{equation}
and 
\begin{equation}\label{eq:conjugate_dual2}
\D^{m,p}_\delta\times L^q_{-\gamma}\to L^q_{-\gamma}\subseteq\big(W^{k,p}_{\gamma}\big)^*,\quad
(\varphi,w)\mapsto \det(\dd\varphi)\,\Big(R_\varphi\circ S(t)\circ R_{\varphi^{-1}}w\Big)\,.
\end{equation}
The analyticity of \eqref{eq:conjugate_dual1} follows easily from the Banach algebra property 
in $W^{m-1,p}_{\delta+1}$.
The analyticity of \eqref{eq:conjugate_dual2} follows from Proposition \ref{prop:S(t)-conjugated_preliminary}.
\end{proof}

As a corollary from Lemma \ref{lem:conjugate_dual} we obtain the analyticity of the evolution of the projection of
the heat flow onto the first component in \eqref{eq:constants_splitting}.

\begin{Prop}\label{prop:constants}
Assume that $m>1+\frac{d}{p}$, $\delta+\frac{d}{p}>-1$, $-1<\gamma+\frac{d}{p}<0$, $0\le k\le m$.
Then, for any given $\varphi_0=\id+f_0\in\D^{m,p}_\delta$ and $0<\vartheta<\pi/2$ there exists
an open neighborhood $\U_\C(\varphi_0)$ of $f_0\in W^{m,p}_{\delta,\C}$ so that the map
\[
(0,\infty)\times\D^{m,p}_\delta\to\big(W^{k,p}_\gamma\big)^*,\quad
(t,\varphi)\mapsto\Pi_c\circ\big(R_\varphi\circ S(t)\circ R_{\varphi^{-1}}\big),
\]
extends to an analytic map $\mathbb{S}_\vartheta\times\U_\C(\varphi_0)\to\big(W^{k,p}_{\gamma,\C}\big)^*$.
\end{Prop}

\subsection{The bootstrapping argument}\label{sec:bootstrapping}
Assume that $m>1+\frac{d}{p}$, $\delta+\frac{d}{p}>-1$, and $\gamma\in\R$.
Then, by Lemma \ref{lem:conjugate_operators} in Appendix \ref{sec:W-spaces}, 
for any $0\le k\le m-2$ the map
\[
\D^{m,p}_\delta\times W^{k+2,p}_\gamma\to W^{k,p}_{\gamma+2},\quad
(\varphi,f)\mapsto\Delta_\varphi f:=R_\varphi\circ\Delta\circ R_{\varphi^{-1}}f\,,
\]
is well defined and analytic. Since the Laplace operator
\[
\Delta : W^{k+2,p}_\gamma\to W^{k,p}_{\gamma+2}
\]
commutes with the heat semigroup $\{S(t)\}_{t\ge 0}$ on $W^{k,p}_\gamma$,
for any $\varphi\in\D^{m,p}_\delta$, $0\le k\le m-2$, and $t>0$, we have
the commutative diagram
\begin{equation}\label{eq:diagram1}
\begin{tikzcd}
W^{k+2,p}_\gamma\arrow[d, swap, "\Delta_\varphi"]\arrow[r, "S_\varphi(t)"]&
W^{k+2,p}_\gamma\arrow[d, "\Delta_\varphi"]\\
W^{k,p}_{\gamma+2}\arrow[r, "S_\varphi(t)"]&W^{k,p}_{\gamma+2}
\end{tikzcd}
\end{equation}
where $S_\varphi(t)= R_\varphi\circ S(t)\circ R_{\varphi^{-1}}$.
Moreover, since by Lemma \ref{lem:continuity_composition_general} in Appendix \ref{sec:W-spaces}
the right translation 
\begin{equation}\label{eq:right_translation2}
R_\varphi : W^{l,p}_\gamma\to W^{l,p}_\gamma,\quad 0\le l\le m,\quad\gamma\in\R,
\end{equation}
is an isomorphism of Banach spaces, it follows from \cite{McOwen}, \cite[Section 3]{McOwenTopalov4},
that the conjugated Laplace operator
\begin{equation}\label{eq:Laplace_conjugated}
\Delta_\varphi : W^{k+2,p}_\gamma\to W^{k,p}_{\gamma+2},\quad 0\le k\le m-2,
\end{equation}
has the following properties:
\begin{itemize}
\item[(a)]  For $\gamma+\frac{d}{p}\in(-1,0)$ the map \eqref{eq:Laplace_conjugated} is onto with
$\Ker\Delta_\varphi$ consisting of the constant valued functions on $\R^d$;
\item[(b)]  For $\gamma+\frac{d}{p}\in(0,d-2)$, $d\ge 3$, the map \eqref{eq:Laplace_conjugated} is an isomorphism;
\item[(c)]  For $\gamma+\frac{d}{p}\in(\ell,\ell+1)$, $\ell\ge d-2$, the map \eqref{eq:Laplace_conjugated}
is injective with closed range,
\begin{equation}\label{eq:range_identity1}
\Range\big(\Delta_\varphi\big)= R_{\varphi}\big(\Range\Delta\big),
\end{equation}
of co-dimension $\ell-d+3$;
\end{itemize}
Let us now fix $\varphi_0=\id+f_0\in\D^{m,p}_\delta$ and assume that $\gamma+\frac{d}{p}>-1$, 
$\gamma+\frac{d}{p}\notin\Z$. Then, we have that the range of
\eqref{eq:Laplace_conjugated} with $\varphi$ replaced by $\varphi_0$ has a finite co-kernel. 
Let $C^{k,p}_{\gamma+2,\varphi_0}$ be a (finite dimensional) closed complement of $\Range\big(\Delta_{\varphi_0}\big)$ in 
$W^{k,p}_{\gamma+2}$, i.e.
\begin{equation}\label{eq:range_identity2}
W^{k,p}_{\gamma+2}=C^{k,p}_{\gamma+2,\varphi_0}\oplus\Range\big(\Delta_{\varphi_0}\big).
\end{equation}
Since $\D^{m,p}_\delta$ is a topological group and since the right translation \eqref{eq:right_translation2} is continuous, 
one easily obtains from \eqref{eq:range_identity2} and the fact that
$\Range\big(\Delta_\varphi\big)=R_{\widetilde{\varphi}}\big(\Range\big(\Delta_{\varphi_0}\big)\big)$
where $\widetilde{\varphi}=\varphi_0^{-1}\circ\varphi$ (see \eqref{eq:range_identity1})
that there exists an open neighborhood $U(\varphi_0)$ of $f_0$ in $W^{k,p}_\delta$ such that 
\begin{equation}\label{eq:range_identity3}
W^{k,p}_{\gamma+2}=C^{k,p}_{\gamma+2,\varphi_0}\oplus\Range\big(\Delta_\varphi\big).
\end{equation}
for any $\varphi=\id+f\in\D^{m,p}_\delta$ with $f\in U(\varphi_0)$. 
The splitting \eqref{eq:range_identity3} allows us to identify the factor space
$W^{k,p}_{\gamma+2}/C^{k,p}_{\gamma+2,\varphi_0}$ with $\Range\big(\Delta_\varphi\big)$. 
It follows from the properties (a), (b), and (c) above that
\begin{equation}\label{eq:kernel}
\Ker\Delta_\varphi=\Ker\Delta_{\varphi_0}=\left\{
\begin{array}{cc}
\langle{\rm constants}\rangle,&\quad-1<\gamma+\frac{d}{p}<0,\\
\{0\},\quad&\quad\gamma+\frac{d}{p}>0.
\end{array}
\right.
\end{equation}
For simplicity of notation, denote again by $\Delta_\varphi$ the induced by \eqref{eq:Laplace_conjugated} map 
$W^{k+2,p}_\gamma\big/\Ker\Delta_{\varphi_0}\to W^{k,p}_{\gamma+2}$ and 
let $\iota : W^{k,p}_{\gamma+2}\to W^{k,p}_{\gamma+2}\big/C^{k,p}_{\gamma+2,\varphi_0}$
be the natural projection onto the factor space.
Then, we obtain from \eqref{eq:range_identity3} and the open mapping theorem that
for any $\varphi=\id+f\in\D^{m,p}_\gamma$ with $f\in U(\varphi_0)$ the map
\begin{equation}\label{eq:range_identity4}
{\widetilde\Delta}_{\varphi} : W^{k+2,p}_\gamma\big/\Ker\Delta_{\varphi_0}\to 
W^{k,p}_{\gamma+2}\big/C^{k,p}_{\gamma+2,\varphi_0},\quad {\widetilde\Delta}_{\varphi}:=\iota\circ\Delta_\varphi,
\end{equation}
is an isomorphism of Banach spaces.
(Note that the spaces in \eqref{eq:range_identity4} do not depend on the choice of $\varphi$.)
In particular, the inverse of \eqref{eq:range_identity4} is well defined.
Moreover, by Lemma \ref{lem:conjugate_operators} in Appendix \ref{sec:W-spaces}, the map
\[
U(\varphi_0)\to
\LL\Big(W^{k+2,p}_\gamma\big/\Ker\Delta_{\varphi_0},W^{k,p}_{\gamma+2}\big/C^{k,p}_{\gamma+2,\varphi_0}\Big),
\quad\varphi\mapsto\widetilde{\Delta}_\varphi,\quad
\]
is analytic. Hence, the map
\begin{equation}\label{eq:range_identity5}
U(\varphi_0)\to
\LL\Big(W^{k,p}_{\gamma+2}\big/C^{k,p}_{\gamma+2,\varphi_0},W^{k+2,p}_\gamma\big/\Ker\Delta_{\varphi_0}\Big),
\quad\varphi\mapsto\widetilde{\Delta}_\varphi^{-1},\quad
\end{equation}
is also analytic. (This follows easily by using Neumann series.)
Hence, for any given $\varphi=\id+f\in\D^{m,p}_\gamma$ with $f\in U(\varphi_0)$, we can amend 
the commutative diagram \eqref{eq:diagram1} and obtain 
the commutative diagram\footnote{Note that $S_\varphi(t) 1=1$ for $t\ge 0$ and $f\in U(\varphi_0)$.}
\begin{equation}\label{eq:diagram2}
\begin{tikzcd}
W^{k+2,p}_\gamma\big/\Ker\Delta_{\varphi_0}
\arrow[d, swap, "\Delta_\varphi"]\arrow[r, "\widetilde{S}_\varphi(t)"]\arrow[r]&
W^{k+2,p}_\gamma\big/\Ker\Delta_{\varphi_0}\arrow[d, swap, "\Delta_\varphi"]\\
W^{k,p}_{\gamma+2}\arrow[r, "S_\varphi(t)"]&W^{k,p}_{\gamma+2}\arrow[r, "\iota"]&
W^{k,p}_{\gamma+2}\big/C^{k,p}_{\gamma+2,\varphi_0}
\arrow[ul, swap,"\widetilde{\Delta}_\varphi^{-1}"]
\end{tikzcd}
\end{equation}
where 
\begin{equation}\label{eq:widetildeS}
\widetilde{S}_\varphi(t) : 
W^{k+2,p}_\gamma\big/\Ker\Delta_{\varphi_0}\to W^{k+2,p}_\gamma\big/\Ker\Delta_{\varphi_0}
\end{equation}
denotes the natural action of $S_\varphi(t)$ on the factor space $W^{k+2,p}_\gamma\big/\Ker\Delta_{\varphi_0}$.
Note that the maps \eqref{eq:range_identity5} and
\[
U(\varphi_0)\to\LL\Big(W^{k+2,p}_\gamma\big/\Ker\Delta_{\varphi_0},W^{k,p}_{\gamma+2}\Big),\quad
\varphi\mapsto\Delta_\varphi,
\]
appearing as vertical arrows in \eqref{eq:diagram2} are analytic. This implies the following Lemma.

\begin{Lem}\label{lem:induction_lemma}
Let $m>1+\frac{d}{p}$, $\delta+\frac{d}{p}>-1$, $\gamma+\frac{d}{p}>-1$, $\gamma+\frac{d}{p}\notin\Z$, 
and $0\le k\le m-2$. Assume that for a given $\varphi_0=\id+f_0\in\D^{m,p}_\delta$ and $0<\vartheta<\pi/2$ there exists
an open neighborhood $\U_{\C}(\varphi_0)$ of $f_0$ in $W^{m,p}_\delta$ so that the conjugated heat flow
extends to an analytic map $\mathbb{S}_\vartheta\times\U_{\C}(\varphi_0)\to\LL\big(W^{k,p}_{\gamma+2,\C}\big)$, 
$(t,f_\C)\mapsto S_{\varphi_\C}(t)$, $\varphi_\C=\id+f_\C$, so that
$\big\|S_{\varphi_\C}(t)\big\|_{\LL(W^{k,p}_{\gamma+2,\C})}\le C\big(1+|t|\big)^{|\gamma|/2}$
with a constant $C>0$ independent of the choice of $f_\C\in\U_{\C}(\varphi_0)$ and $t\in\mathbb{S}_\vartheta$. 
Then, the map \eqref{eq:widetildeS} extends to an analytic map 
\begin{equation*}
\U_\C(\varphi_0)\to\LL\big(W^{k+2,p}_{\gamma,\C}/\Ker\Delta_{\varphi_0}\big),\quad
\varphi_\C\mapsto\widetilde{S}_{\varphi_\C}(t),
\end{equation*}
where the neighborhood $\U_{\C}(\varphi_0)$ is possibly shrunk, so that
\begin{equation}\label{eq:bound_tilde}
\big\|\widetilde{S}_{\varphi_\C}(t)\big\|_{\LL(W^{k+2,p}_{\gamma,\C}/\Ker\Delta_{\varphi_0})}\le
C\big(1+|t|\big)^{|\gamma|/2}
\end{equation}
with a constant $C>0$ independent of the choice of $f_\C\in\U_{\C}(\varphi_0)$ and $t\in\mathbb{S}_\vartheta$. 
\end{Lem}

\noindent The Lemma follows directly from the commutative diagram \eqref{eq:diagram2} and the analyticity of the maps
represented by the vertical arrows in the diagram.

\medskip

We will now use an induction argument involving Lemma \ref{lem:induction_lemma} and 
Proposition \ref{prop:S(t)-conjugated_preliminary}, and an interpolation result on $W$-spaces 
due to Triebel (see Lemma \ref{lem:interpolation} in Appendix A, \cite{Triebel}), 
to prove the main result in this Section.

\begin{Prop}\label{prop:S(t)-conjugated}
Assume that $m>1+\frac{d}{p}$, $\delta+\frac{d}{p}>-1$, and $\gamma+\frac{d}{p}>-1$. 
Then, for any given $\varphi_0=\id+f_0\in\D^{m,p}_\delta$ and $0<\vartheta<\pi/2$ there exists
an open neighborhood $\U_\C(\varphi_0)$ of $f_0$ in $W^{m,p}_{\delta,\C}$ so that 
for any $0\le k\le m$ the map
\begin{equation}\label{eq:analyticity_general}
\mathbb{S}_\vartheta\times\U_\C(\varphi_0)\to\LL\big(W^{k,p}_{\gamma,\C}\big),\quad
(t,f_\C)\mapsto S_{\varphi_\C}(t),
\end{equation}
is analytic. Moreover, there exists a constant $C>0$ such that
\begin{equation}\label{eq:bound_general}
\big\|S_{\varphi_\C}(t)\big\|_{\LL(W^{k,p}_{\gamma,\C})}\le C\big(1+|t|\big)^{|\gamma|/2}
\end{equation}
for any $0\le k\le m$, $f_\C\in\U_\C(\varphi_0)$, and $t\in\mathbb{S}_\vartheta$.
\end{Prop}

\begin{Rem}
In fact, Proposition \ref{prop:S(t)-conjugated} holds for any $\gamma\in\R$.
The proof follows the lines of the proof of Proposition \ref{prop:S(t)-conjugated} below
and uses the fact that for $-\ell-1<\gamma+\frac{d}{p}<-\ell$, $\ell\in\Z_{\ge 0}$,
the map \eqref{eq:Laplace_conjugated} is onto with kernel consisting of harmonic polynomials on 
$\R^d$ of degree $\le\ell$ (cf. \cite{McOwen}). 
Since we will not need this more general result, we omit the details.
\end{Rem}

\begin{proof}[Proof of Proposition \ref{prop:S(t)-conjugated}]
Assume that $m>1+\frac{d}{p}$, $\delta+\frac{d}{p}>-1$, and $\gamma+\frac{d}{p}>-1$.
Take $\varphi_0=\id+f_0\in\D^{m,p}_\delta$ and $0<\vartheta<\pi/2$.
Let us first consider the case when $m$ is {\em even} and $\gamma+\frac{d}{p}\notin\Z$.
We will argue by induction in $k\ge 0$ and interpolation. 
For $k=0$ the Proposition follows from Proposition \ref{prop:S(t)-conjugated_preliminary}.
Then, we apply Lemma \ref{lem:induction_lemma} with $k=0$ to obtain that there exists
an open neighborhood $\U_\C(\varphi_0)$ of $f_0$ in $W^{m,p}_{\delta,\C}$ such that
\begin{equation}\label{eq:analyticity_k=2}
\mathbb{S}_\vartheta\times\U_\C(\varphi_0)\to\LL\big(W^{2,p}_{\gamma,\C}\big/\Ker\Delta_{\varphi_0}\big),\quad
(t,f_\C)\mapsto\widetilde{S}_{\varphi_\C}(t),
\end{equation}
is analytic and that the estimate \eqref{eq:bound_tilde} holds for $k=2$. 
If $\gamma+\frac{d}{p}>0$ then $\Ker\Delta_{\varphi_0}=\{0\}$ (cf. \eqref{eq:kernel}), and hence,
the map
\begin{equation}\label{eq:analyticity_k=2'}
\mathbb{S}_\vartheta\times\U_\C(\varphi_0)\to\LL\big(W^{2,p}_{\gamma,\C}),\quad
(t,f_\C)\mapsto S_{\varphi_\C}(t),
\end{equation}
is analytic and the estimate \eqref{eq:bound_general} holds for $k=2$. 
If $-1<\gamma+\frac{d}{p}<0$ the analyticity of \eqref{eq:analyticity_k=2'} follows from
Proposition \ref{prop:constants} and the fact that $\Ker\Delta_{\varphi_0}=\langle\rm constants\rangle$.
More specifically, in this case we argue as follows: The splitting \eqref{eq:constants_splitting} in 
Lemma \ref{lem:constants_splitting} gives a splitting of the corresponding complexified spaces
\begin{equation}\label{lem:constants_splittingC}
W^{2,p}_{\gamma,\C}=\langle{\rm constants}\rangle_\C\oplus\Ker\alpha_\chi
\end{equation}
where $\alpha_\chi\in\big(W^{2,p}_{\gamma,\C}\big)^*$ is the complex linear functional given by \eqref{eq:alpha}.
For any given $\varphi_\C=\id+f_\C$ with $f_\C\in\U_\C(\varphi_0)$ and $t\in\mathbb{S}_\vartheta$ we have
the commutative diagram
\begin{equation}\label{eq:diagram3}
\begin{tikzcd}
W^{2,p}_{\gamma,\C}\arrow[r]\arrow[rd, swap, "\Pi_o"]\arrow[rrr, bend left, "S_{\varphi_\C}(t)"]&
W^{2,p}_{\gamma,\C}\big/\langle{\rm constants}\rangle_\C
\arrow[d, swap, "\widetilde{\Pi}_o"]\arrow[r, "\widetilde{S}_{\varphi_\C}(t)"]&
W^{2,p}_{\gamma,\C}\big/\langle{\rm constants}\rangle_\C\arrow[d, "\widetilde{\Pi}_o"]&
W^{2,p}_{\gamma,\C}\arrow[l]\arrow[ld, "\Pi_o"]\\
&\Ker\alpha_\chi\arrow[r, "\widetilde{S}_{\varphi_\C}^0(t)"]&\Ker\alpha_\chi&
\end{tikzcd}
\end{equation}
where $\widetilde{S}_{\varphi_\C}^0(t):=\widetilde{\Pi}_o\circ\widetilde{S}_{\varphi_\C}(t)\circ\widetilde{\Pi}_o^{-1}$ and
$\widetilde{\Pi}_o : W^{2,p}_{\gamma,\C}\big/\langle{\rm constants}\rangle\to\Ker\alpha_\chi$ is
the isomorphisms induced by the projection $\Pi_o$ onto the second component in \eqref{lem:constants_splittingC} 
(cf. \eqref{eq:projections}). In view of the commutative diagram \eqref{eq:diagram3}, we have that
\[
\Pi_o\circ S_{\varphi_\C}(t)=\widetilde{S}_{\varphi_\C}^0(t)\circ\Pi_o\,.
\]
It now follows from the analyticity of the map \eqref{eq:analyticity_k=2} and the fact that
$\widetilde{S}_{\varphi_\C}^0(t)=\widetilde{\Pi}_o\circ\widetilde{S}_{\varphi_\C}(t)\circ\widetilde{\Pi}_o^{-1}$
(cf. \eqref{eq:diagram3}) that the map
\begin{equation}\label{eq:analyticity_k=2''}
\mathbb{S}_\vartheta\times\U_\C(\varphi_0)\to\LL\big(W^{2,p}_{\gamma,\C},\Ker\alpha_\chi\big),\quad
(t,f_\C)\mapsto\Pi_o\circ S_{\varphi_\C}(t)=\widetilde{S}_{\varphi_\C}^0(t)\circ\Pi_o,
\end{equation}
is analytic. The analyticity of the map
\begin{equation}\label{eq:analyticity_k=2'''}
\mathbb{S}_\vartheta\times\U_\C(\varphi_0)\to\LL\big(W^{2,p}_{\gamma,\C},\langle{\rm constants}\rangle_\C\big),\quad
(t,f_\C)\mapsto\Pi_c\circ S_{\varphi_\C}(t),
\end{equation}
where $\Pi_c$ is the projection onto the first component in \eqref{lem:constants_splittingC}, follows from 
Proposition \ref{prop:constants}. (If necessary, we shrink the neighborhood $\U_\C(\varphi_0)$ one more time so that 
Proposition \ref{prop:constants} can be applied.)
The analyticity of the map \eqref{eq:analyticity_k=2'} then follows from \eqref{eq:analyticity_k=2''}, \eqref{eq:analyticity_k=2'''},
and \eqref{lem:constants_splittingC}.
The estimate \eqref{eq:bound_general} for $k=2$ follows from the estimate \eqref{eq:bound_tilde} for $k=0$.
By continuing the induction process, we prove Proposition \ref{prop:S(t)-conjugated} for even values of $m$ and $k$ with 
$0\le k\le m$ and for $\gamma+\frac{d}{p}>-1$ with $\gamma+\frac{d}{p}\notin\Z$.
For $m$ even and $k$ odd with $0\le k\le m$, we interpolate between the spaces 
$W^{k-1,p}_{\gamma,\C}$ and $W^{k+1,p}_{\gamma,\C}$ to conclude that \eqref{eq:analyticity_general} and 
\eqref{eq:bound_general} hold for $\gamma+\frac{d}{p}\notin\Z$ 
(see Lemma \ref{lem:interpolation}\,(i) in Appendix \ref{sec:W-spaces}).

In the case when $m>1+\frac{d}{p}$ is odd and $\gamma+\frac{d}{p}\notin\Z$, we argue as follows: 
First, we use the arguments above to conclude that \eqref{eq:analyticity_general} is analytic and 
that \eqref{eq:bound_general} holds for $0\le k<m$ even. We then interpolate between $W^{0,p}_{\gamma,\C}$
and $W^{2,p}_{\gamma,\C}$ to obtain that \eqref{eq:analyticity_general} and \eqref{eq:bound_general}
hold for $k=1$. Then, we use \eqref{eq:diagram2} and the induction argument above starting with $k=1$ to prove
that the Proposition holds also for $0\le k\le m$, $k$ odd.

Finally, the case when $\gamma+\frac{d}{p}\in\Z$ is treated by interpolation between
$W^{k,p}_{\gamma_1,\C}$ and $W^{k,p}_{\gamma_2,\C}$ where $\gamma_1<\gamma<\gamma_2$ 
and $\gamma_1$ and $\gamma_2$ are chosen so that $\gamma_j+\frac{d}{p}>-1$ and
$\gamma_j+\frac{d}{p}\notin\Z$, $j=1,2$ (see Lemma \ref{lem:interpolation}\,(ii) in Appendix \ref{sec:W-spaces}).
\end{proof}

\medskip

In view of Proposition \ref{prop:S(t)-conjugated}, consider for any $t>0$ and $\ell\ge 0$ 
the $\ell$-th differential in the direction of the first variable 
\[
\dd_\varphi^\ell\big(S_{(\cdot)}(t)\big)\in\LL\big(\big(W^{m,p}_\gamma\big)^\ell,\LL\big(W^{m,p}_\gamma\big)\big)
\]
of the map
\begin{equation}\label{eq:S(t)-conjugated_real}
(0,\infty)\times\D^{m,p}_\delta\to\LL\big(W^{m,p}_\gamma\big),\quad
(t,\varphi)\mapsto S_\varphi(t).
\end{equation}
As a corollary from Proposition \ref{prop:S(t)-conjugated} we obtain the following result.

\begin{Prop}\label{prop:S(t)-conjugated_derivatives}
Assume that $m>1+\frac{d}{p}$, $\delta+\frac{d}{p}>-1$, and $\gamma+\frac{d}{p}>-1$. 
Then, the map \eqref{eq:S(t)-conjugated_real} is analytic.
Moreover, for any given $\varphi_0=\id+f_0\in\D^{m,p}_\delta$ and
for any $\ell_0\ge 1$ there exist an open neighborhood $U(\varphi_0)$ of $f_0$ in $W^{m,p}_\delta$
and a positive constant $C>0$ such that for any $0\le\ell\le\ell_0$,
\[
\big\|\dd_\varphi^\ell\big(S_{(\cdot)}(t)\big)\big\|_{\LL((W^{m,p}_\gamma)^\ell,\LL(W^{m,p}_\gamma))}\le 
C\,\big(1+|t|\big)^{|\gamma|/2},
\]
for any $\varphi=\id+f\in\D^{m,p}_\delta$ with $f\in U(\varphi_0)$ and
for any $t\in[0,\infty)$.\footnote{Note that $S(0)$ is the identity map on $W^{m,p}_\gamma$.}
\end{Prop}

\begin{proof}[Proof of Proposition \ref{prop:S(t)-conjugated_derivatives}]
The proof follows directly from Proposition \ref{prop:S(t)-conjugated} and the Cauchy estimate
for complex analytic functions (see, e.g.,\ \cite[Lemma A.2]{KP}).
\end{proof}

We are now ready to prove Theorem \ref{th:conjugate_Laplace_operator}.

\begin{proof}[Proof of Theorem \ref{th:conjugate_Laplace_operator}]
Assume that $m>1+\frac{d}{p}$, $\delta+\frac{d}{p}>-1$, and $\gamma\in\R$.
Since $\D^{m,p}_\delta$ is a group, one concludes from 
the continuity of \eqref{eq:right_translation} that for any given $\varphi\in\D^{m,p}_\delta$
the right translation
\begin{equation}\label{eq:R-isomorphisms}
R_\varphi : W^{m,p}_\gamma\to W^{m,p}_\gamma
\end{equation}
is an isomorphism of Banach spaces (with inverse given by $R_{\varphi^{-1}}$).
Item (i) then follows from the corresponding property of the Laplace operator $\Delta$ when considered
as an unbounded operator on $W^{m,p}_\gamma$ with domain $\widetilde{W}^{m+2,p}_\gamma$
(see \cite[Theorem 2.1\,(i)]{McOwenTopalov5}).
The resolvent estimate in \cite[Theorem 2.1\,(ii)]{McOwenTopalov5} and \eqref{eq:R-isomorphisms}
show that $\Delta_\varphi$ is a sectorial operator with the required properties. This implies (ii).
Let us now prove item (iii).
Since $\{e^{t\Delta}\}_{t\ge 0}$ is a strongly continuous semigroup on $W^{m,p}_\gamma$ with
domain $\widetilde{W}^{m+2,p}_\gamma$, for any $v_0\in\widetilde{W}^{m+2,p}_\gamma$,
$v(t):=e^{t\Delta} v_0$ belongs to 
$C\big([0,\infty),\widetilde{W}^{m+2,p}_\gamma\big)\cap C^1\big([0,\infty),W^{m,p}_\gamma\big)$
and satisfies the heat equation $v_t=\Delta v$, $v|_{t=0}=v_0$. 
This implies that $u(t):=R_\varphi\circ e^{t\Delta}\circ R_{\varphi^{-1}}u_0$ with
$u_0\in D_{\Delta_\varphi}= R_\varphi\big(\widetilde{W}^{m+2,p}_\gamma\big)$, belongs to
$C\big([0,\infty), D_{\Delta_\varphi}\big)\cap C^1\big([0,\infty),W^{m,p}_\gamma\big)$
and
\[
u_t(t)=R_\varphi\big(\Delta\,e^{t\Delta}(R_{\varphi^{-1}}u_0)\big)=\Delta_\varphi u(t),\quad
u|_{t=0}=u_0.
\]
By the uniqueness theorem \cite[Theorem 2.2.2 ]{LL} and item (i) (proven above),
we conclude that for any $u_0\in D_{\Delta_\varphi}$,
\[
e^{t\Delta_\varphi}u_0=R_\varphi\circ e^{t\Delta}\circ R_{\varphi^{-1}}u_0\quad\text{\rm for}\quad t>0.
\]
Since $D_{\Delta_\varphi}$ is densely embedded in $W^{m,p}_\gamma$ we
obtain that $e^{t\Delta_\varphi}=R_\varphi\circ e^{t\Delta}\circ R_{\varphi^{-1}}$ for $t>0$.
Item (iii) then follows from Proposition \ref{prop:S(t)-conjugated_derivatives} and the fact 
that $S(t)=e^{t\Delta}$ for $t>0$.
\end{proof}

\section{The Lie-Trotter product formula}\label{sec:Lie-Trotter}
In this Section we state a variant of the Lie-Trotter product formula for nonlinear semigroups.
These results are proven in Section \ref{sec:Lie-Trotter_proofs} and are a modification of 
\cite[Theorem 2.1]{Marsden}. Note that some of the conditions in \cite{Marsden} are too restrictive
for our purposes, and often do not hold (see Remark \ref{rem:condition(iii)} below for a discussion). 

Let $U$ be a nonempty open set in a Banach space $X$. We say that a map $S : [0,T)\times U\to X$, $T>0$,
satisfies the {\em local semigroup property} if for any $x\in U$ we have that $S_t(x)|_{t=0}=x$ and
$S_t\circ S_s(x)=S_{t+s}(x)$ provided that $s,t,s+t\in[0,T)$ and $S_s(x)\in U$. 
Here, and in what follows, we set $S_t(x):=S(t,x)$.
Let $N$ be an open set in a Banach space $X$. A distance function $d : N\times N\to\R_{\ge 0}$ on $N$ is 
{\em equivalent to the standard distance} on $N$ if there exist constants $0<L_1<L_2<\infty$ such that
\begin{equation}\label{eq:equivalent_distances}
L_1\,d(x,y)\le\|x-y\|\le L_2\,d(x,y)
\end{equation}
for any $x,y\in N$.

\begin{Def}\label{def:quasi-contraction}
Assume that $S : [0,T)\times U\to X$ is continuous and satisfies the local semigroup property.
Then, we say that $S$ is {\em locally equivalent to a quasi-contraction in $U$} if 
for any $x_0\in U$ there exist an open neighborhood $N(x_0)\subseteq U$ of $x_0$ in $X$,
a distance function $d : N(x_0)\times N(x_0)\to\R_{\ge 0}$ on $N(x_0)$ equivalent to the standard distance on $N(x_0)$, 
and a constant $\beta>0$, such that if the open set $V\subseteq N(x_0)$ in $X$ and $0<\tau\le T$ are such that 
$S_t(V)\subseteq N(x_0)$ for any $t\in[0,\tau)$ then 
\begin{equation*}
d\big(S_t(x),S_t(y)\big)\le e^{\beta t}d(x,y)
\end{equation*}
holds for any $x,y\in V$ and $t\in[0,\tau)$. The map $S$ is {\em globally equivalent to a quasi-contraction in $U$}
if the above holds for $N(x_0)$ replaced by $U$.
\end{Def}

\noindent  Two maps $S : [0,T)\times U\to X$ and  $G : [0,T)\times U\to X$ are called {\em locally equivalent to 
a quasi-contraction in $U$ with the same distance function} if for any $x_0\in U$ there exist an open neighborhood $N(x_0)$ and 
a constant $\beta>0$ so that if the open set $V\subseteq N(x_0)$ in $X$ and $0<\tau\le T$ are such that $G_t(V),S_t(V)\subseteq N(x_0)$
for any $t\in[0,\tau)$ then 
\begin{equation}\label{eq:quasi-contraction}
d\big(G_t(x),G_t(y)\big)\le e^{\beta t}d(x,y)\quad\text{\rm and}\quad
d\big(S_t(x),S_t(y)\big)\le e^{\beta t}d(x,y)
\end{equation}
for any $x,y\in V$ and $t\in[0,\tau)$.
We will also need the following Definition. Let $X$ and $Y$ be Banach spaces and $U$ be a nonempty
open set in $X$.

\begin{Def}\label{def:C^k_along_curves}
A continuous map $S : [0,T)\times U\to Y$, $T>0$, is of {\em class $C^k$ along curves} for some $k\in\Z_{\ge 0}$ if
for any $(t,x)\in[0,T)\times U$, multi-indexes $\mu,\alpha\in\Z^\ell_{\ge 0}$ for some $\ell\ge 1$,
$|\mu|+|\alpha|\le k$, and directions $\xi_1,...,\xi_{|\alpha|}\in X$, there exists a directional partial derivative 
\[
\big(D^{\mu,\alpha}S\big)(t,x)\big(\xi_1,...\xi_{|\alpha|}\big):=
\big(D_1^{\mu_1} D_2^{\alpha_1}\cdots D_1^{\mu_\ell} D_2^{\alpha_\ell} S\big)(t,x)\big(\xi_1,...,\xi_{|\alpha|}\big),
\]
where $D_1$ denotes the directional derivative of $S$ in the direction of the first variable $t$ and $D_2$ the
directional derivative in the direction of the second $x$, and the map
\[
[0,T)\times U\times\underbrace{X\times\cdots\times X}_{|\alpha|\text{\rm--times}}\to Y,\quad
(t,x,\xi)\mapsto\big(D^{\mu,\alpha}S\big)(t,x)\big(\xi_1,...\xi_{|\alpha|}\big),
\]
is continuous.
\end{Def} 

\begin{Rem}\label{rem:C^k_along_curves}
The following example explains the reasoning behind Definition \ref{def:C^k_along_curves}.
Let $\{S(t)\}_{t\ge 0}$ be a strongly continuous semigroup of bounded linear operators
$S(t)\in\LL(X)$ on $X$ with dense domain $D_A\subseteq X$ (equipped with the graph norm) and a generator $A$.
Then the map $S : [0,T)\times D_A\to X$ is not $C^1$ since $[0,T)\to\LL(D_A,X)$, $t\mapsto\big(D_2 S\big)(t,x)=S(t)$, 
is not necessarily continuous, but $S : [0,T)\times D_A\to X$ is $C^1$ along curves.
Definition \ref{def:C^k_along_curves} is a special case of the definition of $C^k$ maps for 
Fr\'echet spaces in Hamilton's paper \cite[Definition 3.6.1]{Hamilton}. 
Note that a $C^k$ map is also a $C^k$ map along curves.
The composition of $C^k$ maps along curves is again a $C^k$ map along curves (\cite[Theorem 3.6.4]{Hamilton}). 
Moreover, these maps satisfy the chain rule and Taylor's formula with remainder in integral form 
(see \cite[Section I.3]{Hamilton} for the details).
In particular, if $S : [0,T)\times U\to X$ is $C^k$ along curves and $\gamma\in C^k\big([0,T),U\big)$ then
the composed map $[0,T)\to X$, $t\mapsto S\big(t,\gamma(t)\big)$, is $C^k$. 
This last example explains our choice of name.
\end{Rem}

\medskip

For our variant of the Trotter product formula, we will require a scale of Banach spaces 
$D\subseteq Z\subseteq X$, as well as two additional Banach spaces $\widetilde{X}\supseteq X$ and
$\widetilde{Z}\supseteq Z$. For the convenience of the reader, we summarize the various maps appearing in 
Proposition \ref{prop:product_formula} below in the commutative diagram 
\begin{equation}\label{eq:diagram4}
\begin{tikzcd}
\widetilde{X}&U\arrow[r, "S_t{,}\,G_t"]\arrow[l, swap, "\Xi{,}\,\,\E"]&X\\
\widetilde{Z}&U_Z\arrow[u, hook]\arrow[l, swap, "\Xi{,}\,\,\E"]\arrow[r, "S_t{,}\,G_t"]\arrow[ur,"S_t"]&Z\arrow[u, hook]\\
&U_D\arrow[u, hook]\arrow[ur,"S_t{,}\, G_t"']&
\end{tikzcd}
\end{equation}
where $t\in[0,T)$ and the (hooked) vertical arrows denote the identity embedding of 
subsets in Banach spaces, and the various restrictions of a map are denoted by the same letter.
We have denoted by $U\subseteq X$ the open set $U\cap Z$ in $Z$. Note that $U_Z$ need not be bounded in $Z$
when $U$ is bounded in $X$. Similarly, denote by $U_D$ the open set $U\cap D$ in $D$.

\begin{Prop}\label{prop:product_formula}
Let $(Z,\|\cdot\|_Z)$ and $(X,\|\cdot\|)$ be Banach spaces such that the embedding $Z\subseteq X$
is bounded and $Z$ is reflexive and dense in $X$. 
Assume that there exist a nonempty open set $U$ in $X$ and a constant $T>0$ such that the maps
\begin{equation}\label{eq:G}
G : [0,T)\times U\to X\quad\text{\em and}\quad G : [0,T)\times U_Z\to Z,
\end{equation}
and
\begin{equation}\label{eq:S}
\quad S : [0,T)\times U\to X\quad\text{\em and}\quad S : [0,T)\times U_Z\to Z
\end{equation}
are continuous and satisfy the local semigroup property. In addition, we assume that:
\begin{itemize}
\item[$(a)$] The maps in \eqref{eq:G} are $C^2$ and the map $S : [0,T)\times U_Z\to X$ is $C^2$ along curves
(see Definition \ref{def:C^k_along_curves} above).
\item[$(b)$]  There exists a densely embedded Banach space $D\subseteq Z$ such that the maps
$G : [0,T)\times U_D\to Z$ and $S: [0,T)\times U_D\to Z$ are $C^1$ along curves.
\item[$(c)$] The following two conditions hold:
\begin{itemize}
\item[$1)$] The maps $G : [0,T)\times U\to X$ and $S : [0,T)\times U\to X$ are locally equivalent 
to a quasi-contraction in $U$ with the same distance function.
\item[$2)$] The maps $G : [0,T)\times U_Z\to Z$ and $S : [0,T)\times U_Z\to Z$ are locally equivalent 
to a quasi-contraction in $U_Z$ with the same distance function.
\end{itemize}
\item[$(d)$] For a given $z\in U_Z$ consider the generators 
\begin{equation*}
\E(z):=\frac{\dd}{\dd t}\Big|_{t=0}G_t(z)\quad\text{\rm and}\quad\Xi(z):=\frac{\dd}{\dd t}\Big|_{t=0} S_t(z)
\end{equation*}
of the local semigroups \eqref{eq:G}, \eqref{eq:S}, and assume that the map $\Xi : U_Z\to X$ extends to $C^1$ maps
\begin{equation}\label{eq:Xi-regularity}
\Xi : U_Z\to\widetilde{Z}\quad\text{\rm and}\quad\Xi : U\to\widetilde{X}
\end{equation}
where $\widetilde{Z}$ and $\widetilde{X}$ are Banach spaces such that $X\subseteq\widetilde{X}$ and
$Z\subseteq\widetilde{Z}\subseteq{X}$ are bounded embeddings.
(By $(a)$, the map $\E : U\to X$ is $C^1$.)
\end{itemize}
Then, for any given $z_0\in U_Z$ there exist $0<\tau\le T$ and an open neighborhood $O(z_0)\subseteq U_Z$ of $z_0$ in $Z$
such that for any given $z\in O(z_0)$ the equation
\begin{equation}\label{eq:pde}
\frac{\dd u}{\dd t}=\E(u)+\Xi(u),\quad u|_{t=0}=z,
\end{equation}
has a unique solution
\[
u\in C_b\big([0,\tau),X\big)\cap C^1_b\big([0,\tau),\widetilde{X}\big)
\]
that depends Lipschitz continuously on the initial data $z\in O(z_0)$ and for any $t\in[0,\tau)$ we have that
$u(t)\in B\subseteq U_Z$ for some bounded closed set $B$ in $Z$. 
Moreover, the flow map
\begin{equation}\label{eq:the_flow_map}
H : [0,\tau)\times O(z_0)\to Z,\quad(t,z)\mapsto u(t,z),
\end{equation} 
is right continuous in $t$, satisfies the local semigroup property, and for any $t\in[0,\tau)$ and $z\in O(z_0)$ we have that
$(S_{t/n}\circ G_{t/n}\big)^n(z)\in B$ for any $n\ge 1$ and
\begin{equation}\label{eq:Lie-Trotter}
H_t(z)=\lim_{n\to\infty}\big(S_{t/n}\circ G_{t/n}\big)^n(z)
\end{equation}
where the limit is taken in $X$. Moreover, the convergence in \eqref{eq:Lie-Trotter} is uniform in
$z\in O(z_0)$ and $t\in[0,\tau)$, the rate of convergence is of order $O(1/n)$, and for some constant $\beta>0$,
\begin{equation}\label{eq:H-Lipschitz*}
d\big(H_t(z_1),H_t(z_2)\big)\le e^{2\beta t}d(z_1,z_2)
\end{equation}
for any $z_1,z_2\in O(z_0)$ and $t\in[0,\tau)$ where $d : O(z_0)\times O(z_0)\to\R_{\ge 0}$ is a distance function
that is defined in an open neighborhood of $z_0$ in $X$, which contains the neighborhood $O(z_0)$, and such that it is
equivalent to the standard distance in $X$.
\end{Prop}

\begin{Rem}\label{rem:condition(iii)}
The condition $(iii)$ in \cite[Theorem 2.1]{Marsden} implies that the conditions $1)$ and $2)$ in 
Proposition \ref{prop:product_formula} $(c)$ hold globally in $U$ and $U_Z$ respectively.
The proof of $(iii)$ would require much stronger global estimates on $S$ that are often not satisfied.
\end{Rem}

The following Proposition is a parametrized version of Proposition \ref{prop:product_formula}.
When applied to the Navier-Stokes equation, it establishes the existence of a zero-viscosity limit
(cf.\ Section \ref{sec:unbounded_solutions}).

\begin{Prop}\label{prop:product_formula_nu}
Assume that the conditions in Proposition \ref{prop:product_formula} are satisfied.
For any value of the parameter $0\le\nu\le 1$ consider the equation
\begin{equation}\label{eq:pde_nu}
\frac{\dd u}{\dd t}=\E(u)+\nu\,\Xi(u),\quad u|_{t=0}=z.
\end{equation}
Then, for any $z_0\in U_Z$ there exist $0<\tau\le T$ and an open neighborhood $O(z_0)\subseteq U_Z$ of
$z_0$ in $Z$ such that for any $z\in O(z_0)$ the equation \eqref{eq:pde_nu} has a unique solution
\[
u^{(\nu)}\in C_b\big( [0,\tau),X\big)\cap C_b\big( [0,\tau),\widetilde{X}\big)
\]
that depends Lipschitz continuously on the initial data, $u^{(\nu)}(t)\in Z$ for $t\in[0,\tau)$, and the map
\[
O(z_0)\times[0,1]\to C_b\big( [0,\tau),X\big)\cap C_b\big( [0,\tau),\widetilde{X}\big),
\quad (z,\nu)\mapsto u^{(\nu)},
\]
is continuous. In particular, $u^{(\nu)}\to u^{(0)}$ as $\nu\to 0+$ where 
$u^{(0)}\in C_b\big( [0,\tau),X\big)\cap C^1_b\big( [0,\tau),\widetilde{X}\big)$ is the unique solution of 
\eqref{eq:pde_nu} with $\nu=0$ and the limit is taken in $C_b\big( [0,\tau),X\big)\cap C_b\big( [0,\tau),\widetilde{X}\big)$.
\end{Prop}

The conditions $1)$ and $2)$ appearing in Proposition \ref{prop:product_formula} $(c)$
can be replaced by sufficient conditions involving global time estimates on the first and
the second derivatives of $S_t : U\to X$. 
Let $(X,\|\cdot\|)$ be a Banach space such that for some $\lambda>1$ the map 
\begin{equation}\label{eq:differentiable_norm}
x\mapsto\|x\|^\lambda,\quad X\to\R,
\end{equation}
is $C^1$ with first derivative that is uniformly bounded on bounded sets in $X$.
This condition is certainly satisfied, for example, if $X$ is a Hilbert space and $\lambda=2$.
More generally, by Lemma \ref{lem:L^p-norm} in Appendix \ref{sec:W-spaces}, the condition also holds
for the $L^p$ based Sobolev spaces for $1<p<\infty$ and $\lambda=p$. We have the following Theorem.

\begin{Th}\label{th:product_formula}
Assume that the norms in $Z$ and $X$ satisfy the property \eqref{eq:differentiable_norm}.
Then, the statements of Proposition \ref{prop:product_formula} and Proposition \ref{prop:product_formula_nu} hold
with the conditions $1)$ and $2)$ in $(c)$ replaced by the following conditions:
\begin{itemize} 
\item[$1^*)$] For any given $x_0\in U$ there exists an open 
neighborhood $N(x_0)\subseteq X$ of $x_0$ in $X$ such that the map
$S : [0,\infty)\times N(x_0)\to X$ is continuous, satisfies the local semigroup property,
and for any given $t\in(0,\infty)$ the map $S_t : N(x_0)\to X$ is $C^2$,
takes values in $N(x_0)$, and  there exist constants $M,\beta>0$ such that
\begin{equation}\label{eq:dS_t-exponential_growth1}
\|\dd_xS_t\|_{\LL(X,X)}\le M e^{\beta t}\quad\text{\rm and}\quad
\|\dd^2_xS_t\|_{\LL(X\times X,X)}\le M e^{\beta t}
\end{equation}
hold for any $x\in N(x_0)$ and for $t\in(0,\infty)$. 

\item[$2^*)$] For any given $z_0\in U_Z$ there exists an open
neighborhood $N(z_0)\subseteq Z$ of $z_0$ in $Z$ such that  the map
$S : [0,\infty)\times N(z_0)\to Z$ is continuous, satisfies the local semigroup property,
and for any given $t\in(0,\infty)$ the map $S_t : N(z_0)\to Z$ is $C^2$, 
takes values in $N(z_0)$, and there exist constants $M_Z,\beta_Z>0$ such that 
\[
\|\dd_zS_t\|_{\LL(Z,Z)}\le M_Z e^{\beta_Z t}\quad\text{\rm and}\quad
\|\dd^2_zS_t\|_{\LL(Z\times Z,Z)}\le M_Z e^{\beta_Z t}
\]
hold for any $z\in N(z_0)$ and for $t\in(0,\infty)$. 
\end{itemize}
\end{Th}

\begin{Rem} 
The condition $2^*)$ is simply a reformulation of the condition $1^*)$ for the space $Z$ and 
it is stated in such a repetitive way only for clarity.
Note also that we do not assume that $N(x_0)\subseteq U$ or that $N(z_0)\subseteq U_Z$.
In this regard, the conditions $1^*)$ and $2^*)$ mean that the maps in \eqref{eq:S} can be extended
to $[0,\infty)\times N(x_0)$ (respectively $[0,\infty)\times N(z_0)$). 
In the applications of Theorem \ref{th:product_formula} the neighborhood $U$,
in contrast to the neighborhood $N(x_0)$, is not invariant with respect to the maps in \eqref{eq:S}.
Note also that the neighborhood $N(z_0)$ in $2^*)$ is not related to the neighborhood $N(x_0)$ in $1^*)$. 
\end{Rem}

The proofs of Proposition \ref{prop:product_formula} and Theorem \ref{th:product_formula}
are given in Section \ref{sec:Lie-Trotter_proofs}.

\section{Proof of Theorem \ref{th:NS} and Theorem \ref{th:NS_nu}}\label{sec:unbounded_solutions}
Our main tool for proving Theorem \ref{th:NS} and Theorem \ref{th:NS_nu} is the variant of 
the Lie-Trotter product formula stated in Section \ref{sec:Lie-Trotter}.

\begin{proof}[Proof of Theorem \ref{th:NS}]
Assume that $m>2+\frac{d}{p}$, $-1/2<\delta+\frac{d}{p}<d+1$, and $m_0:=m+8$.
Without loss of generality we will assume that $\nu=1$.
In Lagrangian coordinates, the Navier-Stokes equation takes the form
\begin{equation}\label{eq:NS_Lagrange_coordinates}
(\dt{\varphi},\dt{v})=\E(\varphi,v)+\Xi(\varphi,v),\quad(\varphi,v)|_{t=0}=(\id,u_0),
\end{equation}
where 
\begin{eqnarray}\label{eq:E-NS}
&\E : \D^{m,p}_\delta\times W^{m,p}_\delta\to W^{m,p}_\delta\times W^{m,p}_\delta,\\
&(\varphi,v)\mapsto\Big(v,\big(R_\varphi\circ\nabla\circ\Delta^{-1}\circ R_{\varphi^{-1}}\big)\circ
\big(R_\varphi\circ Q\circ R_{\varphi^{-1}}\big)(v)\Big),\nonumber
\end{eqnarray}
is the Euler vector field where $Q(u):=\tr\big([\dd u]^2\big)$, $[\dd u]$ is the Jacobi matrix of
the vector field $u\in W^{m,p}_\delta$,
\begin{equation}\label{eq:Xi-NS}
\Xi : \D^{m,p}_\delta\times W^{m,p}_\delta\to W^{m-2,p}_{\delta+2}\times W^{m-2,p}_{\delta+2},\quad
(\varphi,v)\mapsto\Big(0,\big(R_\varphi\circ\Delta\circ R_{\varphi^{-1}}\big)(v)\Big),
\end{equation}
and $\Delta_\varphi:=R_\varphi\circ\Delta\circ R_{\varphi^{-1}}$ and $\Delta$ is the Laplace operator 
on $\R^n$. (We refer to \cite[Section 3]{McOwenTopalov4} for the mapping properties of $\Delta$ and
its inverse on weighted spaces.) The map \eqref{eq:Xi-NS} is continuous by Lemma \ref{lem:W-multiplication}
and Lemma \ref{lem:continuity_composition_general}.
Moreover, by \cite[Theorem 4.1]{McOwenTopalov4}, the vector field \eqref{eq:E-NS} is 
$C^k$ for any $k\ge 1$.\footnote{In fact, the vector field $\E$ is analytic.}
Then, by the existence of solutions of ODEs (see e.g. \cite[Chapter IV]{Lang}) for any 
$u_\bullet\in\aW^{m,p}_\delta$ there exist an open neighborhood $U$ of $(\id,u_\bullet)$ in
$\D^{m,p}_\delta\times W^{m,p}_\delta$ and $T>0$ such that for any $(\varphi,v)\in U$ there exists
a unique integral curve $[0,T)\to\D^{m,p}_\delta\times W^{m,p}_\delta$, $t\mapsto\G(t;\varphi,v)=\G_t(\varphi,v)$,
of the vector field $\E$ that starts at $(\varphi,v)$ and the map
\begin{equation}\label{eq:G-NS}
\G : [0,T)\times U\to X,\quad
(t;\varphi,v)\mapsto\G(t;\varphi,v),
\end{equation}
is $C^k$ for any $k\ge 1$ with
\[
X:=W^{m,p}_\delta\times W^{m,p}_\delta.
\] 
Recall that we identify $\D^{m,p}_\delta$ with an open set in $W^{m,p}_\delta$ 
(see the discussion after \eqref{eq:group_D}). We also set
\begin{align*}
&Z:=W^{m_0,p}_\delta\times W^{m_0,p}_\delta,\quad
D:=W^{m_0+4,p}_\delta\times W^{m_0+4,p}_\delta,\\
&\widetilde{X}:=W^{m-2,p}_\delta\times W^{m-2,p}_\delta,\quad
\widetilde{Z}:=W^{m_0-2,p}_\delta\times W^{m_0-2,p}_\delta\,.
\end{align*}
We will now apply Theorem \ref{th:product_formula} with the maps $S$ and $G$ replaced 
respectively by $\Sz$ and $\G$.
It follows from Proposition \ref{prop:T-independent_on_regularity} in Appendix \ref{sec:W-spaces}
that the restrictions 
\[
\G : [0,T)\times U_Z\to Z\quad\text{\rm and}\quad\G : [0,T)\times U_D\to D
\]
of the map \eqref{eq:G-NS} to $ U_Z= U\cap Z$ and $U_D= U\cap D$
are well-defined $C^k$ maps for any $k\ge 1$.
It follows from Theorem \ref{th:conjugate_Laplace_operator} that the map \eqref{eq:Xi-NS} generates
a nonlinear semigroup 
\begin{equation}\label{eq:S-NS}
\Sz : [0,\infty)\times U\to X,\quad(t;\varphi,v)\mapsto\Sz_t(\varphi,v)=
\Sz(t;\varphi,v):=\big(\varphi,S_\varphi(t)v\big)
\end{equation}
where 
\begin{equation}\label{eq:S(t)-conjugated'}
S_\varphi(t)= R_\varphi\circ S(t)\circ R_{\varphi^{-1}}
\end{equation}
and $\{S(t)\}_{t\ge 0}$ (cf.\ \cite[Theorem 2.2]{McOwenTopalov5}) is the analytic semigroup generated
by the Laplace operator $\Delta$ on $\R^d$ considered as an unbounded operator on $W^{m,p}_\delta$ with
domain $\widetilde{W}^{m+2,p}_\delta$  (cf. \eqref{eq:W-tilde}).
It follows from Lemma \ref{lem:continuity_composition_general} that the map \eqref{eq:S-NS} and its restrictions
\[
\Sz : [0,\infty)\times U_Z\to Z\quad\text{\rm and}\quad
\Sz : [0,\infty)\times U_D\to D
\]
are continuous. By \cite[Theorem 2.2]{McOwenTopalov5}, $\{S(t)\}_{t\ge 0}$ is
a strongly continuous semigroup on $W^{l,p}_\delta$ for any $l\ge 0$ with domain 
$D_{\Delta}=\widetilde{W}^{l+2,p}_\delta$. For the domain of the operator $\Delta^2$ on $W^{l,p}_\delta$
we have
\[
D_{\Delta^2}:=\big\{f\in D_\Delta\,\big|\,\Delta f\in D_\Delta\big\}=
\big\{f\in\Sz'\,\big|\,\partial^\alpha f\in W^{l,p}_\delta, |\alpha|\le 4\big\}.
\]
The spaces $D_\Delta$ and $D_{\Delta^2}$ are equipped with the graph norms.
Since the subspaces $W^{l+2,p}_\delta\subseteq D_\Delta$ and 
$W^{l+4,p}_\delta\subseteq D_{\Delta^2}$ are boundedly embedded, we conclude from
Lemma \ref{lem:S-smoothness} in Appendix \ref{sec:W-spaces} that the map
$S : [0,\infty)\times W^{l+2,p}_\delta\to W^{l,p}_\delta$ is $C^1$ along curves and that
\begin{equation}\label{eq:map0}
S : [0,\infty)\times W^{l+4,p}_\delta\to W^{l,p}_\delta,\quad l\ge 0,
\end{equation}
is $C^2$ along curves.
It follows from \eqref{eq:S-NS} and \eqref{eq:S(t)-conjugated'} that the map $\Sz : [0,\infty)\times U_Z\to X$ is 
a composition of the following $C^2$ maps along curves
\begin{equation}\label{eq:map1}
[0,\infty)\times U_Z\to[0,\infty)\times\D^{m_0-2,p}_\delta\times W^{m_0-2,p}_\delta,\quad
\big(t;\varphi,v)\mapsto(t;\varphi,\varphi^{-1}\circ v\big),
\end{equation}
\begin{equation}\label{eq:map2}
[0,\infty)\times\D^{m_0-2,p}_\delta\times W^{m_0-2,p}_\delta\to
[0,\infty)\times\D^{m+2,p}_\delta\times W^{m+2,p}_\delta,\quad
(t;\varphi,u)\mapsto\big(t;\varphi,S(t)u\big),
\end{equation}
and
\begin{equation}\label{eq:map3}
[0,\infty)\times\D^{m+2,p}_\delta\times W^{m+2,p}_\delta\to
[0,\infty)\times X,\quad(t;\varphi,u)\mapsto(t;\varphi,\varphi\circ v)\,.
\end{equation}
The maps \eqref{eq:map1} and \eqref{eq:map3} are $C^2$ by 
Theorem \ref{th:composition} in Appendix \ref{sec:W-spaces} and 
the map \eqref{eq:map2} is $C^2$ along curves since the map \eqref{eq:map0}
with $l=m+2$ is $C^2$ along curves and since $m_0-2=m+6$.
By using that the composition of $C^2$ maps along curves is again a $C^2$ map along curves
(cf. \cite[Theorem 3.5.5]{Hamilton}), we then conclude that
\[
\Sz : [0,\infty)\times U_Z\to X
\]
is $C^2$ along curves (cf. Definition \ref{def:C^k_along_curves} and Remark \ref{rem:C^k_along_curves}).
Similar arguments show that the map
\[
\Sz : [0,\infty)\times U_D\to Z
\]
is $C^1$ along curves. In particular, for any $(\varphi,v)\in U_Z$ we have that
\[
\frac{\dd}{\dd t}\Big|_{t=0}\Sz_t(\varphi,v)=\Xi(\varphi,v)
\]
where the maps $\Xi : U_Z\to\widetilde{Z}$ and $\Xi : U\to\widetilde{X}$ are $C^1$.
In this way, we see that all conditions on the maps in Proposition \ref{prop:product_formula}
except condition (c) are satisfied by the maps \eqref{eq:G-NS} and \eqref{eq:S-NS} respectively. 
Since by Theorem \ref{th:conjugate_Laplace_operator} the conditions $1^*)$ and $2^*)$ in
Theorem \ref{th:product_formula} also hold (with $N(\varphi_0,v_0):=V(\varphi_0)\times W^{m,p}_\delta$ where
$V(\varphi_0)$ is the open neighborhood of $\varphi_0$ in $\D^{m,p}_\delta$ given by 
Theorem \ref{th:conjugate_Laplace_operator} (iii) and $(\varphi_0,v_0)\in U$) and since 
by Lemma \ref{lem:L^p-norm} in Appendix \ref{sec:W-spaces} the norm in $X$ and $Z$ satisfy 
the condition \eqref{eq:differentiable_norm} with $\lambda=p$, we can apply 
Theorem \ref{th:product_formula} to conclude that for any 
$(\id,u_0)\in U_Z=\D^{m_0,p}_\delta\times W^{m_0,p}_\delta$, $\Div u_0=0$, there exist
$\tau>0$ and a unique solution
\begin{equation}\label{eq:NS-colution(LC)}
(\varphi,v)\in C_b\big([0,\tau),\D^{m,p}_\delta\times W^{m,p}_\delta\big)\cap
C^1_b\big([0,\tau),\D^{m-2,p}_\delta\times W^{m-2,p}_\delta\big)
\end{equation}
of the Navier-Stokes equation \eqref{eq:NS_Lagrange_coordinates} in Lagrangian coordinates.
The solution depends continuously on the initial data $u_0\in\aW^{m_0,p}_\delta$,
$(\varphi(t),v(t))\in\D^{m_0,p}_\delta\times W^{m_0,p}_\delta$ for $t\in[0,\tau)$, and the curve
$t\mapsto(\varphi(t),v(t))$, $[0,\tau)\to\D^{m_0,p}_\delta\times W^{m_0,p}_\delta$, is right continuous.
By setting $u(t):=v(t)\circ\varphi(t)^{-1}$ for $t\in[0,\tau)$ we then obtain from 
\eqref{eq:NS_Lagrange_coordinates}, Theorem \ref{th:composition},
and Lemma \ref{lem:continuity_composition_general}, that
\begin{equation}\label{eq:NS-solution}
u\in C_b\big([0,\tau),W^{m,p}_\delta\big)\cap C_b^1\big([0,\tau),W^{m-3,p}_\delta\big)
\end{equation}
is a solution of the equation
\begin{equation}\label{eq:NS-modified}
u_t+u\cdot\nabla u=\nabla\circ\Delta^{-1}\circ Q(u)+\Delta u,\quad u|_{t=0}=u_0,
\end{equation}
and it depends continuously on $u_0\in\aW^{m_0,p}_\delta$. 
It follows from \eqref{eq:NS-solution} and \eqref{eq:NS-modified} that 
\[
u(t)=u_0+\int_0^t\Big[-u\cdot\nabla u+\nabla\circ\Delta^{-1}\circ Q(u)+\Delta u\Big]\,ds
\]
where the integrand is in $C_b\big([0,\tau),W^{m-2,p}_\delta\big)$ since 
$u\in C_b\big([0,\tau),W^{m,p}_\delta\big)$ (cf. \cite[Proposition 4.1]{McOwenTopalov4}).
This implies that 
\begin{equation}\label{eq:NS-modified'}
u\in C_b\big([0,\tau),W^{m,p}_\delta\big)\cap  C^1_b\big([0,\tau),W^{m-2,p}_\delta\big)
\end{equation}
and it depends continuously on $u_0\in W^{m_0,p}_\delta$.
By taking the divergence in \eqref{eq:NS-modified} one obtains from
the equality $\Div\big(u\cdot\nabla u\big)=Q(u)+u\cdot\nabla(\Div u)$ that
$(\Div u)_t+u\cdot\nabla(\Div u)=0$ with $(\Div u)|_{t=0}=0$.
This implies that $(\Div u)(t,\varphi(t))$ is independent of $t\in[0,\tau)$, and hence
$u(t)$ is divergence free for any $t\in[0,\tau)$ (cf. the proof of \cite[Proposition 4.1]{McOwenTopalov5}).
Hence, \eqref{eq:NS-modified'} is a solution of the Navier-Stokes equation \eqref{eq:NS}.
Conversely, if \eqref{eq:NS-modified'} is a solution of the Navier-Stokes equation \eqref{eq:NS} such that
$|\nabla\p(t,x)|=o(1)$ as $|x|\to\infty$ for $t\in[0,\tau)$, then we consider the solution
$\varphi\in C^1_b\big([0,\tau),\D^{m,p}_\delta\big)$ of the equation 
$\dt{\varphi}(t)=u(t)\circ\varphi(t)$, $\varphi|_{t=0}=\id$ (see \cite[Proposition 2.1]{McOwenTopalov4}),  
and set $v(t):=u(t)\circ\varphi(t)$ for $t\in[0,\tau)$. A straightforward computation then shows that
$(\varphi(t),v(t))$ is a solution of \eqref{eq:NS_Lagrange_coordinates} such that 
\eqref{eq:NS-colution(LC)} holds. (Note that by Theorem \ref{th:composition} we have that
$(\varphi,v)\in C_b\big([0,\tau),\D^{m,p}_\delta\times W^{m,p}_\delta\big)$ but
$(\varphi,v)\in C_b^1\big([0,\tau),\D^{m-3,p}_\delta\times W^{m-3,p}_\delta\big)$.
The additional smoothness \eqref{eq:NS-colution(LC)} of this curve follows since it is a solution
of \eqref{eq:NS_Lagrange_coordinates}.) This and the uniqueness of the solutions in
Theorem \ref{th:product_formula} then imply that \eqref{eq:NS} has a unique solution
such that \eqref{eq:NS-modified'} holds. This completes the proof of Theorem \ref{th:NS}.
\end{proof}

\begin{proof}[Proof of Theorem \ref{th:NS_nu}]
We already verified that the maps \eqref{eq:G-NS} and \eqref{eq:S-NS} satisfy the conditions of 
Theorem \ref{th:product_formula}. Theorem \ref{th:NS_nu} then follows from Theorem \ref{th:product_formula}
and the arguments at the end of the proof of Theorem \ref{th:NS} needed for the proof that the solution \eqref{eq:NS-solution}
belongs to $C_b\big([0,\tau),\aW^{m,p}_\delta\big)\cap  C^1_b\big([0,\tau),\aW^{m-2,p}_\delta\big)$.
\end{proof}

\section{Proof of the product formula}\label{sec:Lie-Trotter_proofs}
In this Section we prove Proposition \ref{prop:product_formula}, Proposition \ref{prop:product_formula_nu}, 
and Theorem \ref{th:product_formula} stated in Section \ref{sec:Lie-Trotter}.  
We will first prove the following stronger variant of Lemma 1 in \cite[Appendix B]{EM}.

\begin{Lem}\label{lem:stability(T)}%T for Trotter
Let $(Z,\|\cdot\|_Z)$ and $(X,\|\cdot\|_X)$ be Banach spaces so that $Z$ is densely embedded in $X$, $Z\subseteq X$,
and let $U$ be a nonempty open set in $X$. Assume that the maps 
\begin{equation}\label{eq:G,S}
G : [0,T)\times U\to X\quad\text{\rm and}\quad S : [0,T)\times U\to X
\end{equation}
are continuous, satisfy the local semigroup property and the condition $1)$ in $(c)$, and the maps
$G : [0,T)\times U_Z\to X$ and $S : [0,T)\times U_Z\to X$ are $C^1$ along curves.
Then, for any $y_0\in U$ and for any open neighborhood $\widetilde{V}(y_0)\subseteq U$ of $y_0$ in $X$ there exist 
an open neighborhood $V= V(y_0)\subseteq\widetilde{V}(y_0)$ of $y_0$ in $X$ and $0<\tau=\tau(y_0)\le T$ 
(generally, depending on the choice of $y_0$ and $\widetilde{V}(y_0)$) such that for any $x\in V$ and $t\in[0,\tau]$ we have that
\begin{equation}\label{eq:stability(T)}
S_{t/n}^{\alpha_1}\circ G_{t/n}^{\mu_1}\circ\cdots\circ S_{t/n}^{\alpha_\ell}\circ 
G_{t/n}^{\mu_\ell}(x)\in\widetilde{V}(y_0)
\end{equation}
for any $n\ge 1$ and for any multi-indexes $\alpha,\mu\in\Z_{\ge 0}^\ell$ with $1\le\ell\le n$,
$|\alpha|\le n$, $|\mu|\le n$, and where we set $S_{t/n}^0(x)=x$ and $G_{t/n}^0(x)=x$ for $x\in U$.
\end{Lem}

\begin{Rem}\label{rem:stability(T)}
In particular, Lemma \ref{lem:stability(T)} implies that for any $x\in V$ and $t\in[0,\tau]$ we have that
$(S_{t/n}\circ G_{t/n})^n\in\widetilde{V}(y_0)$ for any $n\ge 1$.
\end{Rem}

\begin{proof}[Proof of Lemma \ref{lem:stability(T)}]
We first note that it is enough to prove the statement in the case when $n\ge 1$, $1\le\ell\le 2n$, 
and the components of the multi-indexes $\alpha,\mu\in\Z_{\ge 0}^\ell$ satisfy the condition
\begin{equation}\label{eq:reduced_condition}
\alpha_j,\mu_j\in\{0,1\}\,\,\,\text{\rm and}\,\,\,\alpha_j\ne\mu_j\,\,\,\text{\rm for}\,\,\, 1\le j\le\ell\,.
\end{equation}
(This follows since for any integer $1\le j\le n$ we have that
$S_{t/n}^j(x)=(S_{t/n}\circ G_{t/n}^0)\circ\cdots\circ(S_{t/n}\circ G_{t/n}^0)(x)$ and
$G_{t/n}^j(x)=(S_{t/n}^0\circ G_{t/n})\circ\cdots\circ(S_{t/n}^0\circ G_{t/n})(x)$ where
the compositions are taken $j$ times, $x\in U$, and $t\in[0,T)$.)
Let us now prove this equivalent statement. For a given $y_0\in U$  we will assume
without loss of generality that $\widetilde{V}(y_0)=U$. 
Since the maps in \eqref{eq:G,S} are continuous and $G_t(y_0)|_{t=0}=S_t(y_0)|_{t=0}=y_0$
we can choose an open neighborhood $W$ of $y_0$, $W\subseteq U$, and 
$0<\tau<T$ such that for any $x\in W$ and $t\in[0,\tau]$ we have that 
\begin{equation}\label{eq:G,S_in_U}
G_t(x), S_t(x)\in U.
\end{equation}
Let us now take $z_0\in W_Z:=W\cap Z$.
Since the maps $G : [0,T)\times U_Z\to X$ and $S : [0,T)\times U_Z\to X$ are $C^1$ along curves,
for any given $\lambda_1,\lambda_2\in\{0,1\}$, $\lambda_1\ne\lambda_2$, the curve $z : [0,\tau]\to X$, 
$z: t\mapsto S_{\lambda_1 t}\circ G_{\lambda_2 t}(z_0)$ is $C^1$.
In view of the compactness of the interval $[0,\tau]$ we then obtain that there exists $K>0$
(depending on the choice of $z_0\in U_Z$ but independent of the choice of $\lambda_1,\lambda_2\in\{0,1\}$,
$\lambda_1\ne\lambda_2$) such that 
\begin{equation}\label{eq:K}
\|z(t)-z_0\|=\big\|\int_0^t\frac{\dd z}{\dd t}(s)\,ds\big\|\le K t
\end{equation}
for any $t\in[0,\tau]$. This, together with \eqref{eq:equivalent_distances}, implies that for any $t\in[0,\tau]$,
\begin{equation}\label{eq:est1}
d\big(z_0,S_{\lambda_1 t}\circ G_{\lambda_2 t}(z_0)\big)\le K_1 t,
\end{equation}
with $K_1:=K/L_1$. Let $\delta>0$ be chosen so small that the ball $B^d_\delta(z_0):=\{x\in X\,|\,d(x,z_0)<\delta\}$ is 
contained in $W$. By taking $\tau>0$ smaller if necessary we ensure that
\begin{equation}\label{eq:est2}
2e^{2\beta\tau}K_1\tau<\delta\,.
\end{equation}
We now take $n\ge 1$, $\alpha,\mu\in\Z_{\ge 0}^\ell$, $1\le\ell\le 2n$, $|\alpha|\le n$, $|\mu|\le n$, 
such that the condition \eqref{eq:reduced_condition} holds.
We will first prove that for any $n\ge 1$, $1\le j\le\ell$, and $t\in[0,\tau]$, we have that
\begin{equation}\label{eq:est3}
d\big(z_0,\Pi_j(z_0)\big)<\delta
\end{equation}
where
\[
\Pi_j(z_0):=(S_{t/n}^{\alpha_1}\circ G_{t/n}^{\mu_1})\circ\cdots
\circ(S_{t/n}^{\alpha_j}\circ G_{t/n}^{\mu_j})(z_0)\in W
\] 
and $d$ is the distance function in condition 1) in (c). To this end, note that it follows from \eqref{eq:est1}
(applied with $\lambda_1=\alpha_j$, $\lambda_2=\mu_j$, and $t$ replaced by $t/n$) 
and \eqref{eq:est2} that for any $t\in[0,\tau]$,
\begin{equation}\label{eq:step1}
d\big(z_0,(S_{t/n}^{\alpha_j}\circ G_{t/n}^{\mu_j})(z_0)\big)\le K_1 t/n<\delta,
\quad 1\le j\le\ell,
\end{equation}
and hence $(S_{t/n}^{\alpha_j}\circ G_{t/n}^{\mu_j})(z_0)\in W$, $1\le j\le\ell$. 
This proves the statement when $\ell=1$ and when $j=1$ and $2\le\ell\le 2n$.
Let us now assume that $2\le\ell\le 2n$.
Then we obtain from \eqref{eq:G,S_in_U} that $\Pi_2(z_0)\in U$ for any $t\in[0,\tau]$, and we have that
\begin{align*}
d\big(z_0,\Pi_2(z_0)\big)&\le
d\big(z_0,(S_{t/n}^{\alpha_1}\circ G_{t/n}^{\mu_1})(z_0)\big)+
d\big((S_{t/n}^{\alpha_1}\circ G_{t/n}^{\mu_1})(z_0),\Pi_2(z_0)\big)\\
&\le d\big(z_0,(S_{t/n}^{\alpha_1}\circ G_{t/n}^{\mu_1})(z_0)\big)
+e^{\beta(\alpha_1+\mu_1)t/n}d\big(z_0,(S_{t/n}^{\alpha_2}\circ G_{t/n}^{\mu_2})(z_0)\big)\\
&\le 2e^{\beta\tau}K_1 t/n\le 2e^{2\beta\tau}K_1\tau<\delta
\end{align*}
where we used the property \eqref{eq:quasi-contraction}, \eqref{eq:step1}, and \eqref{eq:est2}.
In particular, we see that $\Pi_2(z_0)\in W$ for $t\in[0,\tau]$
(and hence $\Pi_3(z_0)\in U$).
By continuing this process inductively in $2\le j\le\ell$ (with $\ell\le 2n$ fixed), we obtain that for any $t\in[0,\tau]$,
\begin{align*}
d\big(z_0,\Pi_j(z_0)\big)&\le
\sum_{1\le k\le j}d\big(\Pi_{k-1}(z_0),\Pi_k(z_0)\big)\\
&\le\sum_{1\le k\le j}e^{\beta(k-1)t/n}d\big(z_0,(S_{t/n}^{\alpha_k}\circ G_{t/n}^{\nu_k})(z_0)\big)\\
&\le j e^{\beta(j-1)t/n} K_1 t/n\le 2e^{2\beta\tau}K_1\tau<\delta,
\end{align*}
where we used \eqref{eq:step1}, the fact that $\alpha_q,\mu_q\in\{0,1\}$, $\alpha_q\ne\mu_q$,
$1\le q\le\ell$, and that $j\le\ell\le 2n$. Hence $\Pi_j(z_0)\in W$ for $t\in[0,\tau]$.
This proves the preliminary statement \eqref{eq:est3}.
We now take $\epsilon>0$ so that $B^d_{3\epsilon}(y_0)\subseteq W$ and choose
$z_0\in B^d_\epsilon(y_0)\cap Z$. By the discussion above, there exists $\tau>0$ (depending on the choice of $z_0$)
so that \eqref{eq:est3} (with $\delta=\epsilon$) holds for any $n\ge 1$, $1\le\ell\le 2n$, $1\le j\le\ell$,
and $t\in[0,\tau]$. Let us now choose $\kappa>0$ so that
\begin{equation}\label{eq:est4}
e^{2\beta\tau}\kappa<\epsilon\,.
\end{equation}
Then, for any $x\in B^d_\kappa(y_0)$ and for any $n\ge 1$ we argue inductively in $1\le j\le\ell$
to conclude that 
\[
\Pi_j(x):=(S_{t/n}^{\alpha_1}\circ G_{t/n}^{\mu_1})\circ\cdots
\circ(S_{t/n}^{\alpha_j}\circ G_{t/n}^{\mu_j})(x)\in U
\]
and that for $t\in[0,\tau]$,
\begin{align*}
d\big(y_0,\Pi_j(x)\big)&\le
d(y_0,z_0)+d\big(z_0,\Pi_j(z_0)\big)+d\big(\Pi_j(z_0),\Pi_j(x)\big)\\
&\le 2\epsilon+e^{\beta jt/n}\kappa<3\epsilon
\end{align*}
where we used \eqref{eq:quasi-contraction}, \eqref{eq:est3} (with $\delta=\epsilon$), and \eqref{eq:est4}.
In particular, $\Pi_j(x)\in W$ (and hence $\Pi_{j+1}\in U$).
We complete the proof by setting $V:=B^d_\kappa(y_0)$.
\end{proof}

\begin{Rem}\label{rem:nu1}
The proof of Lemma \ref{lem:stability(T)} shows that under the assumption of Lemma \ref{lem:stability(T)}
there exist an open neighborhood $V=V(y_0)\subseteq\widetilde{V}(y_0)$ of $y_0$ in $X$ and 
$0<\tau\le T$ as above such that \eqref{eq:stability(T)} holds with $S(t)$ replaced by $S^{(\nu)}_t=S_{\nu t}$ 
uniformly in $0\le\nu\le 1$. This follows since the second estimate in \eqref{eq:quasi-contraction}
with $S_t$ replaced by $S^{(\nu)}_t$ holds uniformly in $0\le\nu\le 1$ and since the constant $K>0$ appearing
in \eqref{eq:K} can be chosen so that the estimate $\max\limits_{t\in[0,\tau]}\big\|\frac{\dd z_\nu}{\dd t}(t)\big\|\le K$
where $z_\nu(t):=S^{(\nu)}_{\lambda_1 t}\circ G_{\lambda_2 t}$ holds uniformly $0\le\nu\le 1$.
\end{Rem}

Let us now prove Proposition \ref{prop:product_formula}.

\begin{proof}[Proof of Proposition \ref{prop:product_formula}]
Assume that the conditions of Proposition \ref{prop:product_formula} are satisfied.
We will assume without loss of generality that $S : [0,T)\times U\to X$ and $G : [0,T)\times U\to X$
are globally equivalent to a quasi-contraction with a distance function $d : U\times U\to\R_{\ge 0}$ and
a constant $\beta>0$ (cf. \eqref{eq:quasi-contraction}).
It follows from the continuity of the maps $S : [0,T)\times U\to X$ and $G : [0,T)\times U\to X$ that
for a given $y_0\in U$ there exist $V= V(y_0)\subseteq U$ and $\tau>0$  such that 
for any $x\in V$ and $t\in[0,\tau]$ the compositions $S_t\circ G_t(x)$ and $G_t\circ S_t(x)$ are 
well defined and belong to $U$. We have the following

\begin{Lem}\label{lem:commutator(T)}
For any $z_0\in V_Z:= V\cap Z$ there exist an open neighborhood $\widetilde{W}(z_0)\subseteq V_Z$ 
of $z_0$ in $Z$ and a constant $K>0$ such that
\begin{equation}\label{eq:commutator(T)}
d\big(S_t\circ G_t(z),G_t\circ S_t(z)\big)\le K t^2
\end{equation}
for any $z\in\widetilde{W}(z_0)$ and for any $t\in[0,\tau]$.
\end{Lem}

\begin{proof}[Proof of Lemma \ref{lem:commutator(T)}]
Take $z_0\in V_Z\subseteq V$. It follows from condition $(a)$ that the map 
$R(t):=S_t\circ G_t(z_0)-G_t\circ S_t(z_0)$, $[0,\tau]\to U\subseteq X$,
is $C^2$ along curves (cf. Remark \ref{rem:C^k_along_curves}). 
In view of Taylor's formula (with reminder in integral form) we have
\[
R(t)=R(0)+\frac{\dd R}{\dd t}(0)\,t+\Big(\int_0^1(1-s)\frac{\dd^2R}{\dd t^2}(st)\,ds\Big)\,t^2
\]
for any $t\in[0,\tau]$. In view of the compactness of the interval $[0,\tau]$ and the fact that
$R(0)=0$ and, by the chain rule (cf. Remark \ref{rem:C^k_along_curves}), we have that 
$\frac{\dd R}{\dd t}(0)=(\Xi+\E)-(\E+\Xi)=0$, 
we conclude that
\[
\|R(t)\|\le M t^2/2
\]
where $M:=\max\limits_{0\le s\le\tau]}\big\|\frac{\dd^2R}{\dd t^2}(s)\big\|$.
The statement of the Lemma (with $K:=M/(2L_1)$) then follows from 
the property \eqref{eq:equivalent_distances}.
The local uniformity of $K>0$ in $U_Z$ follows easily from condition $(a)$ and the compactness of $[0,\tau]$
(cf. Remark \ref{rem:C^k_along_curves}).
\end{proof}

\begin{Rem}\label{rem:nu2}
Note that the open neighborhood $V=V(y_0)\subseteq U$ of $y_0$ in $X$ and $\tau>0$ can be chosen
so that  for any $x\in V$ and $t\in[0,\tau]$ the compositions $S_t^{(\nu)}\circ G_t(x)$ and 
$G_t\circ S_t^{(\nu)}(x)$ are 
well defined and belong to $U$ for any $0\le\nu\le 1$.
Then, Lemma \ref{lem:commutator(T)} holds with $S_t$ replaced by $S^{(\nu)}_t$ uniformly 
in $0\le\nu\le 1$. This follows since the second estimate in \eqref{eq:quasi-contraction}
with $S_t$ replaced by $S^{(\nu)}_t$ holds uniformly in $0\le\nu\le 1$ and since the constant $M>0$
can be chosen so that the estimate $\max\limits_{t\in[0,\tau]}\|R_\nu(t)\|\le Mt^2$ where 
$R_\nu(t):=S^{(\nu)}_t\circ G_t(z_0)-G_t\circ S_t^{(\nu)}(z_0)$ holds uniformly in $0\le\nu\le 1$.
\end{Rem}

Note that $\tau>0$ in Lemma \ref{lem:commutator(T)} is independent of the choice of $z_0\in V_Z$.

\begin{Lem}\label{lem:fundamental(T)}
Take $z_0\in V_Z= V\cap Z$ and let $K>0$, $\tau>0$, and the open neighborhood $\widetilde{W}(z_0)\subseteq V_Z$
of $z_0$ in $Z$ be given by Lemma \ref{lem:commutator(T)}.
Then, there exist an open neighborhood $W(z_0)\subseteq\widetilde{W}(z_0)$ of $z_0$ in $Z$
and $0<\tau_1=\tau_1(z_0)\le\tau$ (generally, depending on the choice of $z_0$ and $\widetilde{W}(z_0)$) 
such that\footnote{In fact, we prove that the stronger condition \eqref{eq:in_the_domain_general(fundamental)} below holds.}
\begin{equation}\label{eq:in_the_domain(fundamental)}
(S_{t/n}\circ G_{t/n})^n(z)\in\widetilde{W}(z_0)\quad\forall n\ge 1
\end{equation}
and 
\begin{equation}\label{eq:fundamental(T)}
d\big((S_{t/n}\circ G_{t/n})^n(z),(S_{t/m}\circ G_{t/m})^m(z)\big)\le 2 K e^{2\beta t}t^2/n
\end{equation}
for any $1\le n\le m$, $z\in W(z_0)$, and $t\in[0,\tau_1]$.
\end{Lem}

\begin{proof}[Proof of Lemma \ref{lem:fundamental(T)}]
Take $z_0\in V_Z= V\cap Z$ and let $K>0$, $\tau>0$, and the open neighborhood $\widetilde{W}(z_0)$
of $z_0$ in $Z$ be given by Lemma \ref{lem:commutator(T)}.
It then follows from condition $(b)$, the continuity of the second map in \eqref{eq:G} and \eqref{eq:S}, 
and Lemma \ref{lem:stability(T)}, that we can choose an open neighborhood $W(z_0)\subseteq\widetilde{W}(z_0)$ of 
$z_0$ in $Z$ and $0<\tau_1\le\tau$ such that \eqref{eq:in_the_domain(fundamental)} holds for any $z\in W(z_0)$, 
$t\in[0,\tau_1]$, and for any $n\ge 1$. (We apply Lemma \ref{lem:stability(T)} in the $Z$ space with $y_0= z_0$, 
$U= V_Z$ (see condition $2)$ in $(c)$), and $\widetilde{V}(y_0)=\widetilde{W}(z_0)$.)
More generally, by Lemma \ref{lem:stability(T)}, we can choose $W(z_0)$ so that for any $z\in W(z_0)$
and $t\in[0,\tau_1]$ we have that
\begin{equation}\label{eq:in_the_domain_general(fundamental)}
S_{t/n}^{\alpha_1}\circ G_{t/n}^{\mu_1}\circ\cdots\circ S_{t/n}^{\alpha_\ell}\circ G_{t/n}^{\mu_\ell}(z)
\in\widetilde{W}(z_0)\subseteq V_Z
\end{equation}
for any $n\ge 1$ and for any $\alpha,\mu\in\Z_{\ge 0}^\ell$, $1\le\ell\le n$, such that $|\alpha|\le n$ and $|\mu|\le n$. 
Let us now prove \eqref{eq:fundamental(T)}. To this end, we will first assume that $m=k n$ for 
some integer $k\ge 1$. Then, for any $z\in W(z_0)$,
\[
(S_{t/m}\circ G_{t/m})^m(z)=\big[(S_{t/(kn)}\circ G_{t/(kn)})^k\big]^n(z)
\]
and, by the semigroup property,
\[
(S_{t/n}\circ G_{t/n})^n(z)=\big[S_{t/(kn)}^k\circ G_{t/(kn)}^k\big]^n(z)\,.
\]
Consider for simplicity the case when $k=3$. Then, the term 
\[
(S_{t/(3n)}\circ G_{t/(3n)})^3(z)=
(S_{t/(3n)}\circ G_{t/(3n)})\circ(S_{t/(3n)}\circ G_{t/(3n)})\circ(S_{t/(3n)}\circ G_{t/(3n)})(z)
\]
can be transformed into $S_{t/(3n)}^3\circ G_{t/(3n)}^3(z)$ by the sequence of terms
\begin{align*}
&A_1:=S_{t/(3n)}\circ G_{t/(3n)}\circ S_{t/(3n)}\circ(G_{t/(3n)}\circ S_{t/(3n)})\circ G_{t/(3n)}(z)\\
&A_2:=S_{t/(3n)}\circ(G_{t/(3n)}\circ S_{t/(3n)})\circ S_{t/(3n)}\circ G_{t/(3n)}\circ G_{t/(3n)}(z)\\
&A_3:=S_{t/(3n)}\circ S_{t/(3n)}\circ(G_{t/(3n)}\circ S_{t/(3n)})\circ G_{t/(3n)}\circ G_{t/(3n)}(z)\\
&A_4:=S_{t/(3n)}\circ S_{t/(3n)}\circ S_{t/(3n)}\circ G_{t/(3n)}\circ G_{t/(3n)}\circ G_{t/(3n)}(z)
\end{align*}
with the property that each term in the sequence is obtained from the previous one by a binary commutation of 
the type described above. In the general case, we obtain in the same way a sequence of terms 
\[
A_1,...,A_{\ell_k}
\] 
such that each term is obtained from the previous one by a binary commutation,
\[
A_1:=(S_{t/(kn)}\circ G_{t/(kn)})^k(z),\quad A_{\ell_k}:=S_{t/(kn)}^k\circ G_{t/(kn)}^k(z),
\]
and $\ell_k\le (k-1)+(k-2)+...+1=k(k-1)/2\le k^2$. Then, by the triangle inequality
\begin{align}
d\big((S_{t/n}\circ G_{t/n})^n(z),(S_{t/m}\circ G_{t/m})^m(z)\big)&=
d\Big(\big[S_{t/(kn)}^k\circ G_{t/(kn)}^k\big]^n(z),\big[(S_{t/(kn)}\circ G_{t/(kn)})^k\big]^n(z)\Big)\nonumber\\
&\le\sum_{j=0}^{n-1}\sum_{l=1}^{\ell_k-1}d\big(B_{j,l}(z),B_{j,l+1}(z)\big)\label{eq:est5}
\end{align}
where
\begin{equation}\label{eq:B_{j,l}}
B_{j,l}(z):=\big[(S_{t/(kn)}\circ G_{t/(kn)})^k\big]^{n-j-1}\circ A_l\circ\big[S_{t/(kn)}^k\circ G_{t/(kn)}^k\big]^j(z)\,.
\end{equation}
Note that the terms appearing in \eqref{eq:B_{j,l}} are of the form \eqref{eq:in_the_domain_general(fundamental)},
and hence belong to $U_Z\subseteq U$ where the distance $d$ is defined.
Moreover, since the terms in \eqref{eq:in_the_domain_general(fundamental)} belong to $\widetilde{W}(z_0)$,
the estimate \eqref{eq:commutator(T)} can be applied for $z$ replaced by any of the terms appearing in 
\eqref{eq:in_the_domain_general(fundamental)}.
In particular, we obtain from \eqref{eq:quasi-contraction}, \eqref{eq:commutator(T)}, \eqref{eq:B_{j,l}}, 
and the definition of the terms $A_1,...,A_{\ell_k}$, that
\begin{align}\label{eq:d(B,B)}
d\big(B_{j,l}(z),B_{j,l+1}(z)\big)&\le e^{2\beta t(n-1)/n}e^{2\beta t(k-1)/(kn)} K(t/{nk})^2\nonumber\\
&\le K e^{2\beta t}(t/{nk})^2\,.
\end{align}
This, together with \eqref{eq:est5}, then implies that for $m=kn$,
\begin{equation}\label{eq:crucial}
d\big((S_{t/m}\circ G_{t/m})^m(z),(S_{t/n}\circ G_{t/n})^n(z)\big)\le (n k^2) K e^{2\beta t}(t/{nk})^2
\le K e^{2\beta t}t^2/n
\end{equation}
where we also used that $\ell_k\le k^2$. For $1\le n\le m$ arbitrary we then have
\begin{align*}
&d\big((S_{t/n}\circ G_{t/n})^n(z),(S_{t/m}\circ G_{t/m})^m(z)\big)\\
&\le d\big((S_{t/n}\circ G_{t/n})^n(z),(S_{t/m}\circ G_{t/(nm)})^{nm}(z)\big)+
d\big((S_{t/m}\circ G_{t/(nm)})^{nm}(z),(S_{t/m}\circ G_{t/m})^m(z)\big)\\
&\le K e^{2\beta t}t^2/n+K e^{2\beta t}t^2/m\le 2K e^{2\beta t}t^2/n\,.
\end{align*}
Hence, the inequality \eqref{eq:fundamental(T)} holds for any $1\le n\le m$,
$z\in W(z_0)$, and $t\in[0,\tau_1]$.
\end{proof}

\begin{Rem}\label{rem:EM-paper}
Note that, contrary to the statement before Lemma 4 in \cite[Appendix B]{EM},
the parameter $\tau_1>0$ in Lemma \ref{lem:fundamental(T)} (and Lemma 2 in \cite{EM})
cannot be (generally) chosen equal to $\tau>0$. The reason is that in the proof of \eqref{eq:d(B,B)}
the estimate \eqref{eq:commutator(T)} is applied not only to $z$, but also to terms of the form
\eqref{eq:in_the_domain_general(fundamental)}. One needs an additional assumption to ensure
that the constant $K>0$ in \eqref{eq:commutator(T)} can be taken to be uniform over terms of the form 
\eqref{eq:in_the_domain_general(fundamental)} for any $z\in W(z_0)$ and $t\in[0,\tau]$. 
As a consequence, the construction of $H_t(x)$ for $x\in U$ and $t\in[0,\tau]$ cannot be completed
as stated in \cite[Appendix B]{EM}. The proof of the analogous \cite[Lemma 2.5]{Marsden} uses
the restrictive condition $(iii)$ (see the proof of \cite[Corolarry 2.4]{Marsden}) that we do not assume. 
In what follows, we deal with the issue in an alternative way.
\end{Rem}

\begin{Rem}\label{rem:nu3}
One verifies by inspection that Lemma \ref{lem:fundamental(T)} holds with 
$S_t$ replaced by $S^{(\nu)}_t$ uniformly in $0\le\nu\le 1$. In particular,
the open neighborhood $W(z_0)$ of $z_0$ in $Z$ and $\tau_1>0$ can be chosen 
independently of $0\le\nu\le 1$, so that the inequalities \eqref{eq:fundamental(T)} and
\eqref{eq:in_the_domain_general(fundamental)} with $S_t$ replaced by $S^{(\nu)}_t$ 
hold uniformly in $0\le\nu\le 1$.
\end{Rem}

By Lemma \ref{lem:fundamental(T)}, for a given $z_0\in V_Z$ we choose
$\widetilde{W}(z_0)\subseteq V_Z$, $W(z_0)\subseteq\widetilde{W}(z_0)$, and $\tau:=\tau_1>0$,
so that \eqref{eq:in_the_domain(fundamental)} and \eqref{eq:fundamental(T)} hold
for any $0\le n\le m$, $z\in W(z_0)$, and $t\in[0,\tau]$. It then follows from \eqref{eq:fundamental(T)} that
the sequence $\big((S_{t/n}\circ G_{t/n})^n(z)\big)_{n\ge 1}$ converges in $X$ to an element $H_t(z)\in X$,
\begin{equation}\label{eq:H-limit}
(S_{t/n}\circ G_{t/n})^n(z)\stackrel{X}{\to} H_t(z)\quad\text{\rm and}\quad H_t(z)\in X
\end{equation}
for any $z\in W(z_0)$ and $t\in[0,\tau]$.
We will assume without loss of generality that $\widetilde{W}(z_0)= B_\rho^Z(z_0)$
where $B_\rho^Z(z_0)$ denotes the open ball of radius $\rho>0$ in $Z$ centered at $z_0$ and where
$\rho>0$ is chosen so that the closure $\overline{B}_\rho^Z(z_0)$ of $B_\rho^Z(z_0)$ in $Z$ is contained in $V_Z$,
i.e., $\overline{B}_\rho^Z(z_0)\subseteq V_Z$. This implies that
\begin{equation}\label{eq:bouded_sequence}
\big((S_{t/n}\circ G_{t/n})^n(z)\big)_{n\ge 1}\subseteq\widetilde{W}(z_0)=\overline{B}_\rho^Z(z_0)\subseteq U_Z\,.
\end{equation}
By the reflexivity of $Z$ and Alaoglu theorem, there exists a subsequence of the sequence above that converges weakly in $Z$ to 
an element in $\tilde{z}\in\overline{B}_\rho^Z(z_0)$.
In view of \eqref{eq:H-limit}, the subsequence converges weakly in $X$ to $H_t(z)\in X$.
Since $X^*\subseteq Z^*$, we then obtain that $f(\tilde{z})=f(H_t(z))$ for any bounded linear functional $f\in X^*$.
By the Hahn-Banach theorem, $\tilde{z}=H_t(z)$, and hence
\begin{equation}\label{eq:H-bounded_in_Z}
H_t(z)\in\overline{B}_\rho^Z(z_0)\subseteq U_Z
\end{equation}
for any $z\in W(z_0)$ and $t\in[0,\tau]$. 
By passing to the limit as $m\to\infty$ in \eqref{eq:fundamental(T)} we obtain that for any $n\ge 1$,
\begin{equation}\label{eq:rate_of _convergence}
d\big((S_{t/n}\circ G_{t/n})^n(z),H_t(z)\big)\le 2 K e^{2\beta t}t^2/n
\end{equation}
for any $z\in W(z_0)$ and $t\in[0,\tau]$. Hence, the convergence in \eqref{eq:H-limit} is {\em uniform} in
$z\in W(z_0)$ and $t\in[0,\tau]$, and the rate of convergence is of order $O(1/n)$. In particular,
we see that the map 
\begin{equation}\label{eq:H-mapX}
H : [0,\tau]\times\big(W(z_0),\|\cdot\|\big)\to X,
\end{equation}
where $\big(W(z_0),\|\cdot\|\big)$ denotes the set $W(z_0)\subseteq Z$ equipped with the norm 
$\|\cdot\|$, is continuous. Clearly, $H_t(z)|_{t=0}=z$ for any $z\in W(z_0)$. 
By summarizing the above, we obtain the following Corollary.

\begin{Coro}\label{coro:fundamental(T)}
For any $z_0\in U_Z$ there exists $\rho_0>0$ such that for any $0<\rho<\rho_0$ there exist 
an open neighborhood $W(z_0)\subseteq B_\rho^Z(z_0)\subseteq U_Z$ of $z_0$ in $Z$ and 
$\tau=\tau(z_0)>0$ such that for any $z\in W(z_0)$ and $t\in[0,\tau]$ we have that 
\eqref{eq:bouded_sequence} holds and the limit \eqref{eq:H-limit} exists and satisfies 
\eqref{eq:H-bounded_in_Z}, \eqref{eq:rate_of _convergence}, and \eqref{eq:H-mapX}. 
In particular, for any $z\in W(z_0)$ the curve $[0,\tau]\to Z$, $t\mapsto H_t(z)$,
is continuous at $t=0$.
\end{Coro}

We have the following global Lipschitz property.

\begin{Lem}\label{lem:H-Lipschitz}
For any $z_1,z_2\in U_Z$ take $\tau_0:=\min\big(\tau(z_1),\tau(z_2)\big)>0$
where $\tau(z_1)>0$ and $\tau(z_2)>0$ are given by Corollary \ref{coro:fundamental(T)}
so that $H_t(z_1), H_t(z_2)\in U_Z$ are defined for $t\in[0,\tau_0]$. Then, we have that
\begin{equation}\label{eq:H-Lipschitz}
d\big(H_t(z_1),H_t(z_2)\big)\le e^{2\beta t} d(z_1,z_2)
\end{equation}
for any $t\in[0,\tau_0]$.
\end{Lem}

\begin{proof}[Proof of Lemma \ref{lem:H-Lipschitz}]
Take $z_1,z_2\in U_Z$. We then take $\tau(z_1)>0$ and $\tau(z_2)>0$ given by Corollary \ref{coro:fundamental(T)} and
set $\tau_0:=\min\big(\tau(z_1),\tau(z_2)\big)>0$. Then, $H_t(z_1), H_t(z_2)\in U_Z$ are defined for $t\in[0,\tau_0]$ 
and we have that
\[
(S_{t/n}\circ G_{t/n})^n(z_1),(S_{t/n}\circ G_{t/n})^n(z_2)\in U_Z
\]
and
\[
(S_{t/n}\circ G_{t/n})^n(z_1)\stackrel{X}{\to} H_t(z_1)\in U_Z\quad\text{\rm and}\quad
(S_{t/n}\circ G_{t/n})^n(z_2)\stackrel{X}{\to} H_t(z_2)\in U_Z
\]
as $n\to\infty$. By passing to the limit as $n\to\infty$ in the inequality (cf. \eqref{eq:quasi-contraction})
\[
d\big((S_{t/n}\circ G_{t/n})^n(z_1),(S_{t/n}\circ G_{t/n})^n(z_1)\big)\le e^{2\beta t} d(z_1,z_2)
\]
we obtain \eqref{eq:H-Lipschitz}.
\end{proof}

Let us now prove the following Lemma.

\begin{Lem}\label{lem:H-semigroup}
Take $z_0\in U_Z$ and let $W(z_0)\subseteq U_Z$ and $\tau=\tau(z_0)>0$ be chosen as in 
Corollary \ref{coro:fundamental(T)} for some $0<\rho<\rho_0$. 
Then, there exist an open neighborhood $O(z_0)\subseteq W(z_0)$ of $z_0$ in $Z$ and $0<\tau_1\le\tau$ such that 
for any $z\in O(z_0)$ and for any given $s,t\in[0,\tau_1]$ the quantities $H_{s+t}(z)$ and $H_s\circ H_t(z)$ 
are defined in $U_Z$ and $H_{t+s}(z)=H_s\circ H_t(z)$.
\end{Lem}

By combining this Lemma with the last statement in Corollary \ref{coro:fundamental(T)}, we obtain 

\begin{Coro}\label{coro:H-right_continuous_in_Z}
With the notations in Lemma \ref{lem:H-semigroup} we have that for any $z\in O(z_0)$ the curve
$t\mapsto H_t(z)$, $[0,\tau_1)\to Z$, is right continuous. 
\end{Coro}

\begin{proof}[Proof of Lemma \ref{lem:H-semigroup}]
We modify the proof of \cite[Lemma 2.7]{Marsden}. 
Take $z_0\in U_Z$ and let $W(z_0)$ and $\tau>0$ be chosen as in Corollary \ref{coro:fundamental(T)}
for some $0<\rho<\rho_0$. For $z\in W(z_0)$ and $t\in[0,\tau)$ we set $P_t(z):=S_t\circ G_t(z)$.
By applying Corollary \ref{coro:fundamental(T)} one more time with $0<\rho<\rho_0$ taken smaller so that
$\overline{B}_\rho^Z(z_0)\subseteq W(z_0)$, we obtain an open neighborhood $O(z_0)\subseteq W(z_0)$
of $z_0$ in $Z$ and $0<\tau_1\le \tau$ such that
\begin{equation}\label{eq:in_the_domain}
H_t(z)\in W(z_0)\quad\text{\rm and}\quad(P_{t/j})^j(z)\in W(z_0)\quad\forall j\in\N
\end{equation}
for any $z\in O(z_0)$ and for any $t\in[0,\tau_1]$. By replacing $\tau_1$ by $\tau_1/2$ we obtain that
$H_{s+t}(z)$ is well defined for any $s,t\in[0,\tau_1]$. Hence, by construction, the compositions
$H_s\circ H_t(z)$ and $H_{s+t}(z)$ are defined for any $z\in O(z_0)$ 
and $s,t\in[0,\tau_1]$. We are now ready to prove the claimed property.
Let us first assume that $t$ and $s$ are rationally dependent, i.e. $t=(p/q) s$ where $p,q\in\N$.
Then, for $n\ge 1,$
\[
\frac{s+t}{n}=\frac{s}{n_1}=\frac{t}{n_2}
\]
where $n_1:=\frac{nq}{p+q}$ and $n_2:=\frac{np}{p+q}$. Since $n_1+n_2=n$ we have
\begin{equation}\label{eq:eqlt1}
(P_{\frac{s+t}{n}})^n(z)=(P_{\frac{s+t}{n}})^{n_1}\circ(P_{\frac{s+t}{n}})^{n_2}(z)=
(P_{s/n_1})^{n_1}\circ(P_{t/n_2})^{n_2}(z)\,.
\end{equation}
In what follows we take $n$ running successively through the infinite set $A_{p,q}:=\big\{(p+q)l\,\big|\,l\in\N\big\}$
of integers. Then,  $n_1=ql$ and $n_2=pl$ (for $l=1,2,...$) are integers and $n_1,n_2\to\infty$ when 
$n\to\infty$, $n\in A_{p,q}$. For any $z\in O(z_0)$, $t\in[0,\tau_1]$, and for any $n\in A_{p,q}$, $n_1$, 
and $n_2$ as above, we have
\begin{align*}
&d\big((P_{s/n_1})^{n_1}\circ(P_{t/n_2})^{n_2}(z),H_s\circ H_t(z)\big)\le
d\big((P_{s/n_1})^{n_1}\circ(P_{t/n_2})^{n_2}(z),H_s\circ(P_{t/n_2})^{n_2}(z)\big)\\
&+d\big(H_s\circ(P_{t/n_2})^{n_2}(z),H_s\circ H_t(z)\big)\le\frac{C_1}{n_1}+
e^{2\beta\tau_1}\,d\big((P_{t/n_2})^{n_2}(z),H_t(z)\big)=C_2\Big(\frac{1}{n_1}+\frac{1}{n_2}\Big)
\end{align*}
where we used \eqref{eq:rate_of _convergence}, \eqref{eq:H-Lipschitz}, and \eqref{eq:in_the_domain}, 
and the constants are independent of the choice of $z\in O(z_0)$ and $t\in[0,\tau_1]$.
Hence, $(P_{s/n_1})^{n_1}\circ(P_{t/n_2})^{n_2}(z)$ converges to $H_s\circ H_t(z)$ in $X$ as 
$n\to\infty$ with $n\in A_{p,q}$. By passing to the limit as $n\to\infty$ with $n\in A_{p,q}$
in \eqref{eq:eqlt1} we then conclude from \eqref{eq:H-limit} (with $t$ replaced by $s+t$) that 
$H_{s+t}(z)=H_s\circ H_t(z)$ in the case when $s$ and $t$ are rationally dependent.
Since the rational numbers are dense in $\R$, the general case follows from the continuity
of the map \eqref{eq:H-mapX} and the Lipschitz estimate \eqref{eq:H-Lipschitz}.
This completes the proof of the Lemma.
\end{proof}

\begin{Lem}\label{lem:H-solution}
Take $z_0\in U_Z$ and let $O(z_0)\subseteq U_Z$ and $\tau_1>0$ be chosen as in Lemma \ref{lem:H-semigroup}
for some $0<\rho<\rho_0$. Then, for any given $z\in O(z_0)$ the curve 
$[0,\tau_1)\to\widetilde{X}$, $t\mapsto H_t(z)$, is differentiable and
\begin{equation}\label{eq:H-solution}
\frac{\dd}{\dd t}\big(H_t(z)\big)=\E(H_t(z))+\Xi(H_t(z)),\quad H_t(z)\big|_{t=0}=z .
\end{equation}
\end{Lem}

\begin{proof}[Proof of Lemma \ref{lem:H-solution}]
Take $z_0\in U_Z$ and let $O(z_0)\subseteq W(z_0)$ and $0<\tau_1<\tau$ be chosen as in Lemma \ref{lem:H-semigroup}.
Take $z\in O(z_0)$. We will first prove that the curve 
\begin{equation}\label{eq:H-curve}
[0,\tau_1)\to X,\quad t\mapsto H_t(z),
\end{equation}
is right differentiable at $t\in[0,\tau_1)$ and that its right derivative satisfies \eqref{eq:H-solution}.
In view of Lemma \ref{lem:H-semigroup}, it is enough to prove this only at $t=0$.
To this end, note that by \eqref{eq:rate_of _convergence} (with $n=0$) we have that
\begin{equation}\label{eq:eqlt2}
d\big(S_t\circ G_t(z),H_t(z)\big)=O(t^2)
\end{equation}
uniformly for $t\in[0,\tau_1]$.
It follows from condition $(a)$ that $t\mapsto S_t\circ G_t(z)$, $[0,\tau_1)\to X$, is differentiable and
\begin{equation}\label{eq:eqlt3}
\frac{\dd}{\dd t}\Big|_{t=0}\big(S_t\circ G_t(z)\big)=\Xi(z)+\E(z)
\end{equation}
(cf.\ Remark \ref{rem:C^k_along_curves}).
Since the distance function $d$ on $U$ is equivalent to the standard distance function on $X$,
we conclude from \eqref{eq:eqlt2} and \eqref{eq:eqlt3} that \eqref{eq:H-curve} is
right differentiable at $t=0$. Hence, by the observation above, \eqref{eq:H-curve} is
right differentiable for $t\in[0,\tau_1)$ and
\begin{equation}\label{eq:eqlt4}
\partial_t^+\big(H_t(z)\big)=\E(H_t(z))+\Xi(H_t(z))
\end{equation}
where $\partial_t^+$ denotes the right derivative of \eqref{eq:H-curve}.
Since \eqref{eq:H-curve} is continuous, we then conclude from \eqref{eq:eqlt4} that
\begin{equation}\label{eq:H-curve_in_Y}
[0,\tau_1)\to\widetilde{X},\quad t\mapsto H_t(z),
\end{equation}
is right differentiable with continuous right derivative on $[0,\tau_1)$.
A classical theorem (see, e.g., the Lemma at the end of the proof of Theorem 2, 
$\S\,3$, Chapter IX  in \cite{Yosida}) then implies that \eqref{eq:H-curve_in_Y} is
differentiable on $[0,\tau_1)$ and that \eqref{eq:H-solution} holds in $\widetilde{X}$.
\end{proof}

Let $O(z_0)\subseteq U_Z$ and $\tau_1>0$ be chosen as in Lemma \ref{lem:H-semigroup}.
Then, for any given $z\in O(z_0)$ the curve \eqref{eq:H-curve} is continuous.
Since $\E : U\to X$ and $\Xi : U\to\widetilde{X}$ (extended from $U_Z$ to $U$ by continuity) are $C^1$, 
we conclude that $t\mapsto \E(H_t(z))+\Xi(H_t(z))$, $[0,\tau_1]\to\widetilde{X}$, is continuous.
By combining this with \eqref{eq:H-solution}, we conclude that the curve \eqref{eq:H-curve_in_Y}
is $C^1$. Hence,
\begin{equation}\label{eq:solution-regularity}
u\in C_b\big([0,\tau_1),X\big)\cap C^1_b\big([0,\tau_1),\widetilde{X}\big)
\end{equation}
where $u(t):=H_t(z)$, $t\in[0,\tau_1)$.
By shrinking the open neighborhood $O(z_0)$ and taking $\tau_1>0$ smaller if necessary, we obtain 
from \eqref{eq:H-solution}, Lemma \ref{lem:H-Lipschitz}, the fact that $\E, \Xi : U\to\widetilde{X}$ are $C^1$
(and hence, locally Lipschitz continuous), and the compactness of $[0,\tau_1]$, 
that the solution \eqref{eq:solution-regularity} depends locally Lipschitz continuously on the initial data.
The uniqueness of the solution \eqref{eq:solution-regularity} of \eqref{eq:pde} follows easily from the arguments in
the proof of \cite[Lemma 6]{EM}.
This completes the proof of Proposition \ref{prop:product_formula} (with $\tau$ replaced by $\tau_1$).
\end{proof}

\begin{Rem}\label{rem:H-differentiable_in_Z}
It follows from Corollary \ref{coro:H-right_continuous_in_Z} that for any given $z\in O(z)$ we have that the curve
$[0,\tau)\to Z,\quad t\mapsto H_t(z)$, is weakly measurable.
If $Z$ is assumed {\em separable}, we obtain from Pettis theorem that the curve is strongly measurable.
By taking $\rho>0$ smaller if necessary, we then see from \eqref{eq:H-bounded_in_Z} and
the continuity of the first map in \eqref{eq:Xi-regularity} that $\big\|\E(H_t(z))+\Xi(H_t(z))\big\|_{\widetilde Z}$ is 
uniformly bounded for $t\in[0,\tau)$. This together with \eqref{eq:H-solution} then implies that
\[
H_t(z)=z+\int_0^t\big(\E(H_s(z))+\Xi(H_s(z))\big)\,ds\in\widetilde{Z}
\]
where the integral converges in $\widetilde{Z}$ in the sense of Bochner (see, e.g., \cite[Ch. V, \S 5]{Kato}).
Hence, the curve $[0,\tau)\to\widetilde{Z}$, $t\mapsto H_t(z)$, is differentiable almost everywhere.
\end{Rem}

\begin{proof}[Proof of Proposition \ref{prop:product_formula_nu}]
The proof follows directly from Remark \ref{rem:nu3}. In fact, the uniformity in $0\le\nu\le 1$ of the 
estimate \eqref{eq:fundamental(T)} with $S_t$ replaced by $S^{(\nu)}_t$ implies that
\begin{equation}\label{eq:H-limit_nu}
(S^{(\nu)}_{t/n}\circ G_{t/n})^n(z)\stackrel{X}{\to}H^{(\nu)}_t(z)\in X
\end{equation}
uniformly in $z\in W(z_0)$, $t\in[0,\tau]$, and $0\le\nu\le 1$. This implies that the map
\begin{equation}\label{eq:H-map}
H : W(z_0)\times[0,\tau]\times[0,1]\to X,\quad (z,t,\nu)\mapsto H^{(\nu)}_t(z),
\end{equation}
is continuous. By taking $\widetilde{W}(z_0)=B^Z_\rho(z_0)$ and then arguing as in the proof of 
Proposition \ref{prop:product_formula} one concludes that $H^{(\nu)}_t(z)\in\overline{B}^Z_\rho(z_0)\subseteq Z$
for any $z\in W(z_0)$, $t\in[0,\tau]$, and $0\le\nu\le 1$. By combining this with the continuity of the map \eqref{eq:H-map}
we conclude that the map
\[
W(z_0)\times[0,1]\to C_b\big([0,\tau),X\big),\quad (z,\nu)\mapsto u^{(\nu)},
\]
where $u^{(\nu)}(z,t):=H^{(\nu)}_t(z)$, is continuous. 
The fact that the curve $u^{(\nu)}: [0,\tau)\to X$ is differentiable, satisfies \eqref{eq:pde_nu},
and belongs to $C_b\big([0,\tau),X\big)\cap C_b^1\big([0,\tau),\widetilde{X}\big)$ follows in the same way as 
in the proof of Proposition \ref{prop:product_formula}.
\end{proof}

The following Proposition is needed for the proof of Theorem \ref{th:product_formula}.

\begin{Prop}\label{prop:equivalent_distance}
Let $N_1\subseteq N\subseteq X$ be open sets in a Banach space $(X,\|\cdot\|)$ that satisfies the property
\eqref{eq:differentiable_norm}. Assume that for some $T>0$ the maps $G : [0,T)\times N_1\to X$ and 
$S : [0,\infty)\times N\to X$ are continuous and satisfy the local semigroup property. 
Assume in addition that:
\begin{itemize}
\item[$(i)$] For any given $t\in(0,\infty)$ the map $S_t : N\to X$ is $C^2$,
takes values in $N$, and there exist constants $M,\beta>0$ such that
\begin{equation}\label{eq:dS_t-exponential_growth}
\|\dd_xS_t\|_{\LL(X,X)}\le M e^{\beta t}\quad\text{\rm and}\quad
\|\dd^2_xS_t\|_{\LL(X\times X,X)}\le M e^{\beta t}
\end{equation}
for any $x\in N$ and $t\in[0,\infty)$.
\item[$(ii)$] The map $G : [0,T)\times N_1\to X$
is $C^2$ and there exists a constant $L>0$ such that for any given $x\in N_1$,
\begin{equation}\label{eq:E-estimates}
\|\E(x)\|\le L\quad\text{\rm and}\quad\|\dd_x\E\|_{\LL(X,X)}\le L
\end{equation}
where $\E(x)=\frac{\dd}{\dd t}\big|_{t=0}G_t(x)$.
\end{itemize}
Then, $S$ and $G$ are locally equivalent to a quasi-contraction in $N_1$ with the same distance function. 
More specifically, for any $x_0\in N_1$ there exist an open neighborhood $U(x_0)\subseteq N_1$ of $x_0$ in $X$,
a distance function $d : U(x_0)\times U(x_0)\to\R_{\ge 0}$ equivalent to the standard distance on $U(x_0)$, 
and a constant $\beta_1>0$, such that if the open set $V\subseteq U(x_0)$ in $X$ and $0<\tau\le T$ are chosen so
that $S_t(V), G_t(V)\subseteq U(x_0)$ for any $t\in[0,\tau)$ then 
\begin{equation}\label{eq:quasi-contraction1}
d\big(S_t(x),S_t(y)\big)=e^{\beta t}d(x,y)\quad\text{\rm and}\quad
d\big(G_t(x),G_t(y)\big)=e^{\beta_1 t}d(x,y)
\end{equation}
for any $x,y\in V$ and $t\in[0,\tau)$.
\end{Prop}

\begin{Rem}
The condition that $G : [0,T)\times N_1\to X$ is $C^2$ in item $(ii)$ can be replaced by
the conditions that $G : [0,T)\times N_1\to X$ and $\E : N_1\to X$ are $C^1$.
\end{Rem}

%{\color{blue}
%\begin{Rem}\label{rem:Marsden-paper}
%Proposition \ref{prop:equivalent_distance} is a variant of \cite[Theorem 4.1]{Marsden}.
%Not that there is a gap in the proof of Theorem 4.1 in \cite{Marsden} since the function 
%$G_t(x):=F_{\int_0^tf(F_s(x))\,ds}(x)$ defined on p.\,64 does not satisfy the semigroup
%property. The semigrup property is crucial for the proof of the quasi-contraction property
%for the distance function $d$ on p.\,65. 
%\end{Rem}
%}

\begin{proof}[Proof of Proposition \ref{prop:equivalent_distance}]
Take $x_0\in N_1$ and let $U(x_0):=B_\epsilon^X(x_0)$ be an open ball in $X$ of radius $\epsilon>0$ centered at $x_0$
such that $B_{2\epsilon}^X(x_0)\subseteq N_1$. Let us choose a cut-off function $\chi\in C_c^\infty(\R)$
such that $\chi(\rho)=1$ for $|\rho|\le\epsilon^\lambda$, $\chi(\rho)=0$ for $|\rho|\ge(2\epsilon)^\lambda$, 
$0\le\chi(\rho)\le 1$ for $\rho\in\R$, and set 
\[
a(x):=\chi\big(\|x-x_0\|^\lambda\big),\quad x\in X
\] 
where $\lambda>1$ is chosen as in \eqref{eq:differentiable_norm}.
By construction, $a : X\to\R$ is $C^1$, $a= 1$ on $U(x_0)$,
$a= 0$ on $X\setminus B_{2\epsilon}^X(x_0)$, and $0\le a\le 1$.
Let us now consider the vector field 
\begin{equation}\label{eq:E-modified}
\tE : X\to X,\quad \widetilde{\E}(x):=a(x)\E(x).
\end{equation}
Then \eqref{eq:E-modified} is $C^1$ and for any $x\in X$,
\begin{equation}\label{eq:G-tilde_property1}
\|\tE(x)\|\le L'\quad\text{\rm and}\quad\|\dd_x\tE\|_{\LL(X,X)}\le L'
\end{equation}
where $L':=\big(1+\max\limits_{x\in X}\|\dd_xa\|_{\LL(X,\R)}\big)\,L$.
This implies that \eqref{eq:E-modified} is globally Lipschitz on $X$, and hence,
its flow $\tG : [0,\infty)\times X\to X$ is a globally defined $C^1$ map.
By construction, 
\begin{equation}\label{eq:G-tilde_property2}
\tG : [0,\infty)\times N\to N\quad\text{\rm and}\quad
\tG_t\big|_V= G_t\big|_V\quad\text{\rm for any}\quad t\in[0,\tau)
\end{equation}
if the open set $V\subseteq U(x_0)$ and $0<\tau\le T$ are chosen so that
$G_t(V)\subseteq U(x_0)$ for $t\in[0,\tau)$.
Following the proof of \cite[Theorem 4.1]{Marsden}, define for any given $x\in N$ and $\xi\in T_xX= X$
the norm on the tangent space $T_xX$,
\begin{equation}\label{eq:Finsler_norm}
\3\xi\3_x:=\sup_{t\ge 0}\big\|e^{-\beta t}(\dd_xS_t)(\xi)\big\|\,.
\end{equation}
The proof that \eqref{eq:Finsler_norm} defines a norm on $T_xX$ is straightforward.
It follows from \eqref{eq:dS_t-exponential_growth} and \eqref{eq:Finsler_norm} that
\begin{equation}\label{eq:equivalent_norms}
\|\xi\|\le\3\xi\3_x\le M\|\xi\|
\end{equation}
for any $x\in N$ and $\xi\in T_xX= X$.
Moreover, for given $x,y\in N$ and $\xi\in X$ we obtain
\begin{align}
\big|\3\xi\3_x-\3\xi\3_y\big|&=\Big|
\sup_{t\ge 0}\big\|e^{-\beta t}(\dd_xS_t)(\xi)\big\|-\sup_{t\ge 0}\big\|e^{-\beta t}(\dd_yS_t)(\xi)\big\|
\Big|\nonumber\\
&\le\sup_{t\ge 0}\big\|e^{-\beta t}(\dd_xS_t)(\xi)-e^{-\beta t}(\dd_yS_t)(\xi)\big\|\nonumber\\
&\le M\|x-y\|\|\xi\|\label{eq:Finsler_norm-continuous}
\end{align}
where we used the triangle inequality, the fact that $\sup_{t\ge 0}\|f(t)\|$ is a norm on 
$C_b\big([0,\infty),X\big)$, and the second estimate in \eqref{eq:dS_t-exponential_growth}.

\begin{Rem}
The estimates \eqref{eq:equivalent_norms} and \eqref{eq:Finsler_norm-continuous} imply
that $\3\xi\3_x$ is continuous when considered as a function of $(x,\xi)\in TN= N\times X$.
\end{Rem}

\noindent In this way, by setting $F(x,\xi):=\3\xi\3_x^2$ for $(x,\xi)\in N\times X$
we obtain a (continuous) Finsler metric on $N$. With the help of this metric we then define
the corresponding distance function on $N$ in a standard way: for any $x,y\in N$,
\begin{equation}\label{eq:Finsler_distance}
d(x,y)= d_N(x,y):=\inf\limits_\gamma\int_0^1\3\dt\gamma(s)\3_{\gamma(s)}\,ds
\end{equation}
where the infimum is taken over all $C^1$-curves $\gamma : [0,1]\to N$ such that $\gamma(0)=x$
and $\gamma(1)=y$ and where $\dt\gamma$ denotes the derivative of $\gamma$.
In view of \eqref{eq:equivalent_norms}, we obtain that
\begin{equation}\label{eq:equivalent_distances1}
\|x-y\|\le d(x,y)\le M\|x-y\|
\end{equation}
where the left side inequality holds for any $x,y\in N$, 
the right side inequality holds for any $x,y\in N$ such that the straight segment $[x,y]$ lies in $N$, and where we used
the standard fact that $\|x-y\|=\inf\limits_\gamma\int_0^1\|\dt\gamma(s)\|\,ds$ 
provided that $[x,y]\subseteq N$.
Hence, the distance function on $N$ defined in \eqref{eq:Finsler_distance} is equivalent to the standard distance
when restricted to convex subsets of $N$. 
Let us now prove the first estimate in \eqref{eq:quasi-contraction1}.
For any $x,y\in N$, a $C^1$ curve $\gamma : [0,1]\to N$, $\gamma(0)=x$, $\gamma(1)=y$,
and for any given $t\in[0,\infty)$ we have
\begin{align}
&d\big(S_t(x),S_t(y)\big)\le\int_0^1\big\3\big(S_t(\gamma(s))\big)^\bfdot\big\3_{S_t(\gamma(s))}\,ds
\le\int_0^1\sup_{\nu\ge 0}\big\|e^{-\beta\nu}(\dd_{S_t(\gamma(s))}S_\nu)\big(S_t(\gamma(s))\big)^\bfdot\big\|\,ds
\nonumber\\
&\le\int_0^1\sup_{\nu\ge 0}\big\|e^{-\beta\nu}(\dd_{\gamma(s)}S_{\nu+t})\big(\dt{\gamma}(s)\big)\big\|\,ds
\le e^{\beta t}\int_0^1
\sup_{\nu+t\ge 0}\big\|e^{-\beta(\nu+t)}(\dd_{\gamma(s)}S_{\nu+t})\big(\dt{\gamma}(s)\big)\big\|\,ds
\nonumber\\
&\le e^{\beta t}\int_0^1\3\dt{\gamma}(s)\3_{\gamma(s)}\,ds\label{eq:eqlt5}
\end{align}
where $(\cdot)^\bfdot$ denotes the derivative with respect to the variable $s$ and where we used the semigroup property
and the fact that, by $(i)$, $S_t(N)\subseteq N$, and hence the curve $s\mapsto S_t(\gamma(s))$ takes values in $N$.
By taking the infimum in the inequality \eqref{eq:eqlt5} over all $\gamma$'s described above we obtain that
\[
d\big(S_t(x),S_t(y)\big)\le e^{\beta t}d(x,y)
\]
for any $x,y\in N$ and $t\in[0,\infty)$. This proves the first estimate in \eqref{eq:quasi-contraction1}.
Let us now prove the second inequality in \eqref{eq:quasi-contraction1}. We follow the arguments on p. 69 in \cite{Marsden}.
For any $t\ge 0$, $s\in[0,t]$, $x\in N$, and $\xi\in X$ we have
\begin{equation}\label{eq:eqlt6}
(\dd_x\tG_t)(\xi)-(\dd_x\tG_s)(\xi)=\int_s^t\frac{\dd}{\dd\mu}(\dd_x\tG_\mu)(\xi)\,d\mu=
\int_s^t(\dd_{\tG_\mu(x)}\tE)(\dd_x\tG_\mu)(\xi)\,d\mu
\end{equation}
where we used that \eqref{eq:E-modified} is $C^1$ and hence, 
by the ODE theorem in Banach spaces (cf., e.g., \cite[Theorem 1.11, Ch. IV]{Lang}),
the map $\mu\mapsto \dd_x\tG_\mu$, $[0,\infty)\to\LL(X,X)$, is $C^1$ and satisfies the equation
\[
\frac{\dd}{\dd\mu}(\dd_x\tG_\mu)(\xi)=(\dd_{\tG_\mu(x)}\tE)(\dd_x\tG_\mu)(\xi).
\] 
Since the integrand in \eqref{eq:eqlt6} is a continuous function of $\mu\in[s,t]$, we obtain from
\eqref{eq:G-tilde_property1} and \eqref{eq:equivalent_norms} that
\begin{align}\label{eq:eqlt7}
\Big|\big\3(\dd_x\tG_t)(\xi)\big\3_{\tG_s(x)}-\big\3(\dd_x\tG_s)(\xi)\big\3_{\tG_s(x)}\Big|
&\le\int_s^t\big\3(\dd_{\tG_\mu(x)}\tE)(\dd_x\tG_\mu)(\xi)\big\3_{\tG_s(x)}\,d\mu\nonumber\\
&\le ML'\int_s^t\big\|(\dd_x\tG_\mu)(\xi)\big\|\,d\mu
\end{align}
for any $x\in N$, $\xi\in X$, and $0\le s\le t<\infty$.
On the other side, we obtain from \eqref{eq:Finsler_norm-continuous} that
\begin{align}\label{eq:eqlt8}
\Big|\big\3(\dd_x\tG_t)(\xi)\big\3_{\tG_s(x)}-\big\3(\dd_x\tG_t)(\xi)\big\3_{\tG_t(x)}\Big|&\le
M\big\|\tG_s(x)-\tG_t(x)\big\|\,\big\|(\dd_x\tG_t)(\xi)\big\|\nonumber\\
&\le ML'|t-s|\,\big\|(\dd_x\tG_t)(\xi)\big\|
\end{align} 
where we used that $\big\|\tG_s(x)-\tG_t(x)\big\|=\Big\|\int_s^t\tE\big(\tG_\mu(x)\big)\,d\mu\Big\|\le L'|t-s|$.
It now follows from \eqref{eq:eqlt7} and \eqref{eq:eqlt8} that
\begin{equation}\label{eq:eqlt9}
\Big|\big\3(\dd_x\tG_t)(\xi)\big\3_{\tG_t(x)}-\big\3(\dd_x\tG_s)(\xi)\big\3_{\tG_s(x)}\Big|\le
ML'\Big(\int_s^t\big\|(\dd_x\tG_\mu)(\xi)\big\|\,d\mu+|t-s|\,\big\|(\dd_x\tG_t)(\xi)\big\|\Big)
\end{equation}
for any $x\in N$, $\xi\in X$, and $0\le s\le t<\infty$. 
Take $t\in(0,\infty)$ and let $0=t_1\le t_2\le...\le t_\ell=t$ be a partition of the interval $[0,t]$.
We then obtain from \eqref{eq:eqlt9} that
\begin{align*}
\Big|\big\3(\dd_x\tG_t)(\xi)\big\3_{\tG_t(x)}(\xi)-\3\xi\3_x\Big|&\le
\sum_{j=1}^{\ell-1}\Big|\big\3(\dd_x\tG_{t_j})(\xi)\big\3_{\tG_{t_j}(x)}-
\big\3(\dd_x\tG_{t_{j+1}})(\xi)\big\3_{\tG_{t_{j+1}}(x)}\Big|\\
&\le ML'\Big(\int_0^t\big\|(\dd_x\tG_\mu)(\xi)\big\|\,d\mu+
\sum_{j=1}^{\ell-1}\big\|(\dd_x\tG_{t_{j+1}})(\xi)\big\|\,|t_{j+1}-t_j|
\Big)\,.
\end{align*}
By passing to the limit as the diameter of the partition $\max\limits_{0\le j\le\ell-1}|t_{j-1}-t_j|\to 0$
we then obtain that
\[
\Big|\big\3(\dd_x\tG_t)(\xi)\big\3_{\tG_t(x)}(\xi)-\3\xi\3_x\Big|\le 2ML'\int_0^t\big\|(\dd_x\tG_\mu)(\xi)\big\|\,d\mu
\le 2ML'\int_0^t\big\3(\dd_x\tG_\mu)(\xi)\big\3_{\tG_\mu(x)}\,d\mu
\]
where we used \eqref{eq:equivalent_norms} and the fact that the Riemann sums converge since
the integrand is a continuous function $[0,t]\to X$. This implies that
\[
\big\3(\dd_x\tG_t)(\xi)\big\3_{\tG_t(x)}\le
\3\xi\3_x+2ML'\int_0^t\big\3(\dd_x\tG_\mu)(\xi)\big\3_{\tG_\mu(x)}\,d\mu
\]
for any $x\in N$, $\xi\in X$, and $t\in[0,\infty)$.
By Gronwall's inequality,
\begin{equation}\label{eq:eqlt10}
\big\3(\dd_x\tG_t)(\xi)\big\3_{\tG_t(x)}\le\3\xi\3_x e^{2ML' t}
\end{equation}
for any $x\in N$, $\xi\in X$, and $t\in[0,\infty)$. 
We then obtain from \eqref{eq:Finsler_distance} and \eqref{eq:eqlt10} that
for any $t\in[0,\infty)$, $x,y\in N$, and for any $C^1$-curve $\gamma : [0,1]\to N$, 
$\gamma(0)=x$, $\gamma(1)=y$, we have
\[
d\big(\tG_t(x),\tG_t(y)\big)\le
\int_0^1\big\3\big(\tG_t\big(\gamma(s)\big)\big)^\bfdot\big\3_{\tG_t(\gamma(s))}\,ds
\le e^{2ML' t}\int_0^1\3\dt{\gamma}(s)\3_{\gamma(s)}\,dt
\]
where we also used that $\tG_t(N)\subseteq N$ to conclude the inequality appearing on the left side.
By taking the infimum over all $\gamma$'s as above we conclude that
\[
d\big(\tG_t(x),\tG_t(y)\big)\le e^{2ML' t}d(x,y)
\]
for any $x,y\in N$ and $t\in[0,\infty)$. 
This, together with the second formula in \eqref{eq:G-tilde_property2}, then proves the second inequality 
in \eqref{eq:quasi-contraction1}.
It follows from \eqref{eq:equivalent_distances1} and the convexity of $U(x_0)$ that
the restriction $d|_{U(x_0)}$ of the distance function \eqref{eq:Finsler_distance} to the subset $U(x_0)$ 
is equivalent to the standard distance on $U(x_0)$. This completes the proof of the Proposition.
\end{proof}

We are now ready to prove Theorem \ref{th:product_formula}.

\begin{proof}[Proof of Theorem \ref{th:product_formula}]
Assume that the condition $1)$ and $2)$ in Proposition \ref{prop:product_formula} $(c)$ are replaced by 
the conditions $1^*)$ and $2^*)$ and that all other conditions in Proposition \ref{prop:product_formula} are satisfied.
Take $x_0\in U$. By $1^*)$, there exists an open neighborhood $N(x_0)\subseteq X$ of $x_0$ in $X$ such that
for any given $t\in[0,\infty)$ the map $S_t : N(x_0)\to X$ is well-defined and $C^2$, $S_t\big(N(x_0)\big)\subseteq N(x_0)$,
and the inequalities \eqref{eq:dS_t-exponential_growth1} hold as stated. Since $G : [0,T)\times U\to X$ is assumed $C^2$
(condition $(a)$) there exists an open neighborhood $N_1(x_0)\subseteq N(x_0)\cap U$ of $x_0$ in $X$ such that 
the inequalities \eqref{eq:E-estimates} hold for any $x\in N_1(x_0)$ and $t\in[0,T)$ (with $T>0$ taken smaller if necessary). 
We then apply Proposition \ref{prop:equivalent_distance} (with $N:= N(x_0)$ and $N_1:= N_1(x_0)$) to obtain
that $G : [0,\infty)\times N_1(x_0)\to X$ and $S : [0,\infty)\times N_1(x_0)\to X$ are locally equivalent to a quasi-contraction
with the same distance function. Hence the condition $1)$ in $(c)$ is satisfied with $U$ replaced by $N_1(x_0)$.
The condition $2)$ in $(c)$ follows from $2^*)$ by the same reasoning.
We can now apply  Proposition \ref{prop:product_formula} [Proposition \ref{prop:product_formula_nu}, respectively]
with $U:=N_1(x_0)$ to obtain for any $z_0\in U_Z$ a positive constant $0<\tau\le T$, an open neighborhood 
$O(z_0)\subseteq U_Z$ of $z_0$ in $Z$, and a flow map $H : [0,\tau)\times O(Z_0)\to Z$ as stated in
Proposition \ref{prop:product_formula} [Proposition \ref{prop:product_formula_nu}, respectively]. 
This completes the proof of Theorem \ref{th:product_formula}.
\end{proof}

\appendix

\section{The space $W^{m,p}_\delta$}\label{sec:W-spaces}
In this Appendix we collect several basic properties of the weighted Sobolev space $W^{m,p}_\delta$
as well as other technical results used in the main body of the paper. 

\begin{Lem}\label{lem:W-decay}
Assume that $1<p<\infty$ and $\delta\in\R$.
Then, for $m>\frac{d}{p}$ and for any $0\le k<m-\frac{d}{p}$ we have that 
$W^{m,p}_\delta\subseteq C^k(\R^d)$. Moreover, there exists a constant $C>0$ such that
for any $|\alpha|\le k$ and for any $f\in W^{m,p}_\delta$ we have that
\[
\x^{\delta+\frac{d}{p}+|\alpha|}\big|(\partial^\alpha f)(x)\big|\le C\,\|f\|_{W^{m,p}_\delta}
\]
for any $x\in\R^d$. In addition, $\x^{\delta+\frac{d}{p}+|\alpha|}|(\partial^\alpha f)(x)|=o(1)$
as $|x|\to\infty$.
\end{Lem}

The elements of $W^{m,p}_\delta$ enjoy the following composition property.

\begin{Lem}\label{lem:W-multiplication}
Assume that $1<p<\infty$, $\delta_1,\delta_2\in\R$, $0\le k\le l\le m$ and $m+l-k>\frac{d}{p}$. 
Then, there exists a constant $C>0$ such that for any $f\in W^{m,p}_{\delta_1}$ and $g\in W^{l,p}_{\delta_1}$,
we have that $fg\in W^{k,p}_{\delta_1+\delta_2+\frac{d}{p}}$ and
\[
\|fg\|_{W^{k,p}_{\delta_1+\delta_2+\frac{d}{p}}}\le C\,
\|f\|_{ W^{m,p}_{\delta_1}}\|g\|_{ W^{l,p}_{\delta_2}}\,.
\]
\end{Lem}

For the proof of Lemma \ref{lem:W-decay} and Lemma \ref{lem:W-multiplication} we refer to 
\cite[Appendix A]{SultanTopalov} (see also \cite{McOwen,McOwenTopalov2}).
Note that as a consequence from Lemma \ref{lem:W-multiplication} we obtain that the space $W^{m,p}_\delta$
is a Banach algebra for $m>\frac{d}{p}$ and $\delta+\frac{d}{p}\ge 0$.
Recall from \cite{McOwenTopalov4} that we have the following Theorem on
the compositions of maps.

\begin{Th}\label{th:composition}
Assume that $m>1+\frac{d}{p}$, $\delta+\frac{d}{p}>-1$, and $r\ge\Z_{\ge 0}$.
Then we have:
\begin{itemize}
\item[(a)] The composition $W^{m+r,p}_\delta\times\D^{m,p}_\delta$,
$(w,\varphi)\mapsto w\circ\varphi$, is a $C^r$ map.
\item[(b)] The map $\D^{m+r,p}_\delta\to\D^{m,p}_\delta$, $\varphi\mapsto\varphi^{-1}$,
is a $C^r$ map.
\end{itemize}
\end{Th}

In \cite{McOwenTopalov4} the Theorem is stated for $r=0,1$. The case $r\ge 2$ follows easily
from the arguments in the proof of \cite[Theorem 1.1]{IKT}.
The following Lemma generalizes \cite[Corollary A.2]{McOwenTopalov4}.

\begin{Lem}\label{lem:continuity_composition_general}
Assume that $1<p<\infty$, $m>1+\frac{d}{p}$, $\delta+\frac{d}{p}>-1$, and $\gamma\in\R$.
Then, for any $0\le k\le m$ the map
\begin{equation}\label{em:continuity_composition_general}
W^{k,p}_\gamma\times \D^{m,p}_\delta\to W^{k,p}_\gamma,\quad
(f,\varphi)\mapsto f\circ\varphi,
\end{equation}
is well defined and continuous.
\end{Lem}

\begin{proof}[Proof of Lemma \ref{lem:continuity_composition_general}]
We follow the arguments in the proof of \cite[Lemma 2.7]{IKT}.
Assume that $1<p<\infty$, $\delta+\frac{d}{p}>-1$, $m>1+\frac{d}{p}$, and $\gamma\in\R$.
Then, for any $0\le k\le m$ the composition map \eqref{em:continuity_composition_general} is well defined by
\cite[Lemma A.4]{McOwenTopalov4}. Let us prove that it is continuous.
Take $(\varphi_0,f_0)\in\D^{m,p}_\delta\times W^{k,p}_\delta$ and $\epsilon>0$.
For any $(\varphi,f)\in\D^{m,p}_\delta\times W^{k,p}_\gamma$ we have
\begin{equation}\label{eq:continuity1}
\big\|f\circ\varphi-f_0\circ\varphi_0\big\|_{W^{k,p}_\gamma}\le
\big\|f\circ\varphi-f\circ\varphi_0\big\|_{W^{k,p}_\gamma}+
\big\|f\circ\varphi_0-f_0\circ\varphi_0\big\|_{W^{k,p}_\gamma}.
\end{equation}
Similarly, for any $\tilde{f}\in C^\infty_c$ we have
\begin{equation}\label{eq:continuity2}
\big\|f\circ\varphi-f\circ\varphi_0\big\|_{W^{k,p}_\gamma}\le
\big\|f\circ\varphi-\tilde{f}\circ\varphi\big\|_{W^{k,p}_\gamma}+
\big\|\tilde{f}\circ\varphi-\tilde{f}\circ\varphi_0\big\|_{W^{k,p}_\gamma}+
\big\|\tilde{f}\circ\varphi_0-f\circ\varphi_0\big\|_{W^{k,p}_\gamma}.
\end{equation}
By \cite[Lemma A.4]{McOwenTopalov4} there exists an open neighborhood $U(\varphi_0)$ of $\varphi_0$ in
$\D^{m,p}_\delta$ and a constant $C>0$ such that $\|w\circ\varphi\|_{W^{k,p}_\gamma}\le C\,\|w\|_{W^{k,p}_\gamma}$
for any $\varphi\in U(\varphi_0)$ and $w\in W^{k,p}_\gamma$.
This implies that for $\varphi\in U(\varphi_0)$ we can choose $f\in W^{m,p}_\gamma$ and
$\tilde{f}\in C^\infty_c$ such that 
\begin{equation}\label{eq:continuity3}
\big\|f\circ\varphi_0-f_0\circ\varphi_0\big\|_{W^{k,p}_\gamma}\le\epsilon/4,\quad
\big\|f\circ\varphi-\tilde{f}\circ\varphi\big\|_{W^{k,p}_\gamma}\le\epsilon/4,\quad
\big\|\tilde{f}\circ\varphi_0-f\circ\varphi_0\big\|_{W^{k,p}_\gamma}\le\epsilon/4\,.
\end{equation}
In view of \cite[Lemma A.6]{McOwenTopalov4}, there exists a constant $C>0$ such that
\begin{equation*}
\big\|\tilde{f}\circ\varphi-\tilde{f}\circ\varphi_0\big\|_{W^{m,p}_\gamma}\le
C\,\|\tilde{f}\|_{W^{m+1,p}_{\gamma-1}}\|\varphi-\varphi_0\|_{W^{m,p}_\gamma}
\end{equation*}
for any $\varphi\in U(\varphi_0)$. This implies that we can choose $\varphi\in U(\varphi_0)$ such that
\begin{equation}\label{eq:continuity4}
\big\|\tilde{f}\circ\varphi-\tilde{f}\circ\varphi_0\big\|_{W^{m,p}_\gamma}\le\epsilon/4.
\end{equation}
The claimed continuity then follows from \eqref{eq:continuity1}, \eqref{eq:continuity2}, 
\eqref{eq:continuity3}, and \eqref{eq:continuity4}.
\end{proof}

Assume that $1<p<\infty$, $\delta+\frac{d}{p}>-1$, and $\gamma\in\R$. 
Lemma \ref{lem:continuity_composition_general} allows us to define
for any $m>1+\frac{d}{p}$ and $0\le k\le m-1$ the map
\begin{equation}\label{eq:gradient}
\D^{m,p}_\delta\times W^{k+1,p}_\gamma\to W^{k,p}_{\gamma+1},\quad
(\varphi,f)\mapsto R_\varphi\circ\nabla\circ R_{\varphi^{-1}},
\end{equation}
as well as the map
\begin{equation}\label{eq:Laplacian}
\D^{m,p}_\delta\times W^{k+2,p}_\gamma\to W^{k,p}_{\gamma+2},\quad
(\varphi,f)\mapsto R_\varphi\circ\Delta\circ R_{\varphi^{-1}},
\end{equation}
for any $m>2+\frac{d}{p}$ and $0\le k\le m-2$.
We have

\begin{Lem}\label{lem:conjugate_operators}
The maps \eqref{eq:gradient} and \eqref{eq:Laplacian} are analytic.
\end{Lem}

\noindent The proof of this Lemma is identical with the proof of \cite[Lemma 2.1]{McOwenTopalov4}.

\medskip\medskip

Let us now discuss the {\em complex interpolation} in the scale of $W$-spaces.
Let $X_0$ and $X_1$ be a pair of Banach spaces over $\C$ such that $X_1\subseteq X_0$ is bounded.
For $0\le\theta\le 1$ consider the complex interpolation space $X_\theta=[X_0,X_1]_\theta$ (see, e.g., \cite[IX, \S 4]{RS},
\cite{KSST}, \cite[\S 3.1]{Triebel}). Let $1<p<\infty$ and $\delta\in\R$.
It follows from \cite[Theorem 2, Definition 2\,(d)]{Triebel} that for $m\in\Z_{\ge 0}$ we have that
\begin{equation}\label{eq:W=h}
W^{m,p}_{\delta,\C}= h^m_{p,\mu}\quad\text{with}\quad\mu:=(\delta-m)p
\end{equation}
where the Banach space $h^s_{p,\mu}$ is defined for all $s\ge 0$ real (see \cite[Definition 2\,(a)]{Triebel}).
This allows us to include the scale of $W$-spaces into a continuous scale of Banach spaces with real regularity exponents.
As a consequence, we have the following interpolation result.

\begin{Lem}\label{lem:interpolation}
Assume that $1<p<\infty$, $m, m_0, m_1\in\Z_{\ge 0}$, and $\delta, \delta_0, \delta_1\in\R$. Then,
we have:
\begin{itemize}
\item[(i)] For $m:=(1-\theta)m_0+\theta m_1$ with $0\le\theta\le 1$ we have that 
$W^{m,p}_{\delta,\C}=\big[W^{m_0,p}_{\delta,\C},W^{m_1,p}_{\delta,\C}\big]_\theta$.
\item[(ii)] For $\delta:=(1-\theta)\delta_0+\theta\delta_1$ with $0\le\theta\le 1$ we have that 
$W^{m,p}_{\delta,\C}=\big[W^{m,p}_{\delta_0,\C},W^{m,p}_{\delta_1,\C}\big]_\theta$.
\end{itemize}
\end{Lem}

\begin{proof}[Proof of Lemma \ref{lem:interpolation}]
The Lemma follows directly from \eqref{eq:W=h} and \cite[Theorem 3\,(e)]{Triebel}.
Let us first prove item (i). Take $m_0, m_1\in\Z_{\ge 0}$, $m_0\le m_1$ and assume that
$m=(1-\theta)m_0+\theta m_1\in\Z_{\ge 0}$ for some $0\le\theta\le 1$.
Then, we obtain from \eqref{eq:W=h} that
\[
\mu=(\delta-m)p=(1-\theta)(\delta-m_0)p+\theta(\delta-m_1)p
=(1-\theta)\mu_0+\theta\mu_1
\]
where we set $\mu_0:=(\delta-m_0)p$ and $\mu_1:=(\delta-m_1)p$.
By \cite[Theorem 3\,(e)]{Triebel} we have that
\[
h^m_{p,\mu}=\big[h^{m_0}_{p,\mu_0},h^{m_1}_{p,\mu_1}\big]_\theta.
\]
Since by \eqref{eq:W=h},
\[
W^{m,p}_{\delta,\C}= h^m_{p,\mu},\quad 
W^{m_0,p}_{\delta,\C}= h^{m_0}_{p,\mu_0},\quad\text{\rm and}\quad
W^{m_1,p}_{\delta,\C}= h^{m_1}_{p,\mu_1},
\]
we then conclude the proof of (i). The proof of (ii) follows by the same arguments.
\end{proof}

The notion of a $C^k$ map along curves is introduced in Section \ref{sec:Lie-Trotter}, 
Definition \ref{def:C^k_along_curves}. The proof of the following Lemma can be done 
in a straightforward way and will be omitted (cf. Remark \ref{rem:C^k_along_curves}).

\begin{Lem}\label{lem:S-smoothness}
Let $\{S(t)\}_{t\ge 0}$ be a strongly continuous semigroup of bounded linear operators 
$S(t)\in\LL(X)$ with a generator $A : D_A\to X$ and a dense domain $D_A\subseteq X$.
Then, the map $S : [0,\infty)\times D_A\to X$, $(t,x)\mapsto S(t)x$, is $C^1$ along curves and
the map $S : [0,\infty)\times D_{A^2}\to X$, $(t,x)\mapsto S(t)x$, where 
$D_{A^2}:=\{f\in D_A\,|\,Af\in D_A\}$, is $C^2$ along curves.\footnote{The Banach spaces $D_A$ and 
$D_{A^2}$ are equipped with the corresponding graph norms.}
The statement above holds with $D_A$ and $D_{A^2}$ replaced by any boundedly embedded
in $D_A$ and $D_{A^2}$ subspaces $\widetilde{D}_A\subseteq D_A$ and 
$\widetilde{D}_{A^2}\subseteq D_{A^2}$.
\end{Lem}

The following Lemma shows that the norm in $W^{m,p}_\delta$, $m\ge 0$, $\delta\in\R$, $1<p<\infty$,
satisfies the property \eqref{eq:differentiable_norm} with $\lambda=p$.

\begin{Lem}\label{lem:L^p-norm}
Assume that $1<p<\infty$. Then the map $F : L^p\to\R$, $f\mapsto\|f\|_{L^p}^p=\int_{\R^d}|f(x)|^p\,dx$, 
is $C^1$. Moreover, we have that $(\dd_f F)(\xi)=p\int_{\R^d}|f(x)|^{p-1}\sign(f(x))\,\xi(x)\,dx$ and
$\|\dd_f F\|_{(L^p)^*}\le p\|f\|_{L^p}^{p/q}$.
\end{Lem}
\noindent For the proof see, e.g., \cite[Proposition 5]{BF}. The bound on the norm of $\dd_f F$ follows from 
H\"older's inequality.

We will also need the following stronger version of \cite[Proposition 1.2]{McOwenTopalov4}
on the independence of the interval of existence of the solutions of the Euler equation on the regularity
of the initial data in the weighted Sobolev spaces.

\begin{Prop}\label{prop:T-independent_on_regularity}
Assume that $m_0>1+\frac{d}{p}$ and $-1/2<\delta_0+\frac{d}{p}<d+1$ and let 
$u\in C\big([0,T),\aW^{m_0,p}_{\delta_0}\big)\cap C^1\big([0,T),\aW^{m_0-1,p}_{\delta_0}\big)$
be a solution of the Euler equation with $u_0\in W^{m_0,p}_{\delta_0}$.
Then, if $u_0\in\aW^{m,p}_\delta$ for $m\ge m_0$ and $\delta\ge\delta_0$ with $\delta+\frac{d}{p}<d+1$
we have that $u\in C\big([0,T),\aW^{m,p}_\delta\big)$.
If $m>3+\frac{d}{p}$ then 
$u\in C\big([0,T),\aW^{m,p}_\delta\big)\cap C^1\big([0,T),\aW^{m-1,p}_\delta\big)$.
\end{Prop}

\begin{Rem}
If $\delta_0+\frac{d}{p}\ge d+1$ then the statement of Proposition \ref{prop:T-independent_on_regularity}
continues to hold but the solution takes the form in \cite[Theorem 1.2]{McOwenTopalov4} (b).
The last statement of  Proposition \ref{prop:T-independent_on_regularity} holds without the assumption
that $m>3+\frac{d}{p}$. (The requirement $m>3+\frac{d}{p}$ in \cite[Theorem 1.2]{McOwenTopalov4} is
technical and can be replaced by $m>1+\frac{d}{p}$ -- see the discussion on p. 1480 in \cite{McOwenTopalov4}.)
\end{Rem}

\begin{proof}[Proof of Proposition \ref{prop:T-independent_on_regularity}]
Assume that the conditions of the Proposition are satisfied. 
Let us first prove that $u\in C\big([0,T),\aW^{m,p}_\delta\big)$.
Assume that $m=m_0+1$ and $\delta\ge\delta_0$ and let 
$\varphi\in C^1\big([0,T),\D^{m_0,p}_{\delta_0}\big)$ be the solution of
the ordinary differential equation $\dt{\varphi}(t)=u(t)\circ\varphi(t)$, $\varphi|_{t=0}=\id$
(see \cite[Proposition 2.1]{McOwenTopalov4}). 
By the preservation of vorticity we have that for any $t\in[0,T)$,
$\omega(t)=\psi(t)^*\omega_0$ where $\psi(t):=\varphi(t)^{-1}$,
$\omega(t):=\dd(u(t)^\flat)$ is the vorticity 2-form, and $u(t)^\flat$ is the 1-form obtained from
the vector field $u(t)$ by lowering the indexes with the help of the Euclidean metric on $\R^d$.
In coordinates,
\begin{equation}\label{eq:conservation_law}
\omega_{kl}(t)=\sum_{\alpha,\beta=1}^d\omega_{0 \alpha\beta}(t)\circ\psi(t)\,
\frac{\partial\psi_\alpha(t)}{\partial x_k}\frac{\partial\psi_\alpha(t)}{\partial x_l},\quad 1\le k,l\le d,
\end{equation}
where $\psi(t)=\id+f(t)$ with $f\in C\big([0,T),W^{m_0,p}_{\delta_0}\big)$, since $\D^{m_0,p}_{\delta_0}$
is a topological group (\cite[Theorem 2.1]{McOwenTopalov4}).
It then follows from \eqref{eq:conservation_law} and \cite[Theorem 2.1]{McOwenTopalov4} that 
$\omega\in C\big([0,T),W^{m_0,p}_{\delta+1}\big)$.
Hence, we obtain from the Boit-Savart identity 
$u_l=\Delta^{-1}\big(\sum_{j=1}^d\frac{\partial\omega_{jl}}{\partial x_j}\big)$, $1\le l\le d$, 
$\delta+\frac{d}{p}\notin\Z$, and \cite[Proposition B.1]{McOwenTopalov4} that 
$u\in C\big([0,T),\aW^{m_0+1,p}_\delta\big)$.
(The fact that there are no asymptotic terms follows from the assumption $-1/2<\delta+\frac{d}{p}<d+1$.
The case $\delta+\frac{d}{p}\in\Z$ follows by approximation as in the proof of Theorem 1.1 in \cite{Top1}.)
By iterating these arguments inductively, we obtain that $u\in C\big([0,T),\aW^{m_0+1,p}_\delta\big)$
for any $m\ge m_0$. If $m>3+\frac{d}{p}$ then we can apply \cite[Theorem 1.2]{McOwenTopalov4}
to conclude that $u\in C\big([0,T),\aW^{m,p}_\delta\big)\cap C^1\big([0,T),\aW^{m-1,p}_\delta\big)$.
\end{proof}

\end{document}